\documentclass[12pt,reqno]{amsart}

\usepackage[margin=2cm]{geometry}

\usepackage{fourier}

\usepackage[colorlinks=true, pdfstartview=FitV, linkcolor=blue,citecolor=red, urlcolor=blue]{hyperref}

\usepackage{hyperref}

\usepackage{amsmath,amsthm,amssymb}
\usepackage{color}

\usepackage{url}
\usepackage{breakurl}
\newcommand{\bburl}[1]{\textcolor{blue}{\url{#1}}}

\theoremstyle{definition}
\newtheorem{thm}{Theorem}[section]
\newtheorem{reg}[thm]{Regularity}
\newtheorem{conv}[thm]{Convention}
\newtheorem{defi}[thm]{Definition}
\newtheorem{cor}[thm]{Corollary}
\newtheorem{lem}[thm]{Lemma}
\newtheorem{prop}[thm]{Proposition}
\newtheorem{rem}[thm]{Remark}

\newcommand{\R}{\ensuremath{\mathbb{R}}}
\newcommand{\C}{\ensuremath{\mathbb{C}}}
\newcommand{\Z}{\ensuremath{\mathbb{Z}}}

\numberwithin{equation}{section}

\DeclareMathOperator{\re}{Re}
\DeclareMathOperator{\im}{Im}

\DeclareMathOperator{\Res}{Res}

\DeclareMathOperator{\balph}{\boldsymbol{\alpha}}

\newcommand{\mbu}{\mathbf{u}}
\newcommand{\mbn}{\mathbf{n}}
\newcommand{\mby}{\mathbf{y}}
\newcommand{\mrU}{\mathrm{U}}
\newcommand{\mrN}{\mathrm{N}}
\newcommand{\mrG}{\mathrm{G}}
\newcommand{\ovmrG}{\overline{\mathrm{G}}}
\newcommand{\mrK}{\mathrm{K}}
\newcommand{\mbF}{\mathbf{F}}

\usepackage{empheq}

\usepackage[
backend=biber,
style=numeric,
sorting=nyt,
maxnames=10,
giveninits=true
]{biblatex}
\usepackage{pst-node}
\usepackage{tikz-cd}
\usepackage{verbatim}
\usepackage{mathdots}
\usepackage{bm}

\usepackage{scalerel,stackengine}
\stackMath

\newcommand\reallywidehat[1]{%
	\savestack{\tmpbox}{\stretchto{%
			\scaleto{%
				\scalerel*[\widthof{\ensuremath{#1}}]{\kern.1pt\mathchar"0362\kern.1pt}%
				{\rule{0ex}{\textheight}}
			}{\textheight}%
		}{2.4ex}}%
	\stackon[-6.9pt]{#1}{\tmpbox}%
}

\title[Spectral Moment Formulae for $\hbox{GL}(3)\times \hbox{GL}(2)$ $L$-functions II]{Spectral Moment Formulae for $\hbox{GL}(3)\times \hbox{GL}(2)$ $\hbox{L}$-functions II: The Eisenstein Case}

\author[C.-H. Kwan]{Chung-Hang Kwan}
\address{Institute for Advanced Study\\ 1 Einstein Drive, Princeton, New Jersey 08540,  USA}
\email{\textcolor{blue}{\href{mailto: chkwan@ias.edu}
{chkwan@ias.edu}}}



\date{\today}

\begin{document}

\renewcommand{\thefootnote}{\fnsymbol{footnote}}
\footnotetext[1]{This research is supported by the EPSRC Grant EP/W009838/1 during a postdoctoral position at University College London, and the Harish-Chandra Fund from the Institute for Advanced Study.}
\renewcommand{\thefootnote}{\arabic{footnote}}

\maketitle

\begin{abstract}
This work is the second in a series, following Part I \cite{Kw23} (\href{https://msp.org/ant/2024/18-10/p02.xhtml}{\emph{Algebra Number Theory} \textbf{18}.10 (2024)}) and preceding Part III \cite{Kw25} (\href{https://link.springer.com/article/10.1007/s00208-024-02914-9}{\emph{Math. Ann.} \textbf{391}.1 (2025)}). We continue our investigation of spectral moments of $\hbox{GL}(3)\times \hbox{GL}(2)$ $L$-functions from the perspective of period integrals. Using an identity between two distinct periods for the $\hbox{GL}(3)$ Eisenstein series,  we establish an exact Motohashi-type identity linking the shifted cubic moment of $\hbox{GL}(2)$ $L$-functions to the shifted fourth moment of $\hbox{GL}(1)$ $L$-functions. In addition, we offer a novel, intrinsic and automorphic account for the sources and symmetries of the full set of main terms for both moments, in agreement with the CFKRS Moment Conjectures (\href{https://londmathsoc.onlinelibrary.wiley.com/doi/abs/10.1112/S0024611504015175}{\emph{Proc. Lond. Math. Soc.} (3) \textbf{91} (2005)}). 
\end{abstract}


\section{Introduction}

\subsection{Background}

Over the past two decades,  the cubic moment of $\hbox{GL}(2)$ automorphic $L$-functions (or simply the \emph{cubic moment}) has driven a wave of remarkable applications in number theory. Notable among these is the landmark result of Conrey--Iwaniec \cite{CI00}, along with the vast recent developments of Petrow--Young \cite{PY19, PY20, PY23}, culminating in the Weyl-strength subconvexity bounds for Dirichlet $L$-functions of any primitive character $\chi\, (\bmod\ q)$: 
\begin{align}
	L(1/2, \, \chi)  \ \ll_{\epsilon} \  q^{1/6+\epsilon},
\end{align}
and analogous bounds for $\mathrm{GL}(2)$ $L$-functions in various aspects. These advances have far-reaching arithmetic consequences; see, for instance, \cite{Y17, DIT16, FM10, LMY13}. 

Structurally, it is striking that the cubic moment of $\hbox{GL}(2)$ $L$-functions is \emph{dual} to the fourth moment of the Riemann zeta function (or the \emph{fourth moment})---a fundamentally different moment that occupies a central place in classical analytic number theory, tracing back to the seminal works of Hardy--Littlewood \cite{HL16} and Ingham \cite{Ing27} in the 1920s. This duality was first discovered by Motohashi \cite{Mo93} in the 1990s, who established a deep spectral identity of the form
\begin{align}\label{basicmoto}
 \int_{-\infty}^{\infty}  \ |\zeta(1/2+it)|^4 w(t)   \ dt 
 \   = \  	\sum_{\substack{f: \text{ level $1$ } \\ \text{cusp forms}}} \  L( 1/2,\, f)^3 \widetilde{w}(t_{f})    \  \ + \  \  (***),
\end{align}
 where $w\mapsto \widetilde{w}$ denotes an intricate integral transform, and $(***)$ encodes a collection of non-trivial arithmetic contributions together with the continuous component. This identity is widely regarded as a ``crowning achievement'' in the theory of the Riemann zeta function (\cite{CK15a, CK15b}), and it represents the first and a particularly intriguing (and challenging) instance of \emph{Spectral Reciprocity}.

To study the cubic moment, one proceeds by  \textit{inverting} the transform $w\mapsto \widetilde{w}$ and the identity (\ref{basicmoto}). This was achieved to some extent (i.e., without revealing the underlying fourth moment) by Ivi\'{c} \cite{Iv01}, and Motohashi sketched a formal derivation in his note  \cite{Mot99}. Proving a spectrally inverted Motohashi-type formula, even in approximate form, requires substantial refinements of the techniques of \cite{CI00} and \cite{Iv01}. Such refinements were obtained only relatively recently by Petrow \cite{Pe15} and Frolenkov \cite{Fr20}, both via relative trace formulae. The relevant analytic machinery has been developed in various ways, e.g.,  \cite{Bl12b, Li11, Y17, JM05, Qi24, KPY19}.


\subsection{Moment conjectures}\label{MoCj}

A central theme in moments of $L$-functions is to understand the \textit{sources} and \textit{structures} underlying the full set of main terms. Over the years, this pursuit has drawn sustained attention and yielded fruitful outcomes. Most notably, two sets of heuristics have led to remarkably precise conjectures for the asymptotics of moments. Conrey--Farmer--Keating--Rubinstein--Snaith  \cite{CFK+05} (hereafter referred to as CFKRS) developed their \textit{``recipes''} based on shifted moments and approximate functional equations, whose use dates back to  \cite{Ing27} and \cite{HL16}.  Their approach reveals surprising connections to \textit{Random Matrix Theory} and brings to light many unexpected intrinsic symmetries and combinatorial structures.  On the other hand, Diaconu--Goldfeld--Hoffstein  \cite{DGH03} introduced the \textit{Multiple Dirichlet Series} for moments of the Riemann zeta function and quadratic Dirichlet  $L$-functions. Their predictions arise from the (conjectural) analytic continuation, polar divisors, and the group of functional equations of these series. 

For the two moments considered in this article, namely those in (\ref{basicmoto}),  extracting the full structure of the main terms, suppressed as part of  $(***)$ in (\ref{basicmoto}), is a delicate task. For the fourth moment of the $\zeta$-function, Heath-Brown  \cite{HB79} and Motohashi \cite{Mo93} proved the asymptotic formula 
\begin{align}\label{HBresul}
	\int_{T}^{2T} \, \left|\zeta(1/2+it)\right|^4 \ dt \ \ = \ \  T\cdot P_{2}(\log T) \, + \,   O(T^{1-\delta}) \hspace{15pt} (T\, \to\, \infty)
\end{align}
with $\delta<1/8$ and  $\delta<1/3$ respectively, where $P_{2}$ is a polynomial of degree $4$. The resulting expressions of $P_{2}$ (see also Conrey \cite{C96}) were, unfortunately, quite complicated, making a complete determination of its coefficients daunting from their formulae.  

In \cite{Mo93}, Motohashi considers the following smoothed and shifted variant of the fourth moment:
\begin{align}\label{shmozet}
	\int_{\R} \  \eta(t) \prod_{i=1}^{2} \, \zeta(1/2 +it + \alpha_{i})   \zeta(1/2 -it+ \beta_{i}) \ dt \hspace{15pt} (\eta \, \in \, C_{c}^{\infty}[T,2T]).
\end{align} 
 The shifts are initially taken with $\re \alpha_{i}, \,\re \beta_{j} \gg 1$, and play a pivotal role in an intricate analytic continuation procedure. CFKRS later discovered that the many degenerate terms in Motohashi's formula for (\ref{shmozet}) can be combined, in rather subtle ways, to yield an elegant, intrinsic description of the main terms.  There are six main terms in the CFKRS formulation (see (\ref{expl4thMT})), from which the polynomial $P_{2}$ can be computed in full and with ease by letting the shifts $\alpha_{i}, \beta_{j} \to 0$. This is explained in \cite[Section 5.1]{CFK+05}. 

The recipe works equally well for the cubic moment of $\hbox{GL}(2)$ $L$-functions. Owing to differences in the families and their arithmetic, this case yields eight distinct main terms. In \cite{CI00}, the authors located only the diagonal term, though they noted the interest in identifying the others (\cite[p. 1177]{CI00}). We review these conjectures in Section \ref{CFKRSSOU},  and examine the difficulties that hindered the identification of these terms in earlier approaches in Section \ref{mainres}.


\subsection{Beyond the Petersson/Kuznetsov paradigm}\label{strucBFWNe}

Motohashi's original work \cite{Mo93} is technical, and the source of \textit{``reciprocity''} is far from apparent. A priori, there is no reason to expect the off-diagonal (or error) terms to be neatly expressible by moments of $L$-functions. Having a conceptual understanding of this phenomenon, without intermediate devices like Kloosterman sums, summation formulae, shifted convolutions, or the Petersson--Kuznetsov formulae, is therefore of significant interest.

This direction was perhaps first pursued by Bruggeman--Motohashi \cite{BM05}, whose approach to constructing the two distinct-looking moments in  (\ref{basicmoto}) involves a combination of spectral projections, Hecke operators, and a two-fold limiting process applied to an automorphic kernel. The choice of the kernel is by no means straightforward: it entails delicate smoothing and regularization, along with careful control of convergence, regularity, and interchange of limits. These ingredients, in turn, illuminate the integral transform through the theories of Whittaker--Kirillov models and Bessel functions for the group $\hbox{PGL}_{2}(\R)$ (see \cite{CPS90}).  

Michel--Venkatesh \cite[Section 4]{MV06}, \cite[Section 4.5]{MV10}, and Reznikov \cite[p. 453]{R08} took a further step towards an account more intrinsic than \cite{BM05}. They envisioned proving Motohashi's formula within the structural framework of the \textit{strong Gelfand formation} and via a \textit{``regularized''} geodesic period for the product of two $\hbox{PGL}(2)$ Eisenstein series, \emph{formally} given by
  \begin{align}\label{regEIspe}
  	\text{Reg}-\int_{\Z^{\times}\setminus\R^{\times}} \, E(iy; s_{1}) E(iy; s_{2}) \, d^{\times} y,
  \end{align}
  which \emph{morally} suggests a matching of the spectra $L^{2}([\mathrm{PGL_{2}}])$ and $L^{2}([\mathrm{GL_{1}}])$, resulting in Motohashi's formula. However, numerous technical obstacles must be overcome to implement their strategy. In particular, the sources of various main (or degenerate) terms are not evident, and many integrals suffer from serious divergence issues due to the non-compactness of domains and the non-square-integrability of Eisenstein series. More refined regularization techniques than those developed in \cite{MV10} are therefore necessary; see also \cite{BHKM20} for other intricacies of this approach. Only very recently did Nelson \cite{Ne20+} establish full rigour for the Michel--Venkatesh--Reznikov period integral strategy, through a range of highly impressive and sophisticated techniques.

Wu \cite{Wu22} gave another interpretation of Motohashi’s formula by generalizing \cite{BM05} with a pre-trace formula on the Schwartz space of $2$-by-$2$ matrices, which incorporates the Godement--Jacquet and Tate $L$-functions on the spectral and geometric sides of the formula. Further explications and analytic subtleties were addressed in \cite{BFW21+}, which draws extensively on deep results in the theory of special functions. The authors of \cite{BFW21+} discuss the obstacles to obtaining an exact spectral inversion of Motohashi's formula.


\subsection{Our main results}\label{mainres} 

In this article, we present a new proof of the \emph{spectral inversion} of Motohashi's identity (\ref{basicmoto}) via the $\hbox{GL}(3)$ period integral method of \cite{Kw23}. Our approach offers fresh perspectives on the reciprocity phenomena and surmounts technical obstacles in earlier treatments. Furthermore, we elucidate the links between periods and the Moment Conjectures of \cite{CFK+05}.

In our previous work \cite{Kw23} (resp. \cite{Kw25}), we examine the (twisted) spectral first moment of $\hbox{GL}(3)\times \hbox{GL}(2)$ Rankin--Selberg $L$-functions with the $\hbox{GL}(3)$ automorphic form being a fixed Hecke--Maass cusp form. We uncover a higher-rank Motohashi-type identity that underpins Li's celebrated convexity-breaking results \cite{Li11}.  

This line of works belong to the recent framework \emph{Period Reciprocity}, for which only a handful of other instances are known: \cite{MV10, Bl12a, Ne20+,  Nu20+, Za20+, Za21, JN21+, Mi21+}.  The tools involved lie largely outside the traditional realm of analytic number theory, and they streamline the intricate chains of analytic-arithmetic transformations prevalent in the literature. We refer the reader to \cite[Sections 3-4]{Kw23}, \cite[Section 1]{Kw25} and \cite{JN21+} for more in-depth discussions of the implementations and technical features of this framework.

To state our main results, let $\alpha\in \mathfrak{a}_{\C}:=  \{ (\alpha_{1},\alpha_{2}, \alpha_{3})\in \C^{3}: \alpha_{1}+\alpha_{2}+\alpha_{3}=0 \}$, and $\mathfrak{M}_{-\alpha}^{(3)}(s;  H)$ denote the following shifted cubic moment:\footnote{ The shifts are labelled with minus signs for automorphic reasons, as we will discuss in the main body of this article. }
\begin{align}\label{def: speccubicmom}
     \sum_{j=1}^{\infty} \,  H(\mu_{j}) \,  & \frac{\prod_{i=1}^{3} \, \Lambda(s- \alpha_{i}, \phi_{j})}{\langle \phi_{j}, \phi_{j}\rangle} 
	    \ + \  \frac{1}{4\pi i } \ \int_{(0)} \ H(\mu) \ \frac{\prod_{i=1}^{3} \prod_{\pm} \, \Lambda(s\pm \mu -\alpha_{i}) }{|\Lambda(1+2\mu)|^2} \ d\mu, 
\end{align}
where ${\textstyle(\phi_{j})_{j=1}^{\infty}}$ is an orthogonal basis of \emph{even} Hecke--Maass cusp forms of $\mathrm{SL}_{2}(\Z)$ with respect to the Petersson inner product $\langle\, \,\cdot\, , \,\cdot \,\rangle$ (see (\ref{def: petersson})), and satisfies ${\textstyle\Delta\phi_{j}= ( 1/4-\mu_{j}^2) \phi_{j}}$ with $\Delta:=-y^2 (\partial_{x}^2+\partial_{y}^2)$. Also, let $\Lambda(s)$ and $\Lambda(s,\, \phi_{j})$ be, respectively, the \emph{completed} $\zeta$-function and \emph{completed} $L$-function of  $\phi_{j}$.

 To simplify our residual calculations (as in \cite{CFK+05}), we introduce:

\begin{conv}\label{con: smallshift}
    Without otherwise specified, we fix $\epsilon_{0}:=1/1000$, and further assume that $\alpha_{i}$'s are distinct and satisfy $ |\alpha_{1}|,  |\alpha_{2}|  <  \epsilon_{0}/1000$.
\end{conv}

In practical applications of any spectral summation formula, it is essential to have an ample supply of admissible test functions. The following regularity assumption serves our purposes and is satisfied by those commonly used in the literature (cf. \cite[Remark 5.27]{Kw23}).

\begin{reg}\label{regassuCeta}
Fix $\eta>40$. Define $\mathcal{C}_{\eta}$ to be the class of holomorphic functions $H$ on the vertical strip  $|\re \mu|< 2\eta$ such that  $H(\mu)=H(-\mu)$ and  $H(\mu)  \ll  e^{-2\pi|\mu|} $ in the strip.
\end{reg}

In the following, let $\Gamma_{\R}(s):= \pi^{-s/2}\,\Gamma(s/2)$; $d^{*}\mby := (y_{0}y_{1}^2)^{-1}\, dy_{0}\, dy_{1}$;\, $e(x):=e^{2\pi i x}$;  $H^{\flat}:(0,\infty)\rightarrow \C$ denote the inverse Kontorovich--Lebedev transform of $H$ (Definition \ref{defwhittrans});  $d^{W}\mu:= |\Gamma_{\R}(\mu)|^{-2}\, \frac{d\mu}{4\pi i}$ be the spherical Whittaker--Plancherel measure for $\mathrm{PGL}_{2}(\R)$;  and $W_{-\alpha}(x,   y_{1})$ be the Jacquet--Whittaker function for $\mathrm{PGL}_{3}(\R)$ (see (\ref{jac})). 

\subsubsection{The reciprocity formula}

\begin{thm}\label{maingl3gl2Eiscase}
  Let  $H\in \mathcal{C}_{\eta}$. On the vertical strip $1/4 + \epsilon_{0}  <  \re s  < 3/4$, the following formula holds: 
	\begin{align}
	  & \mathfrak{M}_{-\alpha}^{(3)}(s;  H) \ + \  \mathcal{R}_{-\alpha}(s; H) \ + \  \mathcal{R}_{\alpha}(1-s; H) \nonumber\\
&\hspace{50pt} 	\ = \  \sum_{i=1}^{3} \, \mathcal{M}_{-\alpha}^{i}(s; H) \, + \ \frac{1}{2}  \  \int_{(1/2)} \zeta\left(2s-s_{0}\right)\  \prod_{i=1}^{3} \   \zeta(s_{0}+\alpha_{i})  \left(\mathcal{F}_{\alpha}H\right)\left(s_{0},  s\right) \ \frac{ds_{0}}{2\pi i},\label{dual4thzeta}
	\end{align}
where
\begin{align}\label{intrans: firstform}
(\mathcal{F}_{\alpha} H)(s_{0},  s )  \ := \ \sum_{\pm} \,  \int_{0}^{\infty}  \int_{0}^{\infty}   \, H^{\flat}\Big( \frac{y_{1}}{\sqrt{1+y_{0}^2}}\Big) \int_{0}^{\infty}  W_{-\alpha}(x,   y_{1})  e(\pm xy_{0})  x^{s_{0}-1}\,  d^{\times} x\, \frac{ y_{0}^{2s-s_{0}}y_{1}^{s-\frac{1}{2}} }{(1+y_{0}^{2})^{\frac{s}{2}+\frac{1}{4}-s_{0}}}  \, d^{*}\mby,
\end{align}
\begin{align}
  \mathcal{M}_{-\alpha}^{0}(s; H) \, &:= \,  \prod_{1\le i<j\le 3} \, \zeta(2s-\alpha_{i}-\alpha_{j})   \   \int_{(0)}  \, H(\mu) \, \prod_{i=1}^{3} \, \prod_{\pm} \, \Gamma_{\R}(s\pm \mu-\alpha_{i}) \ d^{W}\mu,\label{thm:0swapterm} \\
  \mathcal{M}_{-\alpha}^{1}(s; H) \, &:= \, \sum_{i=1}^{3}\,  \ \zeta(2s-\sum_{\substack{1\le j\le 3\\ j\neq i}} \, \alpha_{j})\prod_{\substack{ 1\le j\le 3\\ j\neq i}}  \, \zeta(1+\alpha_{i}-\alpha_{j}) \nonumber\\
  &\hspace{80pt}\times \   \int_{(0)} \, H(\mu)\prod_{\pm} \, \Gamma_{\R}(1-s\pm \mu+\alpha_{i}) \prod_{\substack{ j\neq i}}\, \Gamma_{\R}(s\pm \mu-\alpha_{j}) \, d^{W}\mu, \label{1swapterm} \\
  \mathcal{M}_{-\alpha}^{2}(s; H) \, &:= \,   \ \sum_{i=1}^{3} \ \zeta\big(2-2s+ \sum_{\substack{1\le j\le 3\\ j\neq i}} \alpha_{j}\big) \prod_{\substack{ 1\le j\le 3\\ j\neq i}}\zeta(1-\alpha_{i}+\alpha_{j})\nonumber\\
  &\hspace{80pt} \, \times \, \int_{(0)} \, H(\mu)  \, \prod_{\pm} \, \Gamma_{\R}( s\pm \mu-\alpha_{i}) \prod_{\substack{1\le j\le 3\\ j\neq i}}\Gamma_{\R}(1-s\pm \mu+\alpha_{j})   \ d^{W}\mu,  \label{2swapterm}\\
   \mathcal{M}_{-\alpha}^{3}(s; H) \, &:= \, \  \prod_{1\le i<j\le 3} \, \zeta(2-2s+\alpha_{i}+\alpha_{j}) \int_{(0)} \,  H(\mu)\,    \prod_{i=1}^{3}  \ \prod_{\pm} \,  \Gamma_{\R}( 1-s\pm \mu + \alpha_{i}) \, d^{W}\mu,  \label{3swapterm}
\end{align}
\begin{align}
  \mathcal{R}_{\alpha}(s; H) \ &:= \  2 \ \sum_{i=1}^{3} \ H(s+\alpha_{i}) \ \frac{\prod_{\substack{1\le j\le 3\\ j\neq i}} \ \Lambda(1+\alpha_{i}-\alpha_{j}) \Lambda(2s+\alpha_{i}+\alpha_{j})}{\Lambda\left(1+2s+2\alpha_{i}\right)}.  \label{4thmoMT2}
\end{align}
\end{thm}

\subsubsection{Main terms}
Recent works on Moment Conjectures (e.g., \cite{Pe15} and \cite{CIS12}) have shown that it is both technically and aesthetically preferable to state and prove results in terms of \textit{complete} (rather than \textit{incomplete}) $L$-functions. Indeed, the complete set of the main terms of a moment of $L$-functions is expected to reflect the symmetries originating from successive applications of the functional equation for the family, as we will explain in Section \ref{CFKRSSOU}. (Incidentally,  the root number of the family is conveniently fixed in our theorem.) Nevertheless, the admissible class of test functions $\mathcal{C}_{\eta}$  is sufficiently large to deduce a version of Theorem  \ref{maingl3gl2Eiscase} for incomplete $L$-functions.

Here is an extra benefit.  The symmetries of all main terms for the cubic moment can now be addressed uniformly via archimedean Rankin--Selberg-type calculations in the style of Bump \cite{Bump88} and Stade \cite{St01}.  The main result of \cite{Bump88} only accounts for the diagonal term.  To capture the symmetries of the remaining seven off-diagonal terms, we prove three new \emph{Mellin--Barnes identities}.

	\begin{thm}[Propositions \ref{1swaprop}, \ref{Eissimpl},  \& \ref{secMTcom}]\label{CFKGamma}
    Let $J_{w, \alpha}(s; h)$ and $(\mathcal{F}_{\alpha}H)(s_{0},  s ) $ be the integral transforms defined in (\ref{1swpint}) and (\ref{intrans: firstform}) respectively. On the vertical strip $\epsilon_{0}  <  \sigma  < 1-\epsilon_{0}$, we have
			\begin{align}
					J_{w, \, -\alpha}(s; h)
				\ &= \   \frac{1}{4} \ \prod_{i=1}^{2} \, \Gamma_{\R}(1-\alpha_{i}^{w}+\alpha_{3}^{w})^{-1} \int_{(0)} \, H(\mu)\prod_{\pm} \, \Gamma_{\R}(1-s\pm \mu+\alpha_{3}^{w}) \prod_{i=1}^{2}\, \Gamma_{\R}(s\pm \mu-\alpha_{i}^{w}) \, d^{W}\mu; \label{thm1swap}\\ 
	\left(\mathcal{F}_{\alpha}H\right)\left(1-\alpha_{i},s\right) \, &= \,  	2 \  \frac{\Gamma_{\R}( 2s-1+\alpha_{i})}{\Gamma_{\R}( 2-2s-\alpha_{i})}  \int_{(0)} \, H(\mu)   \prod_{\pm} \, \Gamma_{\R}( s\pm \mu-\alpha_{i}) \prod_{\substack{1\le j\le 3\\ j\neq i}}\Gamma_{\R}(1-s\pm \mu+\alpha_{j})   \ d^{W}\mu; \label{thm2swap}\\
\left(\mathcal{F}_{\alpha}H\right)\left(2s-1,  s \right) 
		\ &= \ 2 \, \prod_{j=1}^{3} \, \frac{\Gamma_{\R}(2s-1+ \alpha_{j})}{\Gamma_{\R}(2-2s- \alpha_{j})} \, \int_{(0)} \,  H(\mu)   \prod_{j=1}^{3} \prod_{\pm} \,  \Gamma_{\R}( 1-s\pm \mu + \alpha_{j}) \, d^{W}\mu,  \label{thm3swap}
			\end{align}
            for $i=1,2,3$, $ w\in \left\{  w_{2}, w_{4}, w_{\ell} \right\}$ (see (\ref{Weyletl})), and $\alpha_{i}^{w}$'s defined in (\ref{weylactLan}). 
		\end{thm}

 In our approach, the sources of the main terms are intrinsic and transparent. In essence, it suffices to know the $\mathrm{GL}(3)$ \emph{Fourier--Whittaker period and expansion} (Lemma \ref{genFourlem} and Proposition \ref{incomf}):
\begin{itemize}
	\item \textbf{Cubic moment.} See Proposition  \ref{MTclass}.  The identities (\ref{thm1swap}), (\ref{thm2swap}) and (\ref{thm3swap}) correspond to the ``$1$-swap''  terms (\ref{1swapterm}), ``$2$-swap''  terms (\ref{2swapterm}), and ``$3$-swap''  term (\ref{3swapterm}),  respectively. The diagonal term (\ref{thm:0swapterm}) is also known as the ``$0$-swap'' term.
	
	\item \textbf{Fourth moment.} Its full set of main terms arises from 
	\begin{itemize}
		\item   The continuous spectrum: resulting in two ``$1$-swap'' terms and one ``$2$-swap'' term that can be collected into $\mathcal{R}_{\alpha}(1-s; H) $ in (\ref{dual4thzeta}); see Lemma \ref{contrcon}; 
		
		\item The non-constant degenerate part of the Fourier expansion:  resulting in two  ``$1$-swap'' terms and one ``$0$-swap'' term that can be collected into $\mathcal{R}_{-\alpha}(s; H) $  in (\ref{dual4thzeta}); see  Proposition \ref{degprop}. 
	\end{itemize}
\end{itemize}
The notion of ``swap'' was introduced in \cite{CK15a, CK15b} as a refinement of the CFKRS conjectures; see Sections \ref{CFKOver}-- \ref{CFKUsym} for detailed discussions. As already hinted in (\ref{thm1swap}),  the symmetries of the Weyl group of $\mathrm{GL}(3)$ are essential in understanding the structures of the main terms.

Previously, Hughes and Young \cite{HY10, Y11} investigated the main terms for the fourth moment of the $\zeta$-function
by a method different from Motohashi’s. They employed the approximate functional equation (for the product of four $\zeta$'s) in conjunction with the $\delta$-symbol expansion (e.g., \cite{DFI94}) applied to the shifted divisor sums, and careful smoothing and truncations. The upshot of their analysis is a set of exotic cancellations and combinations of many complicated-looking terms that, strikingly, reproduces the CFKRS predictions up to acceptable error. Petrow \cite[Sections 2-3]{Pe15}  studied the full set of main terms for the cubic moment via an argument of comparable difficulty. 

Our approach to extracting the full set of main terms differs from the existing ones. We are able to treat the main terms for  \textit{both} the cubic and fourth moments in a unified manner, without ad hoc cancellations or delicate combinations of terms. In Sections \ref{CFKRS} and \ref{CFKRSU}, we verify the agreement with the CFKRS conjectures by a short argument, using only Theorem \ref{CFKGamma}, the relation $\alpha_{1}+\alpha_{2}+\alpha_{3}=0$, and the functional equation for the $\zeta$-function. Our perspective may shed light on recent developments in the mean value theorems for long Dirichlet polynomials \cite{CK19, HN22, BTB22+, CF22+},  which aim to identify the arithmetic sources of main terms in various families of moments of $L$-functions \textit{without} the approximate functional equations.  

For more abstract treatments of the main terms via regularized or degenerate functionals in automorphic representations; see \cite[Section 11]{Ne20+} and  \cite[Section 1.5]{Wu22}.

\subsubsection{Integral transform}
In this work, we carry out explicit calculations using Mellin--Barnes integrals, thereby deducing a range of results directly from first principles (Section \ref{idsect: Barnes}). Using essentially the same ideas as in the proof of Theorem \ref{CFKGamma}, we express the integral transform $\mathcal{F}_{\alpha}H$ in a symmetric form via the Gauss hypergeometric function $\mbF(\cdots)$ defined in (\ref{2F1define}).

\begin{thm}[Corollary \ref{cor: hypergtrans}]\label{thm: newGaussintrans}
Suppose $s=1/2$ and $\re s_{0}=1/2$. Then
\begin{align}
	(\mathcal{F}_{\alpha} H)(s_{0}, 1/2) 
		\ = \    &4 \, \Gamma_{\R}(1-s_{0})\prod_{i=1}^{3}\, \Gamma_{\R}(s_{0}+\alpha_{i})  \int_{0}^{\infty} \,  \mbF\left(\begin{matrix}
   \frac{1-s_{0}}{2},\; \frac{1-s_{0}-\alpha_{1}}{2}\\
    \frac{1}{2}
  \end{matrix}\Bigm| -x^2\right)\mbF\left(\begin{matrix}
\frac{s_{0}+\alpha_{2}}{2},\; \frac{s_{0}+\alpha_{3}}{2}\\
    \frac{1}{2}
  \end{matrix}\Bigm| -x^2\right)\nonumber\\
        & \hspace{30pt}\cdot \, \int_{(0)} \ H(\mu) \prod_{\pm}\, \prod_{\pm} \,  \Gamma_{\R}(1/2\pm\mu\pm\alpha_{1}) \mbF\left(\begin{matrix}
   \frac{1}{4} +  \frac{\alpha_{1}+\mu}{2},\; \frac{1}{4} +  \frac{\alpha_{1}-\mu}{2}\\
    \frac{1}{2}
  \end{matrix}\Bigm| -x^2\right)\,  \,	d^{W}\mu \ dx.  
		\end{align}    
\end{thm}


	\subsection{Sketch, road map and discussion of our argument}\label{feaDIS}
 Let $\mrG_{n}:= \mathrm{GL}_{n}(\R)$, $\Gamma_{n}:=\mathrm{SL}_{n}(\Z)$, $\mrU_{n}$ be the standard maximal unipotent subgroup of $\mrG_{n}$, $[\mathbb{G}_{m}]:= \Z^{\times}\setminus \R^{\times}$,  $[\mrU_{n}]:= (\mrU_{n}\cap \Gamma_{n})\setminus \mrU_{n}$, and $[\ovmrG_{n}]:= \Gamma_{n}\setminus \mrG_{n}/\R_{>0}^{\times}$.  In \cite{Kw23}, we study the analogue of Theorem \ref{maingl3gl2Eiscase} for a Hecke--Maass cusp form $\Phi: [\ovmrG_{3}]\rightarrow \C$, and show that the corresponding Motohashi-type formula follows from a seemingly trivial equality for a  ``Fourier--Hecke period'':
	\begin{align}\label{trivstart}
		\int_{[\mrU_{2}]}  \, \Big( \int_{[\mathbb{G}_{m}]} \,  \Phi(zn)  \ d^{\times}z\Big) \, \overline{\psi(n)}\, dn  \ = \ \int_{[\mathbb{G}_{m}]}  \ \Big( 	\int_{[\mrU_{2}]} \,  \Phi(nz) \overline{\psi(n)} \  dn\Big) \, d^{\times}z,
	\end{align}
   where  $\Phi(g):= \Phi\begin{psmallmatrix}
        g & \\
          & 1
    \end{psmallmatrix}$ \footnote{ Some authors use the notation  $(\mathbb{P}_{2}^{3}\Phi)(g)$ to indicate the projection (or restriction) of a function on $\mrG_{3}$ to $\mrG_{2}$. }, and  $\psi \in \mathrm{Hom}([\mrU_{2}], \, \C^{\times})$ is fixed and non-trivial. In \cite[Chapter 1.3.1]{Kwa-thesis}, the overarching structure for implementing (\ref{trivstart}) is illustrated by a Reznikov diagram (\cite{R08}):
\begin{center}
    \begin{tikzcd}[sep=small]
              & \mrG_{3}  &          \\
   \mrG_{2}  \arrow[hookrightarrow]{ur}     &        & \mrU_{3} \arrow[hookrightarrow]{ul}  \\
              & \mathrm{U}_{2} \arrow[hookrightarrow]{ul}
       \arrow[hookrightarrow]{ur}  &
    \end{tikzcd}
     \end{center}
Each arrow of the form $H\hookrightarrow G$ represents a pairing (or period) over $[H]$, obtained by suitably restricting functions on $[G]$ to $[H]$. In our case, the pairings behind the arrows refer to:
\begin{itemize}
    \item \textbf{(Upper-left/ $\mathrm{\mathbf{GL}}_{3}\times \mathrm{\mathbf{GL}}_{2}$ Rankin--Selberg)} $\langle \, \Phi,  \,  \sigma \, \rangle_{[\ovmrG_{2}]}:= \int_{[\ovmrG_{2}]} \, \varphi(g) \overline{\sigma(g)} \, dg$,  where  $\sigma  \in L^{2}([\ovmrG_{2}])$, and  $\varphi(g):= \int_{[\mathbb{G}_{m}]} \, \Phi(zg) \, d^{\times} z$. 

    \item  \textbf{(Lower-left/ $\mathrm{\mathbf{GL}}_{2}$ Fourier--Whittaker)} $ \langle \, \sigma, \, \psi \, \rangle_{[\mathrm{U}_{2}]}:= \int_{[\mathrm{U}_{2}]} \, \sigma(n) \overline{\psi}(n) \, dn$.

    \item \textbf{(Upper-right/ $\mathrm{\mathbf{GL}}_{3}$ Fourier--Whittaker)} $\langle \, \Phi, \, \Psi \, \rangle_{[\mathrm{U}_{3}]}:= \int_{[\mathrm{U}_{3}]} \, \Phi(u) \overline{\Psi}(u) \, du$, for $\Psi\in \mathrm{Hom}([\mrU_{3}], \,\C^{\times})$.
    
    \item \textbf{(Lower-right/ $\mathrm{\mathbf{GL}_{3}}$ partial unipotent)} $\langle\, \Phi, \ \psi \, \rangle_{[\mathrm{U}_{2}]}':= \int_{[\mathrm{U}_{2}]}  \, \Phi(n)\overline{\psi}(n) \, dn$.
\end{itemize}
The diagram also indicates the use of the \emph{spectral expansion} of $\varphi$ in $L^{2}([\ovmrG_{2}])$, together with the \emph{Fourier--Whittaker expansion} of $\Phi$ (over $[\mrU_{3}]$ and with respect to $g\mapsto \begin{psmallmatrix}
    1 & \\
      & g
\end{psmallmatrix}$) when evaluating $\langle\, \Phi, \ \psi \, \rangle_{[\mathrm{U}_{2}]}'$. As shown in \cite[Section 6A]{Kw23}, a partial abelian Fourier expansion already suffices for the latter step and, in fact, captures the essence of our method more neatly. Then, the \emph{multiplicity-one principle} and \emph{Mellin inversion formula} render the emergence of the dual moment and the structure of the integral transform transparent; see Remark \ref{rem: project}, \cite[Section 5.3]{Kw25} and \cite[Section 4]{Kw23}.  Theorem \ref{CFKGamma}, and consequently the agreement with the CKFRS Conjectures, follows naturally from this form of the integral transform, as we explain in Section \ref{sect: 23swap}.

In this work, we investigate (\ref{trivstart}) in the case of the $\mathrm{SL}_{3}(\Z)$  Eisenstein series associated with the isobaric sum $ |\cdot|^{\alpha_{1}} \, \boxplus |\cdot|^{\alpha_{2}} \, \boxplus \, |\cdot|^{\alpha_{3}} $. Although periods of Eisenstein series play a fundamental role in many notable instances of spectral reciprocity, handling them rigorously remains highly nontrivial, and no general method is currently available. In addition to \cite{Ne20+}, we also mention \cite{JN21+, Mi21+}, which employ the regularization method of Ichino--Yamana \cite{IY15}, and \cite{Za20+}, which develops Michel--Venkatesh’s deformation and regularization techniques.

In (\ref{trivstart}), the most immediate obstruction comes from the divergence of the integrals over $[\mathbb{G}_{m}]$ when $\Phi$ is Eisenstein. Resolving this highlights the advantage of interpreting the classical Motohashi phenomena of (\ref{basicmoto}) and Theorem \ref{maingl3gl2Eiscase}, which are intrinsically dualities between $\mathrm{GL}(1)$ and $\mathrm{GL}(2)$ $L$-functions, through the periods and automorphic forms for the ``larger'' group $\mathrm{GL}(3)$. Our strategy differs from Nelson's \cite{Ne20+}.  We crucially exploit the structure of the $\mathrm{GL}(3)\times \mathrm{GL}(2)$ Rankin--Selberg period, in which the unfolding \emph{does not} rely on the full automorphy of the $\mathrm{GL}(3)$ form, at least within the region of absolute convergence (upon attaching the factor $|\det *|^{s-1/2}$ to (\ref{trivstart}) and taking $\re s\gg 1$). This feature allows the spectral calculations in the cuspidal and Eisenstein cases to proceed in complete parallel, and it is worth noting that only the \emph{non-degenerate} part of the Fourier--Whittaker expansion of $\Phi$ is used here.

The development of the dual side of (\ref{trivstart}) embodies most of the novelties of this work and reveals several aspects not present in our earlier papers \cite{Kw23, Kw25}. We begin by evaluating the unipotent period over $[\mrU_{2}]$, which, upon invoking a suitable automorphy of $\Phi$ (Proposition \ref{incomf}), decomposes into a sum of three components (Proposition \ref{prop: decompo}) via the Fourier--Whittaker expansion of $\Phi$ (Lemma \ref{genFourlem}). The unipotent integration eliminates the degenerate components of $\Phi$ responsible for divergence, by a two-line argument (Lemma \ref{lem: whywecanreg}). We are then in a position to evaluate the integral over $[\mathbb{G}_{m}]$ for $\re s \gg 1$ (again with the factor $|\det *|^{s-1/2}$). The remaining ``benign'' degenerate terms of $\Phi$ contribute to the main terms of our spectral reciprocity formula in non-trivial ways (Sections \ref{1-swap}--\ref{degt}). We exploit the equivariance of the Fourier--Whittaker periods at several points, which offers significant advantages over the complicated Hecke combinatorics that typically arise from the $\mathrm{GL}(3)$ Voronoi formula in Kuznetsov-based approaches. Finally, by examining the explicit forms of each component in the development of (\ref{trivstart}) in terms of $\zeta$- and $L$-functions, we find that the whole expression admits analytic continuation to the region for $s$ described in Theorem \ref{maingl3gl2Eiscase}. The rest of the main terms predicted by the CFKRS conjectures arise naturally as polar contributions in the continuation process. These are the subjects of Sections \ref{diagpreoff}, \ref{offdiagEis} and \ref{CFKRS}.

The unipotent integration also yields an ample supply of admissible test functions on the spectral side of our reciprocity formula, with their regularity and shapes tailored to standard ones in the literature of moments of $L$-functions. This is obtained through the Whittaker--Plancherel formula (Lemma \ref{inKLconv}) with a Poincar\'{e} series incorporated into the period construction (\ref{trivstart}) (Corollary \ref{prop: spectralcont}). The Poincar\'{e} series further enables us to specify the archimedean component within the spectrum of $L^{2}([\ovmrG_{2}])$, a feature particularly favourable for spectral applications, much as in the many uses of the Kuznetsov formula. The focus of this work is an \emph{exact}, \emph{spectral} version of Motohashi’s formula restricted to the spherical part of $L^{2}([\ovmrG_{2}])$, complementing the results of \cite{Mo93, BM05}; accordingly, we work with a spherical Poincar\'{e} series. If desired, one may similarly localize to another subspace of $L^{2}([\ovmrG_{2}])$ by prescribing a different minimal $K$-type.

Moreover, to present the period integral and automorphic machinery to a broader analytic number theory audience, we illustrate our ideas with the full modular group and $\Phi$ spherical, which allows for explicit formulae (e.g., Lemma \ref{vtmellin}) and more transparent computations. These choices serve purely expository purposes and are by no means essential to our method.
  
One should be able to recover the results of \cite{Mo93, BM05} with our method.  In that setting,  our analysis should be simpler---we do not require the Atkinson dissection  (or any geometric dissection), and their four-variable continuation argument can be replaced by a single-variable one, as  the ``shifts'' play no essential role in  our analytic continuation argument (they are used only to simplify residual calculations).  We leave these considerations for future work.


\subsection{Additional Remarks}
It is clear that the techniques of this paper carry over to the case when $\Phi$ is a  $\mathrm{SL}_{3}(\Z)$ Eisenstein series associated with $\phi \, \boxplus \, 1$, where $\phi$ is a Hecke--Maass cusp form of $\mathrm{SL}_{2}(\Z)$.  Theorem  \ref{maingl3gl2Eiscase} now takes the shape
\begin{align}
	\hspace{10pt} 	\sum_{j=1}^{\infty} \,  H(\mu_{j}) \ & \frac{\Lambda( 1/2, \phi \otimes \phi_{j}) \Lambda(1/2, \phi_{j}) }{\langle \phi_{j}, \phi_{j}\rangle}  \nonumber\\
	&	\  \ =  \  \ \frac{1}{4\pi} \ \int_{\R} \ L(1/2+it, \phi)  \left| \zeta\left(1/2+it\right)\right|^2  (\mathcal{F}_{\alpha(\phi)} H)\left(1/2+it,  1/2 \right) \ dt \ \ + \ \  (***). 
\end{align}
This type of \textit{mixed moments} is also of interest to the literature due to the applications to simultaneous non-vanishing;  for example, \cite{BBFR20, Li09, Nu20+, X11}.

We have not yet fully exploited the strength of our method in this work, particularly in its non-archimedean aspects. For example, in \cite{Kw25}, we establish a $\hbox{GL}(3)$ Motohashi-type formula that dualizes $\hbox{GL}(2)$ twists of Hecke eigenvalues into $\hbox{GL}(1)$ twists by Dirichlet characters. In the isobaric case, this yields a short proof of the key spectral identity of  \cite{Y11} and  \cite{BHKM20} for the fourth moment of Dirichlet $L$-functions.

 \subsection{Outline}\label{prelim}
  Section \ref{notconv} introduces the notations and conventions used throughout this article. Section \ref{prel} reviews the essential definitions and preliminary results. Section \ref{Stirl} focuses on the integral transform of our spectral reciprocity formula and contains the proof of Theorem \ref{thm: newGaussintrans}. For readers less familiar with the CFKRS conjectures, an overview is included in Section \ref{CFKRSSOU}. The proof of Theorem \ref{maingl3gl2Eiscase} occupies the whole Section \ref{proofEismaTHm}, with a road map provided in Section \ref{feaDIS}. The proof of Theorem \ref{CFKGamma} can be found in Sections  \ref{1-swap} and \ref{sect: 23swap}. In Sections \ref{CFKRS} and \ref{CFKRSU}, we explain the connections between Theorems \ref{maingl3gl2Eiscase}--\ref{CFKGamma} and the CFKRS conjectures. We end the article with two observations (Section \ref{sect: comment}).


  \section{Notations and conventions}\label{notconv}

  \subsection{The parameters (Convention \ref{con: smallshift})}
  We fix $\epsilon_{0}:=1/1000$, and assume that $\alpha\in \mathfrak{a}_{\C}:= \{ (\alpha_{1},\alpha_{2}, \alpha_{3})\in \C^{3}: \alpha_{1}+\alpha_{2}+\alpha_{3}=0 \}$, the components $\alpha_{i}$'s are distinct and satisfy $ |\alpha_{1}|,  |\alpha_{2}|  <  \epsilon_{0}/1000$. We let $\balph :=\max_{1\le i\le 3} \, |\re \, \alpha_{i}| \,(<\epsilon_{0}/500)$.

  \subsection{Test functions (Regularity \ref{regassuCeta})}
 We fix $\eta>40$. Our test function $H$ lies in the class $ \mathcal{C}_{\eta}$. The function $h=H^{\flat}$ is the Kontorovich--Lebedev inversion of $H$ as defined in (\ref{invers}).

\subsection{Analysis}
   We use the same symbol to denote a function and its analytic continuation. We denote the real part of $s$ by $\sigma$. We write ``$f(y) \ll_{A} g(y)$'' (or ``$g(y)\gg_{A} f(y)$'') if there exists a  constant $C>0$, depending on $A$, such that $|f(y)|\le  C|g(y)|$ for sufficiently large $y$.
   
   The Euler $\Gamma$-function is defined by 
   \begin{align}
       \Gamma(s) \, := \, \int_{0}^{\infty} e^{-x} x^{s} \, d^{\times} x
   \end{align}
   on $\re s >0$. It admits a meromorphic continuation to $\C$. Let $\Gamma_{\R}(s):= \pi^{-s/2}\Gamma(s/2)$ for $s\in \C$. We use both $\Gamma(s)$ and $\Gamma_{\R}(s)$ as the former is common in our references, but the latter simplifies our formulae on several occasions. 

The contours of the Barnes integrals 
\begin{align}
    \int_{-i\infty}^{i\infty} \, \cdots \,  \int_{-i\infty}^{i\infty} \ (\cdots) \    \frac{ds_{1}}{2\pi i} \, \cdots \,  \frac{ds_{k}}{2\pi i}
\end{align}
are chosen so that they pass to the right of all poles of the gamma functions of the form $\Gamma(s_{i}+ a)$, and to the left of all poles of those of the form $\Gamma(a -s_{i})$.

Let $\Lambda(s) := \Gamma_{\R}(s)\zeta(s)$ denote the completed Riemann zeta function. It admits a holomorphic continuation to $\C$ except for simple poles at $s=0,1$, and satisfies the functional equation 
\begin{align}\label{riemannFE}
    \Lambda(s) \ = \  \Lambda(1-s)
\end{align}
for any $s\in \C$. 

The $K$-Bessel function is defined by
\begin{align}\label{def: Kbessfunc}
    K_{\mu}(y) \, := \, \frac{1}{2}\int_{0}^{\infty} \exp\Big(-\frac{y}{2}(t+t^{-1})\Big) t^{-\mu} \, d^{\times} t
\end{align}
for any $y>0$ and $\mu\in \C$.

The Gauss $_{2}F_{1}$ hypergeometric function is defined by
	\begin{align}\label{2F1define}
  \  \mbF\left(\begin{matrix}
    a,\; b \\
    c
  \end{matrix}
  \Bigm|z
\right)  \ := \  \frac{\Gamma(c)}{\Gamma(a) \Gamma(b) }\, \sum_{n=0}^{\infty} \ \frac{\Gamma(a+n) \Gamma(b+n) }{\Gamma(c+n) } \frac{z^{n}}{n!}. 
 	\end{align}
The series converges absolutely	on $|z|<1$ if $a,b \not \in \Z_{\le 0}$;  and  on $|z|=1$ if $\re (c-a-b)>0$. Furthermore, it has an analytic continuation (in $z$) to $\C - [1, \infty)$.

Let $h: (0,\infty)\rightarrow \C$ be a continuous function. Its Mellin transform and inversion formula are given, respectively,  by
\begin{align}\label{def: Mell}
    \widetilde{h}(s) \, := \, \int_{0}^{\infty} h(y)y^{s} \, d^{\times} y \hspace{15pt} \text{and} \hspace{15pt} h(y) \, = \, \int_{(\sigma)} \, \widetilde{h}(s)y^{-s} \, \frac{ds}{2\pi i},
\end{align}
provided that both integrals converge absolutely. 

 \subsection{Groups}\label{sect: groupconven}
   The Weyl group $W_{3}$ of $\mathrm{GL}_{3}$ is described as in (\ref{Weyletl}). For $n \in \{2,3\}$, let $\mrG_{n}:= \mathrm{GL}_{n}(\R)$, $\overline{\mrG}_{n}:= \mrG_{n}/\R_{>0}^{\times}$, $\Gamma_{n}:=\mathrm{SL}_{n}(\Z)$ the full modular group;  $\mathrm{K}_{n}:=\mathrm{O}_{n}$ the maximal compact subgroup of $\mrG_{n}$; $\mrU_{n}$ the subgroup of upper-triangular unipotent matrices in $\mrG_{n}$; $\mrN_{ij}$ the one-parameter unipotent subgroup attached to the $(i,j)$-th entry, where $i\neq j$. 
   
   The following quotients occur frequently in this work:  $[\mrU_{n}]:= (\mrU_{n}\cap\Gamma_{n})\setminus \mrU_{n}$, $[\mrN_{ij}]:= (\mrN_{ij}\cap \Gamma_{n})\setminus \mrN_{ij}$, $[\mrG_{n}]:= \Gamma_{n}\setminus \mrG_{n}$, $[\ovmrG_{n}]:= \Gamma_{n}\setminus \ovmrG_{n}$, $[\Gamma_{n}]:= (\mrU_{n}\cap \Gamma_{n})\setminus \Gamma_{n}$, and  $\mathfrak{h}^n :=\mrG_{n}/\mathrm{K}_{n}$.  Let $\mathrm{Y}^{+}$ be the group of matrices of the form $\mby:= \mathrm{diag}(y_{0}y_{1}, y_{0},1)$ with $y_{0}, y_{1}>0$. We have the following identifications: 
\begin{align*}
    \mathfrak{h}^2 \simeq \left\{ \begin{psmallmatrix}
 		1 & x\\
 		& 1
 	\end{psmallmatrix} \begin{psmallmatrix}
 		y & \\
 		& 1
 	\end{psmallmatrix} : x \in \R, \ y >  0  \right\} \hspace{20pt} \text{and} \hspace{20pt} \mathfrak{h}^3 \simeq \{\mbu \mby:\, \mbu\in \mrU_{3}, \, \mby \in \mathrm{Y}^{+}\}.
\end{align*}   
   
    We write  $\mbu:= (u_{ij})_{1\le i,j\le 3} \in \mrU_{3}$;  $\mbn, \, \mbn_{ij}(x) \in \mrN_{ij}$ with $x\in \R$ being the $(i,j)$-th entry of $\mbn_{ij}(x)$; \, $\mby(y_{0}, y_{1})= \mathrm{diag}(y_{0}y_{1}, y_{0},1)$ and $f(\mby)=f(y_{0}, y_{1})$ for $y_{0},\, y_{1}>0$; and for $\gamma\in \Gamma_{2}$ and $g\in \mrG_{3}$, we set
    \begin{align}\label{matrixshort}
        \gamma g \, := \, \mathrm{diag}(\gamma, 1)g.
    \end{align}

    \subsection{Additive characters}
     Let $e(x):=e^{2\pi i x}$ for $x\in \R$. We label the characters of $[\mrU_{3}]$ and $[\mrN_{ij}]$ by $\psi_{(m_{1},m_{2})}(\mbu):= e(m_{1} u_{23}+m_{2} u_{12})$  \ and $\psi_{m}(\mbn):= e(m n_{ij})$ respectively, where $m, m_{1}, m_{2}\in \Z$, $\mbu\in \mrU_{3}$, and   $\mbn \in \mrN_{ij}$. We also write $\psi_{\pm}= \psi_{(1, \, \pm 1)}$ and $\psi=\psi_{+}$.

  \subsection{Measures}
We use the shorthand $d^{*}\mby = (y_{0}y_{1}^2)^{-1}\,dy_{0} \,dy_{1}$.  The normalized spherical Whittaker--Plancherel measure for $\mrG_{2}$ is given by
  \begin{align}\label{def: LWmeas}
      d^{W}\mu \, := \,  |\Gamma_{\R}(\mu)|^{-2} \, \frac{d\mu}{4\pi i} \hspace{20pt} (\mu \, \in \, i\R).
  \end{align}
   The invariant measure on $\mrG_{2}$ can be described by the Iwasawa decomposition as follows:
  \begin{align*}
   dg \, = \,  \frac{ dx\, dy}{y^2} \, \frac{dz}{z} \, dk\hspace{15pt} \text{for} \hspace{15pt} g=\begin{psmallmatrix}
      1 & x\\
        & 1
  \end{psmallmatrix}\begin{psmallmatrix}
      y & \\
        & 1
  \end{psmallmatrix}\begin{psmallmatrix}
      z & \\
        & z
  \end{psmallmatrix} k,   
  \end{align*}
  where $x\in \R$, $y, z>0$, $k\in \mrK_{2}$, and $\mathrm{meas}(\mathrm{K}_{2})=1$. The Petersson inner product is defined by
\begin{align}\label{def: petersson}
	\left\langle \, \phi_{1},\,  \phi_{2} \, \right\rangle \ := \ \int_{[\ovmrG_{2}]} \ \phi_{1}(g) \,\overline{\phi_{2}(g)} \ dg 
\end{align}
for smooth functions $\phi_{1}, \phi_{2}$ on $[\ovmrG_{2}]$. If $\phi_{1}, \phi_{2}$ are spherical (i.e., $\mrK_{2}$-invariant), the domain of (\ref{def: petersson}) can of course be replaced by $\Gamma_{2}\setminus \mathfrak{h}^2$.

\subsection{Automorphic objects}
In this article, our Whittaker and automorphic functions are defined on  $\overline{\mrG}_{n}$ and are \emph{spherical} (or right $\mathrm{K}_{n}$-invariant). Equivalently, they are functions on the locally symmetric space $\mathfrak{h}^n$. For brevity, we may omit the term ``spherical'' in what follows.

 Our automorphic \emph{forms} are \emph{Hecke--Maass} forms for $\Gamma_{n}$, i.e., they are smooth functions of moderate growth on $\overline{\mrG}_{n}$ that are right $\mathrm{K}_{n}$-invariant, left $\Gamma_{n}$-invariant, and are joint eigenfunctions to the commutative ring of invariant differential operators on $\mathfrak{h}^n$ (denoted by $\mathcal{D}_{n}$) and the ring of Hecke operators. Their first Fourier coefficients are normalized to be $1$. They are either cuspidal or Eisenstein.

We use $E_{\min}^{(3)}$ to denote the minimal parabolic Eisenstein series for $\Gamma_{3}$, which we define in (\ref{minparaEis}). From Section \ref{Gl3Eiseri} onward, we use $\Phi$ exclusively for the \textit{completed} version of $E_{\min}^{(3)}$ as given in (\ref{compEis}). 

The Fourier--Whittaker coefficients and periods of an automorphic form $\phi$ of $\Gamma_{2}$ are denoted, respectively, by $\mathcal{B}_{\phi}(a)$ and $\mathcal{W}_{a}(g; \, \phi)$; those of an automorphic form $\Phi$ of $\Gamma_{3}$ are denoted by $\mathcal{B}_{\Phi}(m_{1}, m_{2})$ and $\mathcal{W}_{(m_{1},m_{2})}(g; \, \Phi)$. We write $W_{\mu}(g)$ and $W_{\alpha}(g)$ for the standard Whittaker functions for $\overline{\mrG}_{2}$ and $\overline{\mrG}_{3}$ respectively, and $W_{\alpha,\, w}^{(0,1)}(\mby)$ and $W_{\alpha, \, w}^{(1,0)}(\mby)$ for the degenerate Whittaker functions for $\overline{\mrG}_{3}$. This is no ambiguity as the group to which $g$ belongs is clear from the context. We use $\mu$ and $\alpha$ to denote the spectral parameters for, respectively, the automorphic forms of $\Gamma_{2}$ and $\Gamma_{3}$ (or Whittaker functions). In our case, $\mu$ is always purely imaginary because of (\ref{def: LWmeas}) and the known Selberg's conjecture for $\Gamma_{2}$ (\cite[Theorem 3.7.2]{Gold}). 

The following unipotent period for $\mrG_{3}$ plays an important role in our work:
\begin{align}
\mathcal{P}_{\psi}(g; \, F) \, := \, 	\int_{[\mrN_{12}]} F(
		\mbn g) \overline{\psi(\mbn)} \ d\mbn,   
\end{align}
where $F: \mrN_{12}(\Z)\setminus \mrG_{3}\rightarrow \C$ is a smooth function. It is customary to write $F(g):= F\begin{psmallmatrix}
        g & \\
          & 1
    \end{psmallmatrix}$ for $g\in \mrG_{2}$ and a function $F$ defined on $\mrG_{3}$. Suppose $F$ and $\phi$ are smooth functions on $[\mrG_{2}]$. We define
\begin{align}\label{def: Inewnotpairing}
   \mathcal{I}(s; \, \phi, \, F) \, := \,  \int_{[\mrG_{2}]}  \, \phi(g)F(g)  |\det g|^{s-\frac{1}{2}} \ dg \hspace{15pt} (s \, \in \, \C)
\end{align}
whenever the integral converges absolutely.  In our previous works \cite{Kw23, Kw25}, the integral (\ref{def: Inewnotpairing}) was instead denoted by $(\phi,\  (\mathbb{P}_{2}^{3}\widetilde{F})\, |\det *|^{\overline{s}-\frac{1}{2}}) _{\Gamma_{2}\setminus \mathrm{GL}_{2}(\R)}$, where $\widetilde{F}(g):= F(^{t}g^{-1})$ for $g\in \mrG_{3}$.

 
  
\section{Preliminary}\label{prel}

In this section, we collect results from the sources  \cite{Gold, Bump84} that are essential to our arguments. We have adjusted the conventions to strive for consistency and to follow the recent shift of conventions (closest to \cite{Bu20, Bu18}), which better align with those in the theory of automorphic representation.  Part of the preliminary results can also be found in  \cite{Kw23}. 

\subsection{Barnes' identities}\label{idsect: Barnes}

 	\begin{lem}\label{secBarn}
 	  For $a,b,c,d,e,f \in \C$ with $f=a+b+c+d+e$, we have
 		\begin{align}\label{secBar}
 			\int_{-i\infty}^{i\infty} \frac{\Gamma(w+a) \Gamma(w+b)\Gamma(w+c) \Gamma(d-w) \Gamma(e-w)}{\Gamma(w+f)} \ &\frac{dw}{2\pi i} \nonumber\\
 			& \hspace{-50pt}  \ = \   \frac{\Gamma(d+a) \Gamma(d+b) \Gamma(d+c) \Gamma(e+a) \Gamma(e+b)\Gamma(e+c)}{\Gamma(f-a) \Gamma(f-b) \Gamma(f-c)}.  
 		\end{align}
 \end{lem}

\begin{proof}
	See \cite[Chapter 6.2]{Ba64}. This is known as the \emph{second Barnes lemma}. 
\end{proof}

\begin{lem}\label{PfaBarnes}
For $x>0$, we have
\begin{align}
    \mbF\left(\begin{matrix}
    a,\; b \\
    c
  \end{matrix}
  \Bigm|-x^{-1}
\right) \, = \, \frac{\Gamma(c)}{\Gamma(a)\Gamma(b)} \, \int_{-i\infty}^{i\infty} \, \frac{\Gamma(w+a)\Gamma(w+b)\Gamma(-w)}{\Gamma(w+c)} x^{-w} \, \frac{dw}{2\pi i}. 
\end{align}
\end{lem}

\begin{proof}
    See \cite[Chaper 1.6]{Ba64}.
\end{proof}

\subsection{Whittaker functions and transforms}\label{prelimwhitt}

It is well-known that the standard Whittaker function for $\ovmrG_{2}$ is given in terms of the $K$-Bessel function (see (\ref{def: Kbessfunc})): 
 \begin{align}\label{gl2wh}
 	W_{\mu}(g) \  := \ 2\sqrt{y} K_{\mu}(2\pi y) e(x)
 \end{align}
 for $\mu\in \C$ and  $g=\begin{psmallmatrix}
 		1 & x\\
 		& 1
 	\end{psmallmatrix} \begin{psmallmatrix}
 		y & \\
 		& 1
 	\end{psmallmatrix} \in \mathfrak{h}^{2}$. We have the following Mellin integral formula for $W_{\mu}$:
 \begin{lem}
 	For $\re w> -1/2+|\re \mu|$, we have
 	\begin{align}\label{melwhitgl2}
 		\int_{0}^{\infty} W_{\mu}(y) y^{w} \ d^{\times} y \ = \  \frac{1}{2} \  \Gamma_{\R}(w+1/2+\mu)\Gamma_{\R}(w+1/2-\mu). 
 	\end{align}
 \end{lem}

\begin{proof}
	Standard, see  \cite[eq. (2.5.2)]{Mo97} for instance.   
	\end{proof}

 For the group $\ovmrG_{3}$, we first introduce the power function 
 \begin{align}\label{defi: powerfunc}
 I_{\alpha}(y_{0}, y_{1}) \, = \, I_{\alpha}(\mby) \, :=  \, y_{0}^{1-\alpha_{3}} y_{1}^{1+\alpha_{1}}    
 \end{align}
  for $\mby\in \mathrm{Y}^{+}$ and  $\alpha\in\mathfrak{a}_{\C}$.\footnote{ We made an identification with the dual of the complexification of the standard Cartan subalgebra of $\mathfrak{sl}_{3}(\R)$. } This can be identified as a function defined on $\mathrm{U}_{3}\setminus \mathfrak{h}^{3}$ by the Iwasawa decomposition. The \emph{standard} Whittaker (or \emph{Jacquet--Whittaker}) function for $\ovmrG_{3}$ is defined by
	\begin{align}\label{jac}
 W_{\alpha}^{\pm}(g) \, := \, \prod_{1\le j<k\le 3} \Gamma_{\R}(1+\alpha_{j}-\alpha_{k}) \ \int_{\mrU_{3}}  \   I_{\alpha}(w_{\ell} \mbu g) \, \overline{\psi_{\pm}(\mbu)} \, d\mbu \hspace{20pt} (g \, \in \,  \mrG_{3}).
	\end{align}
The integral of (\ref{jac}) converges absolutely whenever $\re (\alpha_{1}-\alpha_{2})>0$ and $\re (\alpha_{2}-\alpha_{3})>0$, and $W_{\alpha}^{\pm}$ admits a holomorphic continuation to $\mathfrak{a}_{\C}$; see \cite[Chapter 5.5]{Gold}. Several normalizations of the standard Whittaker function are present in the literature. We opt for the one such that the functional equation $W_{\alpha}^{\pm}(g)=W_{\alpha^{w}}^{\pm}(g)$ holds for any $w\in W_{3}$ and $g\in \mrG_{3}$.

Moreover, we have $W_{\alpha}^{+}(\mby)=W_{\alpha}^{-}(\mby)$, and for brevity, we denote this quantity by $W_{\alpha}(\mby)=W_{\alpha}(y_{0}, y_{1})$.  Since $I_{\alpha}$ is an eigenfunction of $\mathcal{D}_{3}$ \footnote{ The ring of invariant differential operators on $\mathfrak{h}^3$.}, so is $W_{\alpha}$. Finally, we have following Mellin--Barnes integral formula (also known as the \textit{Vinogradov--Takhtadzhyan formula}) for $W_{\alpha}(\mby)$:

\begin{lem}\label{vtmellin}
Let $\balph :=\max_{1\le i\le 3} \, |\re \, \alpha_{i}|$.	 For any $\sigma_{0},\, \sigma_{1} > \balph$ and $\mby\in \mathrm{Y}^{+}$,  we have 
		\begin{align}\label{whittaker}
		W_{-\alpha}(\mby) \ =  \ \frac{1}{4} \  \int_{(\sigma_{0})}\int_{(\sigma_{1})} G_{\alpha}(s_{0},s_{1}) y_{0} ^{1-s_{0}} y_{1}^{1-s_{1}} \  \frac{ds_{0}}{2\pi i } \frac{ds_{1}}{2\pi i },
		\end{align}
         where
		\begin{align}\label{vtgamm}
		G_{\alpha}(s_{0},s_{1})  \ := \ \frac{\prod_{i=1}^{3} \Gamma_{\R}\left( s_{0}+\alpha_{i}\right) \Gamma_{\R}\left( s_{1}-\alpha_{i}\right)}{\Gamma_{\R}(s_{0}+s_{1})}. 
		\end{align}
		\end{lem}
		
\begin{proof}
See  \cite[Chapter X]{Bump84}.   
\end{proof}

\begin{rem}
The sign convention of the $\alpha_{i}$'s in (\ref{whittaker}) is consistent with \cite{Bu20}, but is opposite to that of  \cite[eqs. (6.1.4) \& (6.1.5)]{Gold}. It is important to keep track of this in Lemma \ref{stadediff1}, Propositions \ref{Eissimpl}--\ref{secMTcom}, and Section \ref{CFKRS}, where the symmetry must be matched judiciously with the CFKRS conjectures.
\end{rem}

 \begin{cor}\label{gl3whittbd}
Under Convention \ref{con: smallshift}, we have, for any $0< a_{i}< 1-\balph$\, and $ A_{i} \, > \, 0$ \, ($i=1,2$), that
 	\begin{align}\label{whitest3}
 	|W_{-\alpha}(\mby)| \ \ll_{A_{0}, \, A_{1}} \ y_{0}^{a_{0}}(1+y_{0})^{-A_{0}} y_{1}^{a_{1}} (1+y_{1})^{-A_{1}}, \hspace{20pt} y_{0}, \, y_{1} \ > \ 0. 
 	\end{align}
 	
 \end{cor}
 
 \begin{proof}
 	Follows directly from Lemma \ref{vtmellin}, contour shifting, and Stirling's formula.  Sharper and more uniform bounds can be found in \cite[Proposition 1]{Blo13} and \cite[Theorem 1]{BHM20}. 
 	\end{proof}

 We will need the explicit evaluation of the archimedean $\hbox{GL}(3)\times \hbox{GL}(2)$ Rankin--Selberg integral, which is a consequence of the second Barnes Lemma. This calculation also verifies the sign convention of (\ref{whittaker}).

 \begin{lem}\label{stadediff1}
For $\sigma> \balph + |\re \mu|$, we have
 	\begin{equation}\label{eqn sta1}
 \mathcal{Z}_{\infty}(s; \, W_{\mu},\,  W_{-\alpha}) \, := \, 	\int_{0}^{\infty} \int_{0}^{\infty} \,  W_{\mu}(y_{1})W_{-\alpha}(y_{0}, y_{1})  (y_{0}^2 y_{1})^{s-\frac{1}{2}} \  d^{*}\mby \ = \  \frac{1}{4} \,  \prod_{\pm} \prod_{k=1}^{3} \, \Gamma_{\R}\left(s\pm \mu-\alpha_{k}\right). 
 	\end{equation}
 \end{lem}

 \begin{proof}
 This was proved in \cite{Bump88}. To clarify our conventions, we provide a short proof here. Take $\sigma_{1}\in (\balph + |\re \mu|, \, \sigma)$. From Lemma \ref{vtmellin} and Mellin inversion, we have, for $\re s_{0}> \balph$, that
	\begin{align}
		\int_{0}^{\infty} \  (y_{0}y_{1})^{-1}W_{-\alpha}(\mby)\, y_{0}^{s_{0}}  \ d^{\times} y_{0} \ = \   \int_{(\sigma_{1})}  \ \frac{G_{\alpha}(s_{0}, s_{1})}{4} y_{1}^{-s_{1}} \ \frac{ds_{1}}{2\pi i}. \nonumber
	\end{align}
	Plugging this into the double integral of \eqref{eqn sta1} and  interchanging the order of  integration, we have
	\begin{align}
		\mathcal{Z}_{\infty}(s; \, W_{\mu},\,  W_{-\alpha}) 
		\ &=\   \int_{(\sigma_{1})}  \ \frac{G_{\alpha}(2s, s_{1})}{4} \, \int_{0}^{\infty} \ W_{\mu}(y_{1}) y_{1}^{s-\frac{1}{2}-s_{1}} \ d^{\times} y_{1} \ \frac{ds_{1}}{2\pi i}. \label{melopenup}
	\end{align}
	The innermost integral can be computed by (\ref{melwhitgl2}).  Together with  (\ref{vtgamm}), it follows that
	\begin{align}
		\mathcal{Z}_{\infty}(s; \, W_{\mu},\,  W_{-\alpha}) 
		\ = \ & \frac{\pi^{-3s}}{8}\,  \prod_{k=1}^{3}  \ \Gamma\left( s+\frac{\alpha_{k}}{2}\right) \,  \int_{(\sigma_{1})} \,  \frac{\prod_{k=1}^{3} \ \Gamma\left( \frac{s_{1}-\alpha_{k}}{2}\right)\cdot \Gamma\left(\frac{s+\mu-s_{1}}{2}\right) \Gamma\left(\frac{s-\mu-s_{1}}{2}\right)}{\Gamma\left( \frac{s_{1}}{2}+s\right)} \, \frac{ds_{1}}{2\pi i}. \nonumber
	\end{align}
	The desired result immediately follows from a change of variable $s_{1}\to 2 s_{1}$ and   (\ref{secBar}) with the choice of the parameters $	\left( a, b,c ;d; e\right)  :=  \left( -\frac{\alpha_{1}}{2}, \ -\frac{\alpha_{2}}{2}, \ -\frac{\alpha_{3}}{2}; \ \frac{s+\mu}{2}, \ \frac{s-\mu}{2}\right)$.
 \end{proof}


The following pair of integral transforms plays a central role in the analysis of this article.

	\begin{defi}\label{defwhittrans}
	Let  $h: (0, \infty) \rightarrow \C$ and $H: i\R \rightarrow \C$ be  measurable functions with $H(\mu)=H(-\mu)$. Let $W_{\mu}$ be given by (\ref{gl2wh}).  Then the Kontorovich--Lebedev transform of $h$ is defined by
	\begin{equation}\label{eqn whitranseq}
	h^{\#}(\mu) \, := \, \int_{0}^{\infty} h(y) W_{\mu}(y) \ \frac{dy}{y^2},   
	\end{equation}
	whereas its inverse transform is defined by 
	\begin{align}\label{invers}
	H^{\flat}(y) \, =  \, \int_{(0)}  \, H(\mu) W_{\mu}(y) \ d^{W}\mu, \hspace{15pt} \text{where} \hspace{10pt} d^{W}\mu \, := \, |\Gamma(\mu)|^{-2} \, \frac{d\mu}{4\pi i},
	\end{align}
	provided that the integrals above converge absolutely. 	
	\end{defi}

	\begin{lem}\label{inKLconv}
	Suppose $H\in \mathcal{C}_{\eta}$. The integral  (\ref{invers}) that defines $H^{\flat}$ converges absolutely, and
	\begin{align}\label{decbdd}
		 H^{\flat}(y) \ \ll \ (y+1/y)^{\eta} 
	\end{align}
   for any $y>0$. Moreover, the Whittaker--Plancherel formula holds for any $|\re \mu|< 2\eta$: 
    \begin{align}\label{plancherel}
		(H^{\flat})^{\#}(\mu) \ = \ H(\mu).  
		\end{align}
		\end{lem}
		
		\begin{proof}
   See \cite[Lemma 2.10]{Mo97}.
	\end{proof}

    \begin{rem}\label{theotherLW}
    From \cite[Lemma 2.6]{Mo97}, one also has $(h^{\#})^{\flat}(y)=h(y)$ for $h\in C_{c}^{\infty}(0,\infty)$. 
    \end{rem}

		\begin{lem}\label{mellingl2}
		Suppose $H\in \mathcal{C}_{\eta}$ and $h:= H^{\flat}$. On the strip $-1/2< \re w < \eta$, we have
		\begin{align}
			\widetilde{h}(w)  \ = \  \frac{1}{2} \int_{(0)}  \ H(\mu) \Gamma_{\R}(w+1/2+\mu)\Gamma_{\R}(w+1/2-\mu) \ d^{W}\mu. \label{mellin}
		\end{align}
	\end{lem}
	
	\begin{proof}
	Follows directly from (\ref{plancherel}) and  (\ref{melwhitgl2}). 
	\end{proof}


\subsection{Automorphic preliminaries}\label{autformgl3}

\subsubsection{$\mathrm{PGL}(2)$}\label{someprelimgl2}
The invariant differential operator on $\mathfrak{h}^2$ is the hyperbolic Laplacian $\Delta:= -y^2( \partial_{x}^2 +\partial_{y}^2)$. An automorphic form  $\phi:  \mathfrak{h}^2 \rightarrow \C$ of $\Gamma_{2}=\mathrm{SL}_{2}(\Z)$ satisfies  $\Delta\phi = ( 1/4 -\mu^2) \phi$ for some $\mu= \mu(\phi) \in \C$. The Fourier coefficient of $\phi$, denoted by $\mathcal{B}_{\phi}(a)$, is defined by
\begin{equation}\label{eiscuspcoeff}
\mathcal{W}_{a}(y; \phi) \ := \  \int_{\Z\setminus \R} \ \phi\left[ \begin{pmatrix}
1 & x\\
& 1
\end{pmatrix} \begin{pmatrix}
y & \\
& 1
\end{pmatrix}\right] e(-ax) \, dx \ = \  \frac{\mathcal{B}_{\phi}(a)}{\sqrt{|a|}}  W_{\mu(\phi)}(|a|y)
\end{equation}
for any $a\in \Z-\{0 \}$ and $y>0$.

Let $I_{\mu}(y):= y^{\mu+\frac{1}{2}}$, which can be regarded as a function on $\mathrm{U}_{2}\setminus\mathfrak{h}^2$. The Eisenstein series of $\Gamma_{2}$ is defined by 
		\begin{align}\label{gl2eins}
	E(g; \mu) \ := \  \frac{1}{2} \  \sum_{\gamma\in [\Gamma_{2}]} \, I_{\mu}(\gamma g) \hspace{20pt} (g \,\in \, \mrG_{2}).
		\end{align}
      The series (\ref{gl2eins}) converges absolutely for $\re \mu>1/2$ and admits a meromorphic continuation to $\C$; see \cite[Chapter 3.1]{Gold}.   We have $\Delta E(*; \mu) = ( 1/4-\mu^2) E(*; \mu)$, and  the Fourier coefficients of $E(*;\mu)$ are explicitly given by
	\begin{align}\label{eisfournorm}
		\mathcal{B}_{\mu}(a) \ = \ \frac{|a|^{\mu}\sigma_{-2\mu}(|a|)}{\Lambda(1+2\mu)}, 
		\end{align}	  
		where $\Lambda(s):= \pi^{-s/2}\Gamma(s/2)\zeta(s)$ and  $\sigma_{-2\mu}(|a|) \ := \  \sum_{d \mid a} d^{-2\mu}$.

\subsubsection{$\mathrm{PGL}(3)$}
Suppose  $\Phi: \mathfrak{h}^3 \rightarrow \C$ is an automorphic form of $\Gamma_{3}:= \mathrm{SL}_{3}(\Z)$. By the Harish-Chandra isomorphism, there exists  $\alpha= \alpha(\Phi)\in \mathfrak{a}_{\C}$ such that for any $D\in \mathcal{D}_{3}$, we have
\begin{align*}
D\Phi \ = \  \lambda_{\alpha}(D) \Phi \hspace{20pt} \text{ and } \hspace{20pt} 
DI_{\alpha} \ = \ \lambda_{\alpha}(D) I_{\alpha}
\end{align*}
for some $\lambda_{\alpha}(D) \in \C$, where $I_{\alpha}$ is given by (\ref{defi: powerfunc}). The triple $\alpha$ is said to be the spectral parameters of $\Phi$. The spectral parameters of the dual form $\widetilde{\Phi}(g):= \Phi( ^{t}g^{-1})$ are given by $-\alpha$.

\begin{defi}\label{fourcoeff}
	Let $(m_{1}, m_{2})\in \Z^{2}$ and $\Phi:  \Gamma_{3}\setminus \mathfrak{h}^3 \rightarrow \C$ be an automorphic form. The  $(m_{1}, m_{2})$-th \textbf{Fourier--Whittaker period} of $\Phi$ refers to the integral 
	\begin{align}\label{fourcoeff2}
	\mathcal{W}_{(m_{1}, m_{2})}(g; \, \Phi) \ &:= \ \int_{[\mrU_{3}]}\Phi(\mbu g) \overline{\psi_{(m_{1},m_{2})}(\mbu)} \ d\mbu \hspace{60pt} (g\in \mathrm{G}_{3}). 
	\end{align} 
When $(m_{1}, m_{2})\in (\Z-\{0\})^{2}$, the $(m_{1}, m_{2})$-th \textbf{Fourier coefficient} of $\Phi$ refers to the complex number  $\mathcal{B}_{\Phi}(m_{1}, m_{2})$ such that
	\begin{align}
		\mathcal{W}_{(m_{1}, m_{2})}(\mby; \, \Phi) \  = \   \frac{\mathcal{B}_{\Phi}(m_{1}, m_{2})}{|m_{1}m_{2}|} \  W_{\alpha(\Phi)}(|m_{1}|y_{0}, \, |m_{2}|y_{1}) \hspace{20pt} (\mby\, \in \, \mathrm{Y}^{+}).
	 \end{align}
\end{defi}

\begin{lem}\label{genFourlem}
	Let $\Phi: [\ovmrG_{3}]\rightarrow \C$ be a smooth automorphic function. Then \footnote{ Notice that we are using the shorthand (\ref{matrixshort}).}
	\begin{align}\label{nocuspfou}
		\Phi(g) \ = \  \mathcal{W}_{(0,0)}(g; \, \Phi) \  + \ \sum_{n_{2}\neq 0}  \  \mathcal{W}_{(0, n_{2})}(g; \, \Phi) \ + \ \sum_{\gamma\in [\Gamma_{2}]} 
		\ \sum_{n_{1}= 1}^{\infty}  \ \mathcal{W}_{(n_{1},0)}(\gamma g; \, \Phi) \ + \  \Phi^{\mathrm{ND}}(g)
	\end{align}
	for $g\in \mrG_{3}$, where $\Phi^{\mathrm{ND}}$ denotes the non-degenerate part of $\Phi$: 
	\begin{align}\label{defnondeg}
		\Phi^{\mathrm{ND}}(g) \ := \ \sum_{\gamma \in [\Gamma_{2}]} \ \sum_{m_{1}=1}^{\infty} \  \sum_{m_{2}\neq 0} \mathcal{W}_{(m_{1},m_{2})}( \gamma g; \, \Phi).
	\end{align}
\end{lem}

\begin{proof}
	See \cite[Proposition 4.2]{IY15}. 
\end{proof}

The following identity is central to our method.

\begin{prop}\label{incomf}
	For any smooth function $\Phi: [\ovmrG_{3}]\rightarrow \C$, we have, for any $g\in \mrG_{3}$, that 
	\begin{align}\label{keylemform}
	\mathcal{P}_{\psi}(g; \, \Phi) \, := \, 	\int_{[\mrN_{12}]} \Phi(\mbn g) \overline{\psi(\mbn)} \ d\mbn 
		\ = \  \sum_{a_{0}\in \Z} \, \sum_{a_{1}\in \Z} \,  \mathcal{W}_{(a_{1},1)}\big(
		\mbn_{32}(-a_{0})g;\, \Phi\big). 
	\end{align}
\end{prop}

\begin{proof}
	See \cite[Proposition 6.1]{Kw23}. 
\end{proof}

\begin{rem}\label{rem: project}
As explained in \cite[Section 4]{Kw23}, it is equivalent to consider the unipotent period
\begin{align}\label{equivunip}
 \int_{[\mrN_{23}]}  \, \Psi\left[\mbn \left(
	\begin{smallmatrix}
	y	& & \\
		& 1& \\
		&   & 1
	\end{smallmatrix} \right)\right] \overline{\psi(\mbn)} \,  d\mbn
\end{align}
for $\Psi:= \rho(w_{\ell})\Phi$. Recall the following construction used in the integral representation of the $\mathrm{GL}_{3}$ standard $L$-function (see \cite[Chapters 6.5 \& 12.3]{Gold}):
    \begin{align}\label{eqn: unipad}
    (\mathbb{P}_{1}^{3}\Psi)(g) \ := \ \iint_{(\Z\setminus \R)^2} \, \Psi\left[ \begin{psmallmatrix}
        1 &  &  v_{1} \\
          & 1 & v_{2} \\
          &   & 1
    \end{psmallmatrix}\begin{psmallmatrix}
        g & \\
          & 1
    \end{psmallmatrix}\right] \overline{\psi}(v_{2}) \, dv
\end{align}
for any $g\in \mrG_{2}$. Then the corresponding identity for (\ref{equivunip}) is
\begin{align}\label{projiden}
\int_{[\mrN_{23}]}  \, \Psi\left[\mbn \left(
	\begin{smallmatrix}
	y	& & \\
		& 1& \\
		&   & 1
	\end{smallmatrix} \right)\right] \overline{\psi(\mbn)} \,  d\mbn \, = \, \, ( \mathbb{P}_{1}^{3}\Psi)\begin{psmallmatrix}
        y & \\
          & 1
    \end{psmallmatrix}  \ + \  \sum_{\pm} \, \sum_{a=1}^{\infty} \  \,   (\mathbb{P}_{1}^{3}\Psi)\left[\left(\begin{smallmatrix}
	      1 &  \\
		\pm a & 1 
	\end{smallmatrix}\right) \left(\begin{smallmatrix}
	y   &  \\
		& 1
	\end{smallmatrix} \right)\right].   
\end{align}

Due to (\ref{fourcoeff2}) and $(\mathbb{P}_{1}^{3}\Psi)(g) \, = \, \sum_{n\neq 0} \, \mathcal{W}_{(1,n)}(g; \,\Psi)$, observe that upon inserting the second term on the right-hand side of (\ref{projiden}) into our Fourier--Hecke period ((\ref{trivstart}) or (\ref{modperiodcons})) and applying Mellin inversion, one obtains an integral moment of a product of the $\zeta$-function (from the $a$-sum) and the $L$-function of $\Phi$ (from the $n$-sum; see (\ref{gl3stdL})), with the integral transform given in (\ref{gener}). These steps were performed carefully in our previous works \cite{Kw23, Kw25} and will be reviewed in Sections \ref{diagpreoff} and \ref{offdiagEis}. 
 \end{rem}


\subsubsection{Eisenstein series for $\mathrm{PGL}(3)$}\label{Gl3Eiseri}

We require explicit computations of the Fourier--Whittaker periods for the Eisenstein series of $\Gamma_{3}$ in order to package the moment $ \mathfrak{M}_{-\alpha}^{(3)}(s;  H)$ as a period, extract the main terms and recast them according to the CFKRS conjectures, and carry out the analytic continuation argument of Section \ref{proofEismaTHm}.

A standard reference is \cite[Chapter VII]{Bump84}; see especially Theorem 7.2. For normalizations of the degenerate Whittaker functions of $\ovmrG_{3}$,  readers should refer to \cite[Chapter III]{Bump84}. It is also helpful to consult the reformulation by  \cite[Section 4]{Bu18}, which is more streamlined and up-to-date. Readers should beware that \cite{Bu18} uses the \textit{``incomplete''} Whittaker functions, whereas both this article and  \cite{Bump84} use the \textit{``complete''} ones.

\begin{defi}
	 The minimal parabolic Eisenstein series of $\Gamma_{3}$ is defined by
	\begin{align}\label{minparaEis}
	E_{\min}^{(3)}(g; \alpha) \ := \ 	\sum_{\gamma \in [\Gamma_{3}]} \ I_{\alpha}(\gamma g) \hspace{20pt} (g  \,\in \,  \mrG_{3}).
	\end{align}
	\end{defi}

The series (\ref{minparaEis}) converges absolutely on $\big\{\alpha\in \mathfrak{a}_{\C} :\,  \re (\alpha_{1}-\alpha_{2}) > 1, \, \re (\alpha_{2}-\alpha_{3}) > 1\big\}$ (see \cite[Section 4.1]{Bu18}). Its meromorphic continuation to $\mathfrak{a}_{\C}$ and functional equation are established explicitly in \cite[Chapter 8]{Bump84} and \cite{Bu18}. Langlands proved these results in great generality.

We must keep track of the degenerate terms in the Fourier expansion of $E_{\min}^{(3)}(g; \alpha)$  for our applications. We prefer to apply the more compact form of the expansion (Lemma \ref{genFourlem}), rather than the fully explicit form used in  \cite{Li14},  because it better detects: 
\begin{enumerate}
	\item the equivariance of the Fourier--Whittaker \textit{periods} under  unipotent translations;  and
	
	\item the annihilation of undesirable degenerate terms in the Fourier expansion; see Lemma \ref{lem: whywecanreg}.
\end{enumerate}
These two features yield much simpler formulae and a clearer presentation.

Recall that the Weyl group $W_{3}$ of $\hbox{GL}(3)$ consists of the permutation matrices
\begin{align}\label{Weyletl}
	\hspace{8pt} I \, &= \,  \begin{psmallmatrix}
		1 &   & \\
		& 1 & \\
		&    & 1
	\end{psmallmatrix}, \hspace{10pt}  w_{2} \, = \,  - \begin{psmallmatrix}
		& 1  & \\
		1	&  & \\
		&    & 1
	\end{psmallmatrix},  \hspace{10pt}  w_{3} \, = \,  - \begin{psmallmatrix}
		1&   & \\
		&  &1 \\
		&  1  & 
	\end{psmallmatrix},  \hspace{10pt} 
	w_{4} \, = \,  \begin{psmallmatrix}
		& 1  & \\
		& & 1 \\
		1	&    & 
	\end{psmallmatrix}, \hspace{10pt}  
	w_{5} \, = \,  - \begin{psmallmatrix}
		&   & 1\\
		1	&  & \\
		&   1 & 
	\end{psmallmatrix},  \hspace{10pt}  w_{\ell} \, = \,  - \begin{psmallmatrix}
		&   &1 \\
		& 1 &\\
		1	&    & 
	\end{psmallmatrix}. 
\end{align}
 The actions of $w\in W_{3}$ on $\mathfrak{a}_{\C}$ are given by:
\begin{align}\label{weylactLan}
	(\alpha_{1}, \alpha_{2}, \alpha_{3})^{I} \ &= \  	(\alpha_{1}, \alpha_{2}, \alpha_{3}), \hspace{15pt} 	(\alpha_{1}, \alpha_{2}, \alpha_{3})^{w_{2}} \ = \  	(\alpha_{2}, \alpha_{1}, \alpha_{3}), \hspace{15pt}  	(\alpha_{1}, \alpha_{2}, \alpha_{3})^{w_{3}} \ = \  	(\alpha_{1}, \alpha_{3}, \alpha_{2}),  \nonumber \\
	(\alpha_{1}, \alpha_{2}, \alpha_{3})^{w_{4}} \ &= \  	(\alpha_{3}, \alpha_{1}, \alpha_{2}), \hspace{15pt} 	(\alpha_{1}, \alpha_{2}, \alpha_{3})^{w_{5}} \ = \  	(\alpha_{2}, \alpha_{3}, \alpha_{1}), \hspace{15pt}  	(\alpha_{1}, \alpha_{2}, \alpha_{3})^{w_{\ell}} \ = \  	(\alpha_{3}, \alpha_{2}, \alpha_{1}).   
\end{align}
Translating \cite[eqs. (3.10)--(3.15), (3.40) \& (3.45)]{Bump84} into the notations of this article, we have  

\begin{defi}
	The degenerate Whittaker functions of $\ovmrG_{3}$ are defined as follows:
	\begin{align}\label{degWhitBes}
	W_{\alpha,\, w}^{(0,1)}(\mby) \ &:= \  (y_{0}^{2}y_{1})^{(1-\alpha_{3}^{w})/2} \,  \Gamma_{\R}\left(1+\alpha_{2}^{w}-\alpha_{3}^{w}\right)\Gamma_{\R}\left(1+\alpha_{1}^{w}-\alpha_{3}^{w}\right)\,  W_{(\alpha_{1}^{w}-\alpha_{2}^{w})/2}(y_{1}) 
	\end{align}
	and 
	\begin{align}\label{anodegWhi}
W_{\alpha,\, w}^{(1,0)}(\mby)  \ := \ (y_{0} y_{1}^2)^{(1+\alpha_{1}^{w})/2} \, \Gamma_{\R}\left( 1+\alpha_{1}^{w}-\alpha_{2}^{w}\right) \Gamma_{\R}\left(1+ \alpha_{1}^{w}-\alpha_{3}^{w}\right) W_{(\alpha_{2}^{w}-\alpha_{3}^{w})/2}(y_{0}),
	\end{align} 
    where $w\in W_{3}$, $\alpha\in \mathfrak{a}_{\C}$, $\mby\in \mathrm{Y}^+$, and $ W_{(\alpha_{i}^{w}-\alpha_{i+1}^{w})/2}$ ($i=1,2$) are the standard Whittaker functions of $\ovmrG_{2}$.
\end{defi}

It is clear that (\ref{degWhitBes}) and (\ref{anodegWhi}) grow polynomially as  $y_{0}\to \infty$ and $y_{1}\to \infty$ respectively.  It is more convenient to consider the  \textit{complete} minimal parabolic Eisenstein series of $\Gamma_{3}$, which is defined by 
\begin{align}\label{compEis}
	(E_{\min}^{(3)})^{*}( \ * \ ; \ \alpha) \ := \   \prod_{1\le i<j\le 3} \, \Lambda(1+ \alpha_{i}-\alpha_{j})E_{\min}^{(3)}( \ * \ ; \ \alpha). 
\end{align}
The completion factors in (\ref{compEis}) simplify many formulae and ensure that the Hecke combinatorics of the non-degenerate Fourier coefficients coincide in both the cuspidal and Eisenstein cases. They also obviate the need to consider the zero-free region of the $\zeta$-function in our argument.   (Note: the formulae stated in \cite{Li14} and \cite{Bu18} concern the \textit{incomplete} Eisenstein series.)

\textbf{From this point onward, the automorphic form  $\Phi$  refers solely to (\ref{compEis}).} We record two results regarding the explicit evaluations of the degenerate Fourier--Whittaker periods of $\Phi$.   

\begin{lem}\label{expldegterm}
	For any $\mby\in \mathrm{Y}^{+}$, $\alpha\in \mathfrak{a}_{\C}$, and $n_{1}\in \Z- \{0\}$, we have
	\begin{align}\label{butForum}
		\mathcal{W}_{(0,1)}(\mby; \, \Phi) \ =   \  \sum_{w=w_{2}, w_{4}, w_{\ell}} \ \zeta(w, \alpha, \psi_{(0,1)}) W^{(0,1)}_{\alpha,\, w}(\mby)
	\end{align}
    and 
    \begin{align}
		\mathcal{W}_{(n_{1},0)}(\mby; \, \Phi) \ =  \   \sum_{w=w_{3}, w_{5}, w_{\ell}} \ \frac{\zeta(w, \alpha, \psi_{(n_{1},0)})}{|n_{1}|^{1-\alpha_{3}^{w}}}\  W^{(1,0)}_{\alpha, \, w}(n_{1}y_{0},\, y_{1}),
	\end{align}
	where 
	\begin{align}
		\zeta\big(w, \alpha, \psi_{(0,1)}\big) \ := \   \zeta\big( 1+\alpha_{2}^{w}-\alpha_{3}^{w}\big) \zeta\big( 1+\alpha_{1}^{w}-\alpha_{3}^{w}\big),
	\end{align}
     and 
	\begin{align}\label{zetaexpldegterm}
		\zeta\big(w, \alpha, \psi_{(n_{1},0)}\big) \ := \  \zeta\big(1+ \alpha_{1}^{w}-\alpha_{2}^{w}\big)\zeta\big(1+ \alpha_{1}^{w}-\alpha_{3}^{w}\big) \sigma_{\alpha_{2}^{w}-\alpha_{3}^{w}}(|n_{1}|). 
	\end{align}
\end{lem}

\begin{proof}
	See \cite[eqs. (7.7)--(7.8)]{Bump84} or  \cite[eqs. (4.6)--(4.7)]{Bu18}. 
\end{proof}


\subsubsection{Automorphic $L$-functions}\label{autoLfunc}

Suppose $\Phi$ is given by (\ref{compEis}) and $\phi$ is an automorphic form of $\Gamma_{2}$ as in Section \ref{someprelimgl2}. Recall that 
\begin{defi}\label{JPSSstdL}
	For $\sigma  >  3/2$,  the standard $L$-functions of $\phi$ and  $\Phi$ are defined, respectively, by
		\begin{align}
        L(s, \phi)\ := \ \sum_{n=1}^{\infty} \ \frac{\mathcal{B}_{\phi}(n)}{n^{s}} \hspace{15pt} \text{and} \hspace{15pt}
			L(s, \Phi)\ := \ \sum_{n=1}^{\infty} \ \frac{\mathcal{B}_{\Phi}(1,n)}{n^{s}}. \label{gl3stdL}
		\end{align}
	\end{defi}

By \cite[Theorem 10.8.6]{Gold}, we have
\begin{align}\label{gl3euspl}
	L(s, \Phi)  \ = \  \zeta(s+\alpha_{1}) \zeta(s+\alpha_{2})\zeta(s+\alpha_{3}). 
\end{align}
When $\phi$ is an even cusp form, we have the functional equation (see \cite[Proposition 3.13.5]{Gold}):
\begin{align}\label{Heckefe}
    \Lambda(s, \phi) \ = \ \prod_{\pm} \, \Gamma_{\R}(s\pm \mu) L(s, \phi) \ = \  \Lambda(1-s, \phi).   
\end{align}

	\begin{defi}\label{DScuspRS}
		 For $\sigma > 3/2$, the Rankin--Selberg $L$-function of $\Phi$ and $\phi$ is defined by
	\begin{align}\label{rankse}
		L(s, \phi\otimes \Phi) \ &:= \    \sum_{m_{1}=1}^{\infty} \ \sum_{m_{2}= 1}^{\infty}  \frac{\mathcal{B}_{\phi}(m_{2})  \mathcal{B}_{\Phi}(m_{1}, m_{2} )}{ (m_{1}^2 m_{2})^{s}}. 
	\end{align}
\end{defi}

	\begin{lem}\label{ranselmainthm}
	 Suppose $\phi$ is an \emph{even} form of $\Gamma_{2}$. Then for $\sigma > 3/2$, we have
		\begin{align}\label{sameparun}
		  \mathcal{I}\big(s; \, \phi, \, \widetilde{\Phi}^{\mathrm{ND}}\big)  \ = \ 	\frac{1}{2} \, \Lambda(s,\,  \phi\otimes \widetilde{\Phi}),
		\end{align}
	where 
		\begin{align}
			\Lambda(s,\, \phi\otimes \widetilde{\Phi}) \ &:= \ L_{\infty}(s,\, \phi\otimes \widetilde{\Phi})\cdot L(s, \,\phi\otimes \widetilde{\Phi})  \, := \,  \prod_{\pm} \, \prod_{k=1}^{3} \, \Gamma_{\R}(s\pm  \mu- \alpha_{k}) \cdot L(s,\, \phi\otimes \widetilde{\Phi}). 
		\end{align}
	\end{lem}
	
	\begin{proof}
Recall (\ref{defnondeg}) and (\ref{def: Inewnotpairing}). Since $\widetilde{\Phi}^{\mathrm{ND}}$ is left-invariant by $\begin{psmallmatrix}
	\gamma &   \\
	& 1
\end{psmallmatrix}$ for any $\gamma\in \Gamma_{2}$, the pairing in (\ref{sameparun}) is well-defined. The rest of the proof follows \cite[Proposition 5.15]{Kw23} (or \cite[Chapter 12.2]{Gold}). 
\end{proof}

The following is an elementary consequence of Lemma \ref{ranselmainthm}:

	\begin{cor}\label{ranseleis}
		For $\sigma > 3/2$, we have
        \begin{align}\label{cusEis}
				 	\mathcal{I}\big(s; \, \phi,\,  \widetilde{\Phi}^{\mathrm{ND}}\big) \   = \   \frac{1}{2} \  \prod_{i=1}^{3} \, \Lambda(s- \alpha_{i}, \phi),
			\end{align}
            and 
            \begin{align}\label{EIsEis}
					\mathcal{I}\big(s; \, E(*; \mu),\,  \widetilde{\Phi}^{\mathrm{ND}}\big) \   = \   \frac{1}{2} \    \frac{ \prod_{\pm} \prod_{i=1}^{3} \, \Lambda( s\pm  \mu -\alpha_{i} )}{\Lambda(1+2\mu)}. 
			\end{align}
	\end{cor}

	\begin{proof}
	    The standard $L$-functions $L(s, \phi)$ and $L(s, \Phi)$ admit Euler products of the form
\begin{align}
L(s, \,\phi) \ = \  \prod_{p} \, \prod_{j=1}^{2} \, (1-\beta_{\phi, j}(p)p^{-s})^{-1} \hspace{15pt} \text{and} \hspace{15pt}
  L(s, \, \Phi) \ = \  \prod_{p} \, \prod_{k=1}^{3} \, (1- p^{-\alpha_{k}}p^{-s})^{-1}\nonumber
\end{align}
for $\sigma > 3/2$, and we have $\{ \beta_{\phi,1}(p), \beta_{\phi,2}(p) \} = \{ p^{\mu}, p^{-\mu}  \}$  if $\phi= E^{*}(*;\mu)$. The results follow from Cauchy's identity in the form \cite[Proposition 7.4.12]{Gold}, i.e.,
\begin{align}
L(s,\, \phi \otimes \Phi) \ = \   \prod_{p} \ \prod_{j=1}^{2}  \ \prod_{k=1}^{3}  \,\left(1- \beta_{\phi, j}(p)p^{-\alpha_{k}}\,p^{-s}\right)^{-1}.\nonumber
\end{align}
	\end{proof}



\section{The integral transform}\label{Stirl}

Let $h= H^{\flat}$ and $G_{\alpha}$ as in (\ref{invers}) and (\ref{vtgamm}). In \cite[Section 5.3]{Kw25}, we show that the integral transform associated with Theorem \ref{maingl3gl2Eiscase} admits the following simple integral representation: \footnote{ The same transform was defined in \cite{Kw23, Kw25} under the notation $(\mathcal{F}_{\Phi}H)(s_{0}, \, s)$. In fact, it depends only on the spectral parameters $\alpha = (\alpha_{1}, \alpha_{2}, \alpha_{3})$, and hence, we adopt the present notation.} 
\begin{align*}
   (\mathcal{F}_{\alpha}H)(s_{0}, s) \, = \,  \sum_{\pm} \, (\mathcal{F}_{\alpha}^{\pm}H)(s_{0}, s),
\end{align*}
where
\begin{align}\label{gener}
 (\mathcal{F}^{\pm}_{\Phi} H)\left(s_{0},  s \right)   :=   \int_{0}^{\infty} \int_{0}^{\infty} h(y_{1}^{-1}) (y_{0}^2 y_{1}^{-1})^{s-\frac{1}{2}}  \int_{0}^{\infty} W_{\rho(w_{\ell}) \Phi} \left[\begin{psmallmatrix}
	\pm x & & \\
	1 & 1& \\
	   &  & 1
\end{psmallmatrix}\begin{psmallmatrix}
	y_{0}y_{1} & & \\
	      & y_{1}& \\
	   &  & 1
\end{psmallmatrix} \right] x^{s_{0}-1} \, d^{\times} x \, d^{\times} y_{0} \, dy_{1};
\end{align}
see also Remark \ref{rem: project}. Since $\Phi$ is assumed to be spherical, the above can be written as
\begin{align}\label{duaintrafir}
(\mathcal{F}^{\pm}_{\alpha} H)\left(s_{0},  s \right)  \ = \  \int_{0}^{\infty}  \int_{0}^{\infty}   \ h\Big( \frac{y_{1}}{\sqrt{1+y_{0}^2}}\Big) \frac{ y_{0}^{2s-s_{0}}y_{1}^{s-\frac{1}{2}} }{(1+y_{0}^{2})^{\frac{s}{2}-s_{0}+\frac{1}{4}}} \int_{0}^{\infty}  W_{-\alpha}(x,   y_{1})  e(\pm xy_{0})  x^{s_{0}-1}\,  d^{\times} x\,  d^{*}\mby
\end{align}
as in \cite[Section 4B]{Kw23}. The innermost integral over $x$ reflects the non-archimedean structure of the dual moment (\ref{dual4thzeta}), as expected. By Lemma \ref{vtmellin},  it follows that $(\mathcal{F}_{\alpha} H)(s_{0}, s )$ is equal to the limit of \, $\sum_{\pm} \, (\mathcal{F}_{\alpha}^{\pm}H)(s_{0}, \, s; \, \phi)$ as $\phi\to \pi/2-$, where
\begin{align}\label{mainintrans}
	(\mathcal{F}_{\alpha}^{\pm}H)(s_{0}, \, s; \, \phi) \   &:=\  	 \int_{(\sigma_{1})} \int_{(\epsilon_{0})}  \  \widetilde{h}(s-s_{1}-1/2) \, \mathcal{G}_{\alpha}^{\pm}( s_{1}, u; \, s_{0}, s;  \, \phi) \   \frac{du}{2\pi i}   \frac{ds_{1}}{2\pi i }
			\end{align}
	for any $15\le \sigma_{1}\le \eta-1/2$, and
			\begin{align}\label{kernelG}
				\mathcal{G}_{\alpha}^{\pm}\left( s_{1}, u; \ s_{0}, s;  \ \phi \right) \ := \  &  G_{\alpha}(s_{0}-u, s_{1})   \,   (2\pi )^{-u} e^{\pm i\phi u}  \Gamma(u) \, \frac{\Gamma\big( \frac{u+1-2s+s_{1}-s_{0}}{2}\big) \Gamma\big(\frac{2s-s_{0}-u}{2}\big) }{\Gamma\big(\frac{s_{1}+1}{2}-s_{0}\big)}. 
			\end{align}
More explicitly, we have \footnote{ This corrects a minor typo in \cite[eq. (9-13)]{Kw23}, where the factor $\pi^{-s_{0}}$ is missing.}
\begin{align}\label{streamlinebarnes}
				(\mathcal{F}_{\alpha} H)\left(s_{0},\,  s \right) 
				\ \ = \ \  &  \pi^{1/2-s_{0}} \  \int_{(\sigma_{1})}\,   \widetilde{h}\big(s-s_{1}-1/2\big) \pi^{-s_{1}}  \,\frac{\prod_{i=1}^{3} \ \Gamma\big( \frac{s_{1}-\alpha_{i}}{2}\big)  }{ \Gamma\big(\frac{1+s_{1}}{2}-s_{0}\big)} \nonumber\\
				& \hspace{70pt} \cdot  \,\int_{(\epsilon_{0})}\,   \frac{ \Gamma\big( \frac{u}{2}\big) \Gamma\big( \frac{u+1-2s+s_{1}-s_{0}}{2}\big)    \prod_{i=1}^{3}\Gamma\big( \frac{s_{0}+\alpha_{i}-u}{2}\big)  \Gamma\big(\frac{2s-s_{0}-u}{2}\big) }{\Gamma\big( \frac{1-u}{2}\big) \Gamma\big(\frac{s_{0}+s_{1}-u}{2}\big)}  \  \frac{du}{2\pi i}   \ \frac{ds_{1}}{2\pi i }.
			\end{align}
     We write
			\begin{align*}
				s \ = \ \sigma+it, \hspace{15pt}  s_{0}=\sigma_{0}+it_{0},  \hspace{15pt} s_{1}=\sigma_{1}+it_{1},  \hspace{15pt} \text{ and }  \hspace{15pt} u 
				= \ \epsilon_{0}+iv. 
			\end{align*}

			\begin{prop}\label{anconpr}
				Suppose  $H\in \mathcal{C}_{\eta}$ and $T_{0}\ge 1000$. 
				\begin{enumerate}
					\item For any $\phi \in (0, \pi/2]$, the transform $(\mathcal{F}_{\alpha}^{\pm}H)(s_{0}, s; \,\phi)$ is holomorphic on the triangular domain
			\begin{align}\label{newdomain}
					\mathcal{D} \, := \,  \Big\{(\sigma_{0},\, \sigma): \, 	\sigma_{0} \, > \, \big(1+\frac{1}{500}\big)\epsilon_{0}, \hspace{5pt}   \sigma \, < \, 4, \hspace{5pt}  \text{ and } \hspace{5pt}   2\sigma-\sigma_{0} \ > \ \epsilon_{0} \Big\}.  
					\end{align}
					
					\item Whenever $(\sigma_{0}, \,\sigma) \in \mathcal{D}$,  $|t| <T_{0}$, and \, $\phi \in (0, \pi/2)$, the transform $(\mathcal{F}_{\alpha}^{\pm}H)(s_{0}, \, s; \, \phi)$ has exponential decay as  $|t_{0}| \to \infty$, more precisely:
\begin{align*}
    |(\mathcal{F}_{\alpha}^{\pm}H)(s_{0}, \, s; \, \phi) | \, \ll_{T_{0}} \,  \exp\big(-(1/2)(\pi/2-\phi)|t_{0}|\big).
\end{align*}
                    
\item  Whenever $(\sigma_{0}, \, \sigma) \in \mathcal{D}$, $|t| <T_{0}$, and $|t_{0}| \gg_{T_{0}} 1$,  we have:
					\begin{align}\label{refbdd}
						|	(\mathcal{F}_{\alpha}^{\pm}H)(s_{0}, \, s; \, \pi/2) | \ \ll_{T_{0}} \  |t_{0}|^{8+\epsilon_{0}-\eta/2}. 
					\end{align}
				\end{enumerate}
			\end{prop}

\begin{proof}
    See \cite[Propositions 8.1 and 9.1]{Kw23}. \footnote{ The first part of Proposition \ref{anconpr} is obtained by inspecting (\ref{streamlinebarnes}). }
\end{proof}

\begin{rem}
    If one were analyzing the moment $ \mathfrak{M}_{-\alpha}^{(3)}(s;  H)$ with the Kuznetsov formulae, one must treat the ``$J$-Bessel'' (``same-sign'') and the ``$K$-Bessel'' (``opposite-sign'') pieces separately (cf.  \cite[Theorems 2.2 \& 2.4]{Mo97}). By contrast, our period integral approach yields a single off-diagonal piece involving the $\hbox{PGL}(3)$ Whittaker function (see (\ref{secondcc})), and this results in much simpler computations. 
\end{rem}


Observant readers may wonder about the formulation of Theorem \ref{maingl3gl2Eiscase}. Since we symmetrize the spectral side to obtain a clean description of the main terms, one may likewise expect the dual side to involve the \emph{completed} $L$-functions. This is indeed the case. It is implicit in \cite[Theorem 10.6]{Kw23}, and is important for reasons to be explained in Remark \ref{checktransform}. Furthermore, this permits us to rewrite the integral transform symmetrically in terms of the \emph{Gauss hypergeometric function} 
(\ref{2F1define}). Although not needed later in the article, we believe that this observation is of independent interest. 

Let us take $s=1/2$ and $\gamma(x):= \Gamma(x)\Gamma(1/2-x)^{-1}$. We first recall:

\begin{thm}\label{beausym}(\cite[Theorem 10.6]{Kw23})
	Suppose  $\re s_{0}= 1/2$. Then
	\begin{align}\label{transker}
	\left(\mathcal{F}_{\alpha} H\right)(s_{0}) \, := \, 	\left(\mathcal{F}_{\alpha} H\right)(s_{0}, 1/2) 
		\ = \  &  2\pi^{-s_{0}}\,  \gamma\Big(\frac{s_{0}+\alpha_{1}}{2}\Big)  \int_{(0)} \ H(\mu) \prod_{\pm} \,  \Gamma\Big( \frac{1/2\pm\mu-\alpha_{1}}{2}\Big) 	\mathcal{K}_{\alpha}(s_{0}; \,  \mu) \ d^{W}\mu,
		\end{align}
        where 
			\begin{align}\label{ISform}
				 \mathcal{K}_{\alpha}(s_{0}; \,  \mu) \, = \, \int_{-i\infty}^{i\infty}  \  \int_{-i\infty}^{i\infty}  \,   \Gamma\Big(t+\frac{1}{2}\Big)&  \Gamma\Big(t+ \frac{1-\alpha_{1}}{2}\Big)  \Gamma\Big(\frac{\alpha_{2}}{2}+\frac{\alpha_{1}}{4}-z  \Big) \Gamma\Big(\frac{\alpha_{3}}{2}+\frac{\alpha_{1}}{4}-z  \Big)  \prod_{\pm} \,  \Gamma\Big( -\frac{1}{4}\pm \frac{\mu}{2}+ \frac{\alpha_{1}}{4}  +z-t\Big) \nonumber\\
			&  \cdot \frac{\gamma(- \frac{s_{0}}{2}-t) \gamma(z-\frac{\alpha_{1}}{4}+ \frac{s_{0}}{2}) }{ \gamma(-\frac{\alpha_{1}}{4}+z-t)} \     \frac{dz}{2\pi i} \  	\frac{dt}{2\pi i}.
		\end{align}
	The contours follow Barnes' convention and can be taken as vertical lines $\re t = a$, $\re z=b$ satisfying  
	\begin{align}\label{IScont}
		-1/2+ \epsilon_{0} \ <  \ a  \ <  \ -1/4-\epsilon_{0}, \hspace{10pt} -1/4+\epsilon_{0} \ <  \ b \ < \ -\epsilon_{0}, \hspace{10pt} \text{ and } \hspace{10pt} b-a  \ > \ 1/4+\epsilon_{0}.
		\end{align}
	\end{thm}

In \cite{Kw23}, the proof of Theorem \ref{beausym} invokes \cite[Proposition 2 \& Corollary 4]{IS07} and \cite[Lemma 1.3]{Ish19} on Mellin--Barnes integrals, both of which are elementary consequences of the first Barnes lemma:
	\begin{align}
		\int_{-i\infty}^{i\infty} \, \Gamma(w+ \alpha)\Gamma(w+\beta)  \Gamma( \gamma-w)\Gamma(\delta-w) \ \frac{dw}{2\pi i} 
		\ = \ \frac{\Gamma( \alpha+\gamma) \Gamma(\alpha+\delta) \Gamma(\beta+\gamma) \Gamma(\gamma+\delta)}{\Gamma( \alpha+\beta+\gamma+\delta)}. \nonumber
	\end{align}
    The ingredient in \cite{IS07} is a recursion expressing the Whittaker function of $\ovmrG_{3}$ in terms of that of $\ovmrG_{2}$, which is subsequently generalized by \cite{Jac09} and \cite{Hum25}.

\begin{cor}\label{cor: hypergtrans}
Suppose $\re s_{0}=1/2$. \footnote{ Upon establishing (\ref{symmetridual}), analytic continuation allows one to relax this assumption. }  
Then
\begin{align}\label{symmetridual}
	\left(\mathcal{F}_{\alpha} H\right)(s_{0}) 
		\ = \    &4 \, \Gamma_{\R}(1-s_{0})\prod_{i=1}^{3}\, \Gamma_{\R}(s_{0}+\alpha_{i})  \int_{0}^{\infty} \,  \mbF\left(\begin{matrix}
   \frac{1-s_{0}}{2},\; \frac{1-s_{0}-\alpha_{1}}{2}\\
    \frac{1}{2}
  \end{matrix}\Bigm| -x^2\right)\mbF\left(\begin{matrix}
\frac{s_{0}+\alpha_{2}}{2},\; \frac{s_{0}+\alpha_{3}}{2}\\
    \frac{1}{2}
  \end{matrix}\Bigm| -x^2\right)\nonumber\\
        & \hspace{40pt}\cdot \, \int_{(0)} \ H(\mu) \prod_{\pm}\, \prod_{\pm} \,  \Gamma_{\R}(1/2\pm\mu\pm\alpha_{1}) \mbF\left(\begin{matrix}
   \frac{1}{4} +  \frac{\alpha_{1}+\mu}{2},\; \frac{1}{4} +  \frac{\alpha_{1}-\mu}{2}\\
    \frac{1}{2}
  \end{matrix}\Bigm| -x^2\right)\,  \,	d^{W}\mu \  dx. 
		\end{align}    
\end{cor}

\begin{proof}
    By Lemma \ref{PfaBarnes} and Mellin inversion, we have
    \begin{align}
        \prod_{\pm} \,  &\Gamma\Big( -\frac{1}{4}\pm \frac{\mu}{2}+ \frac{\alpha_{1}}{4}  +z-t\Big) \gamma\Big(-\frac{\alpha_{1}}{4}+z-t\Big)^{-1} \nonumber\\
        \, &\hspace{30pt}= \,  \frac{1}{\sqrt{\pi}} \, \prod_{\pm}   \Gamma\Big( \frac{1}{4} +  \frac{\alpha_{1}\pm\mu}{2}\Big) \int_{0}^{\infty} \, \mbF\left(\begin{matrix}
   \frac{1}{4} +  \frac{\alpha_{1}+\mu}{2},\; \frac{1}{4} +  \frac{\alpha_{1}-\mu}{2}\\
    \frac{1}{2}
  \end{matrix}\Bigm| -x^{-1}\right)x^{z-t-1/2-\alpha_{1}/4} \, d^{\times} x.\nonumber
    \end{align}
    Inserting this into (\ref{ISform}) and rearranging the integrals (with $z\to -z$), 
\begin{align}
				 \mathcal{K}_{\alpha}(s_{0}; \,  \mu) \, = \, \frac{1}{\sqrt{\pi}} &\, \prod_{\pm} \,   \Gamma\Big( \frac{1}{4} +  \frac{\alpha_{1}\pm\mu}{2}\Big) \int_{0}^{\infty} \, x^{-1/2-\alpha_{1}/4} \,  \mbF\left(\begin{matrix}
   \frac{1}{4} +  \frac{\alpha_{1}+\mu}{2},\; \frac{1}{4} +  \frac{\alpha_{1}-\mu}{2}\\
    \frac{1}{2}
  \end{matrix}\Bigm| -x^{-1}\right)\nonumber\\
  &\hspace{30pt}\cdot \, \int_{-i\infty}^{i\infty}  \, \Gamma\Big(t+\frac{1}{2}\Big)  \Gamma\Big(t+ \frac{1-\alpha_{1}}{2}\Big) \gamma\Big(- \frac{s_{0}}{2}-t\Big)x^{-t} \, \frac{dt}{2\pi i}\, \nonumber\\
  & \hspace{60pt}\cdot \, \int_{-i\infty}^{i\infty}  \,     \Gamma\Big(z+\frac{\alpha_{2}}{2}+\frac{\alpha_{1}}{4}  \Big) \Gamma\Big(z+\frac{\alpha_{3}}{2}+\frac{\alpha_{1}}{4}  \Big)  \gamma\Big(-\frac{\alpha_{1}}{4}+ \frac{s_{0}}{2}-z\Big) x^{-z} \,   \frac{dz}{2\pi i} \, d^{\times} x.\nonumber
		\end{align}
 We are in a position to apply Lemma \ref{PfaBarnes} again (with $t\to t-s_{0}/2$ and $z\to z-\alpha_{1}/4+s_{0}/2$):
 \begin{align}
				 \mathcal{K}_{\alpha}(s_{0}; \,  \mu) \, = \, \frac{1}{\sqrt{\pi}} &\, \prod_{\pm} \,   \Gamma\Big( \frac{1}{4} +  \frac{\alpha_{1}\pm\mu}{2}\Big) \int_{0}^{\infty} \, x^{-\frac{1}{2}-\frac{\alpha_{1}}{4}} \,  \mbF\left(\begin{matrix}
   \frac{1}{4} +  \frac{\alpha_{1}+\mu}{2},\; \frac{1}{4} +  \frac{\alpha_{1}-\mu}{2}\\
    \frac{1}{2}
  \end{matrix}\Bigm| -x^{-1}\right)\nonumber\\
  &\hspace{30pt}\cdot \, \frac{x^{\frac{s_{0}}{2}}}{\sqrt{\pi}} \, \Gamma\Big(\frac{1-s_{0}}{2}\Big)\Gamma\Big(\frac{1-s_{0}-\alpha_{1}}{2}\Big)\mbF\left(\begin{matrix}
   \frac{1-s_{0}}{2},\; \frac{1-s_{0}-\alpha_{1}}{2}\\
    \frac{1}{2}
  \end{matrix}\Bigm| -x^{-1}\right)\nonumber\\
  & \hspace{60pt}\cdot \, \frac{x^{\frac{\alpha_{1}}{4}-\frac{s_{0}}{2}}}{\sqrt{\pi}}\, \Gamma\Big(\frac{s_{0}+\alpha_{2}}{2}\Big)\Gamma\Big(\frac{s_{0}+\alpha_{3}}{2}\Big)\mbF\left(\begin{matrix}
\frac{s_{0}+\alpha_{2}}{2},\; \frac{s_{0}+\alpha_{3}}{2}\\
    \frac{1}{2}
  \end{matrix}\Bigm| -x^{-1}\right) \, d^{\times} x.\nonumber
		\end{align}
Substituting the last expression into (\ref{transker}) and rewriting with $x\to x^{-2}$ and $\Gamma_{\R}(s):= \pi^{-s/2}\, \Gamma(s/2)$, our desired result follows.
\end{proof}

\begin{rem}\label{checktransform}
Here we check that the integral transform (\ref{symmetridual}) has the expected symmetry. Recall our main result (\ref{dual4thzeta}) with $s=1/2$. It is clear that the sum $(\mathcal{R}_{\alpha}+\mathcal{R}_{-\alpha})(1/2; H)$ is invariant under $\alpha\to -\alpha$. The same is true for the terms \footnote{In fact, they enjoy a larger set of invariance $\alpha\to (\epsilon_{1}\alpha_{1},\epsilon_{2}\alpha_{2}, \epsilon_{3}\alpha_{3})$ with $\epsilon_{i}=\pm$, and the significance will be discussed in Section \ref{CFKOver}.} $\sum_{i=1}^{3} \, \mathcal{M}_{-\alpha}^{i}(1/2; H)$ and $ \mathfrak{M}_{-\alpha}^{(3)}(1/2; H)$; for the latter, this follows from the functional equations (\ref{riemannFE}) and (\ref{Heckefe}). 

Now, applying (\ref{riemannFE}) to each of the four copies of the $\zeta$-function in the dual moment of (\ref{dual4thzeta}),  the integral transform \emph{should} satisfy the functional equation
\begin{align}\label{funceq}
        \frac{(\mathcal{F}_{\alpha} H)(s_{0})}{(\mathcal{F}_{-\alpha} H)(1-s_{0})} \, = \, \frac{ \Gamma_{\R}(1-s_{0})\prod_{i=1}^{3}\, \Gamma_{\R}(s_{0}+\alpha_{i}) }{ \Gamma_{\R}(s_{0})\prod_{i=1}^{3}\, \Gamma_{\R}(1-s_{0}-\alpha_{i}) }.
    \end{align}
We prove (\ref{funceq}) directly as follows. Recall the Euler transformation:
    \begin{align}
        \mbF\left(\begin{matrix}
   a,\; b\\
  c
  \end{matrix}\Bigm| z\right) \, = \,  (1-z)^{c-a-b}\,  \mbF\left(\begin{matrix}
   c-a,\; c-b\\
  c
  \end{matrix}\Bigm| z\right).
    \end{align}
    It follows that
\begin{align}
    \mbF\left(\begin{matrix}
   \frac{1-s_{0}}{2},\; \frac{1-s_{0}-\alpha_{1}}{2}\\
    \frac{1}{2}
  \end{matrix}\Bigm| -x\right) \, &= \, (1+x)^{s_{0}+\frac{\alpha_{1}}{2}-\frac{1}{2}} \, \mbF\left(\begin{matrix}
   \frac{s_{0}}{2},\; \frac{s_{0}+\alpha_{1}}{2}\\
    \frac{1}{2}
  \end{matrix}\Bigm| -x\right), \nonumber\\
  \mbF\left(\begin{matrix}
\frac{s_{0}+\alpha_{2}}{2},\; \frac{s_{0}+\alpha_{3}}{2}\\
    \frac{1}{2}
  \end{matrix}\Bigm| -x\right) \, &= \, (1+x)^{\frac{1}{2}-s_{0}+\frac{\alpha_{1}}{2}}  \mbF\left(\begin{matrix}
\frac{1-s_{0}-\alpha_{2}}{2},\; \frac{1-s_{0}-\alpha_{3}}{2}\\
    \frac{1}{2}
  \end{matrix}\Bigm| -x\right), \nonumber\\
 \mbF\left(\begin{matrix}
   \frac{1}{4} +  \frac{\alpha_{1}+\mu}{2},\; \frac{1}{4} +  \frac{\alpha_{1}-\mu}{2}\\
    \frac{1}{2}
  \end{matrix}\Bigm| -x\right) \, &= \, (1+x)^{-\alpha_{1}}\,  \mbF\left(\begin{matrix}
   \frac{1}{4} - \frac{\alpha_{1}}{2} -\frac{\mu}{2},\;   \frac{1}{4} - \frac{\alpha_{1}}{2} +\frac{\mu}{2}\\
    \frac{1}{2}
  \end{matrix}\Bigm| -x\right),\nonumber
  \end{align}
  where we used $\alpha_{1}+\alpha_{2}+\alpha_{3}=0$ several times. The result follows from the observation that
\begin{align}
    (1+x)^{s_{0}+\frac{\alpha_{1}}{2}-\frac{1}{2}} (1+x)^{\frac{1}{2}-s_{0}+\frac{\alpha_{1}}{2}}  (1+x)^{-\alpha_{1}} \, = \, 1. \nonumber
\end{align}
\end{rem}

The following schematic diagram, which guides the determination of the integral transform, is influenced by \cite{Ne20+, BJN25}. Let $\Pi$ be the automorphic representation of $\ovmrG_{3}$ generated by $\Phi$ and $\widetilde{\Pi}$ be its contragredient. Let $\mathcal{W}(\Pi)$ and $\mathcal{W}(\widetilde{\Pi})$ be their Whittaker models.  Let $\widehat{[\ovmrG_{2}]}$ be the set of isomorphism classes of irreducible unitary generic automorphic representations of $\ovmrG_{2}$ in  $L^{2}([\ovmrG_{2}])$. In our setting, the diagram takes the form:
\[
\begin{tikzcd}[column sep=large, row sep=large]
\mathcal{W}(\widetilde{\Pi}) \arrow[r] \arrow[d, two heads] & \mathcal{W}(\Pi) \arrow[d] \\
\big\{H(\pi): \ \widehat{[\ovmrG_{2}]} \ \rightarrow \  \C\big\} \arrow[r, "\mathcal{F}_{\alpha}", dotted]               &  \big\{h(z): \ \widehat{\R}_{\mathrm{unit}}^{\times}  \ \rightarrow \  \C\big\}
\end{tikzcd}
\]
The left-hand arrow can be realized by the Whittaker transform (\ref{eqn whitranseq}) and the surjectivity of Kirillov model with respect to the embedding $g\mapsto \left(\begin{smallmatrix}
    g & \\
      & 1
\end{smallmatrix}\right)$. In this work, we adopt the more explicit, classical approach via a Poincar\'{e} series; see Section \ref{Poincspecexp}. The top arrow refers to the map $ \widetilde{W} \mapsto \rho(w_{\ell})W$, where $\widetilde{W}(g):= W(w_{\ell}\, ^{t}g^{-1})$. The right-hand arrow is described by a double Mellin transform, with respect to the embedding\,  $(x,y)\in (\R_{>0}^{\times})^2 \mapsto \left(\begin{smallmatrix}
		x & & \\
		y & 1 & \\
		&  & 1
	\end{smallmatrix}\right)$, for functions in $\mathcal{W}(\Pi)$.

 Also, we wish to point out several other forms of the integral transforms obtained by Humphries--Khan  \cite[eq. (3.5)]{HK22+} and  Bir\'{o}  \cite[p. 5]{Bi22+}. Readers should consult their works for the formulae and the relevant regularity assumptions.


\section{Structures of the main terms}\label{CFKRSSOU}

Readers may proceed directly to Section \ref{proofEismaTHm} for the proofs of the main results. This section, however, situates those results in a broader context, explains their motivation, and indicates why methods of period integral offer useful alternative approaches to subjects surrounding the Moment Conjectures.

The authors of  \cite{CFK+05} proposed an elegant heuristic for obtaining the full sets of main terms for very general classes of moments of $L$-functions. It avoids repeated Taylor/Laurent expansions and proliferation of transcendental constants. More importantly, it encodes the underlying combinatorics and symmetries, and shows remarkable agreement with random matrix theory, specifically, with moments of characteristic polynomials. 

We now summarize the conjectures relevant to our work. In Sections \ref{CFKRS} and \ref{CFKRSU}, we show that they follow from a short, elementary manipulation.


\subsection{CFKRS for orthogonal symmetry}\label{CFKOver}

 For the shifted cubic moment $ \mathfrak{M}_{-\alpha}^{(3)}(s;  H)$, 
 the CFKRS conjecture predicts $2^3=8$ main terms in total, which can be classified as follows \cite{CF22+}. Firstly, the diagonal term, or the  ``$0$-swap'' term,  is of the form: 
\begin{align}
	\zeta(1-\alpha_{1}-\alpha_{2})\zeta(1-\alpha_{1}-\alpha_{3})\zeta(1-\alpha_{2}-\alpha_{3}) \times \text{(archimedean part)}.
\end{align}

Next, we alter the sign of exactly one of the components of $(-\alpha_{1}, -\alpha_{2}, -\alpha_{3})$ and we obtain the following three triples: $(\alpha_{1}, -\alpha_{2}, -\alpha_{3})$, $(-\alpha_{1}, \alpha_{2}, -\alpha_{3})$, $(-\alpha_{1}, -\alpha_{2}, \alpha_{3})$. Accordingly, the following three main terms can be written down, up to the associated archimedean factors:
\begin{align*}
	& \zeta(1+\alpha_{1}-\alpha_{2})\zeta(1+\alpha_{1}-\alpha_{3})\zeta(1-\alpha_{2}-\alpha_{3}),\\ 
    &\zeta(1-\alpha_{1}+\alpha_{2})\zeta(1-\alpha_{1}-\alpha_{3})\zeta(1+\alpha_{2}-\alpha_{3}) ,\\  
    &\zeta(1-\alpha_{1}-\alpha_{2})\zeta(1-\alpha_{1}+\alpha_{3})\zeta(1-\alpha_{2}+\alpha_{3}).
\end{align*}
These three terms are known as the ``$1$-swap'' terms.  The three ``$2$-swap'' terms and the one ``$3$-swap'' term can be written down similarly.

Proposition \ref{incomf} pinpoints the sources of the main terms in a simple and unified way:

\begin{prop}\label{MTclass}
	The main terms of $ \mathfrak{M}_{-\alpha}^{(3)}(s;  H)$ are classified by the rightmost expression in  (\ref{keylemform}): the sums over
	\begin{enumerate}
		\item   $a_{0}=0$ and $a_{1}\neq 0$ give the $0$-swap term;
		
		\item   $a_{0}\neq 0$  and $a_{1}=0$  give the three $1$-swap terms;
		
		\item  $a_{0}\neq 0$  and $a_{1}\neq 0$ contain the three $2$-swap terms  and  the $3$-swap term. 
	\end{enumerate}
\end{prop}

Observe that the  ``$4=3+1$'' decomposition adopted in this work fits more naturally with the structures of the main terms for $ \mathfrak{M}_{-\alpha}^{(3)}(s;  H)$ than the conventional  ``$4=2+2$'' counterpart (see \cite{Iv02, Fr20, Mo93, Mo97}, for instance).  

A salient feature in many explicit archimedean computations is the \emph{complete reduction} of the Mellin--Barnes integrals to ratios of products of $\Gamma$-factors, provided the arguments lie in certain ``nice'' configurations.  This phenomenon typically reflects deeper structures of the integrals. In our case, the reductions of the integral transforms $(\mathcal{F}_{\alpha}H)(s_{0}, s)$ and  $J_{w, \alpha}(s; h)$ correspond exactly to the CFKRS conjecture for  $ \mathfrak{M}_{-\alpha}^{(3)}(s;  H)$. This is the content of Theorem \ref{CFKGamma}.

\begin{rem}
Another example of complete reduction of Mellin--Barnes integrals occurs in \cite[Theorem 3.3]{St01}. It corresponds to the exterior-square lifting for $\mathrm{GL}(n)$.
\end{rem}


\subsection{CFKRS for unitary symmetry}\label{CFKUsym}
The CFKRS conjecture for the fourth moment of the Riemann $\zeta$-function is stated in terms of the following Dirichlet series:
\begin{align}\label{diagRSnaive}
	\mathcal{Z}_{A; B}(s) \ := \   	\sum_{\substack{m,n \ge 1 \\m=n }} \ \frac{\tau_{A}(m)\tau_{B}(n)}{(mn)^{\frac{1}{2}+s}},
\end{align}
where  $\tau_{A}(m):= \sum_{d_{1} d_{2}=m} \ d_{1}^{-\alpha_{1}} d_{2}^{-\alpha_{2}}$ if $A= \{ \alpha_{1}, \alpha_{2} \}$, and similarly for $\tau_{B}(n)$ if $B= \{\beta_{1}, \, \beta_{2}\}$.  The series (\ref{diagRSnaive}) comes up naturally as the \emph{recipe} of \cite[Sections 1.7 and 2]{CFK+05} captures only the diagonal. For $\sigma >0$, we have
\begin{align}\label{naiveRS}
	\mathcal{Z}_{A;B}(s) \ \ :=  \ \ \prod_{\alpha \in A} \ \prod_{\beta \in B} \ \zeta(1+2s+\alpha+\beta)\cdot  \widetilde{\mathcal{A}}_{AB}(s),
\end{align}
and  the infinite Euler product   $ \widetilde{\mathcal{A}}_{AB}(s)$ converges absolutely on $\sigma > -\delta$ for some small $\delta>0$ (see \cite[Theorem 2.4.2.1]{CFK+05}). The full set of the main terms for the shifted moment
	\begin{align}
	 \int_{\mathbb{R}} \ \eta(t)\prod_{\alpha\in A} \ \zeta\left(\tfrac12 +it + \alpha\right)  \ \prod_{\beta\in B}  \zeta\left(\tfrac12 -it+ \beta\right)
		\ dt  
	\end{align}
	is conjectured to be 
	\begin{align}\label{recipe6mo}
	   \int_{\mathbb{R}} \,   \eta(t)  \,  \Big\{  \sum_{\ell =0}^{2} \ \sum_{\substack{U\subset A, \ V\subset B\\ |U|=|V|=\ell}} \   \big(\frac{t}{2\pi}\big)^{-U-V}  \mathcal{Z}_{A_{U}; \,B_{V} }(0) \Big\} \ dt,
	\end{align}
	where $A_{U}:= A-U+V^{-}$, $B_{V}:= B-V+U^{-}$, and \ $-U-V:= - \sum_{\alpha\in U} \alpha - \sum_{\beta\in V} \beta$.  The expression inside the braces in (\ref{recipe6mo}) can be written out more explicitly as follows:
\begin{align}\label{expl4thMT}
Z_{\{\alpha_{1}, \alpha_{2} \}; \left\{ \beta_{1}, \beta_{2}\right\}}(0)  \ &+\  Z_{\{-\beta_{1}, \alpha_{2} \}; \left\{ -\alpha_{1}, \beta_{2}\right\}}(0) \left(\frac{t}{2\pi}\right)^{-\alpha_{1}-\beta_{1}} \ +\ Z_{\{-\beta_{2}, \alpha_{2} \}; \left\{ \beta_{1}, -\alpha_{1}\right\}}(0) \left(\frac{t}{2\pi}\right)^{-\alpha_{1}-\beta_{2}}\nonumber\\
  \ &+\ Z_{\{\alpha_{1}, -\beta_{1} \}; \left\{ -\alpha_{2}, \beta_{2}\right\}}(0) \left(\frac{t}{2\pi}\right)^{-\alpha_{2}-\beta_{1}} \ +\ Z_{\{\alpha_{1}, -\beta_{2} \}; \left\{ \beta_{1}, -\alpha_{2}\right\}}(0) \left(\frac{t}{2\pi}\right)^{-\alpha_{2}-\beta_{2}} \nonumber\\
  \ &+\ Z_{\{-\beta_{1}, -\beta_{2} \}; \left\{ -\alpha_{1},  -\alpha_{2}\right\}}(0) \left(\frac{t}{2\pi}\right)^{-\alpha_{1}-\alpha_{2}-\beta_{1}-\beta_{2}}.  
\end{align}
The first and the last term of (\ref{expl4thMT}) are known as the ``$0$-swap'' and ``$2$-swap'' terms respectively, whereas the middle four terms of (\ref{expl4thMT}) are the ``$1$-swap'' terms. These six terms are the ``non-oscillatory'' terms that remain in the recipe, out of the total $2^4=16$ terms arising from all possible sign combinations of the shifts. Furthermore,  the Ramanujan identity provides a nice formula when $|A|=|B|=2$:
\begin{align}
    \widetilde{\mathcal{A}}_{AB}(s) \,  = \,      \zeta\big( 2+ 2s+ \sum_{\alpha\in A} \ \alpha + \sum_{\beta\in B} \ \beta\big)^{-1}.
\end{align}

\begin{rem}
The formulation (\ref{recipe6mo}) generalizes to any finite sets  $A$ and $B$ of shifts with $|A|=|B|$, where $\tau_{A}$, $\tau_{B}$ are defined similarly, and the upper limit  ``$2$'' of the $\ell$-sum is replaced by $|A|$ $(=|B|)$. However, the corresponding Euler product  $ \widetilde{\mathcal{A}}_{AB}(s)$  does not admit a simple closed form in general (\cite{CIS12}).  
\end{rem}



\section{Proof of Theorems  \ref{maingl3gl2Eiscase} and \ref{CFKGamma}}\label{proofEismaTHm}


\subsection{Spectral expansion and Eisenstein contribution}\label{Poincspecexp}
	\begin{defi}\label{poindef}
		Let $h\in C^{\infty}(0,\infty)$ and $\psi(x):= e(x)$, which are regarded as functions on $\ovmrG_{2}$ via
        \begin{align*}
            h(g) \, = \, h(y) \hspace{15pt} \text{and} \hspace{15pt}  \psi(g) \, = \, \psi(x) \hspace{15pt} \text{if } \hspace{15pt}  g \, = \, \begin{psmallmatrix}
                1 & x\\
                  & 1
            \end{psmallmatrix}\begin{psmallmatrix}
                y & \\
                  & 1
            \end{psmallmatrix} k. 
        \end{align*}
        The associated Poincar\'e series of $\Gamma_{2}$   is defined by
		\begin{equation}\label{defpoin}
		P(g;\, h) \ := \   \sum_{\gamma\in [\Gamma_{2}]} (h\psi)(\gamma g) \hspace{20pt} (g \, \in \,  \mrG_{2}),
		\end{equation}
	 provided it converges absolutely. 
	\end{defi}
	
If the bound $h(y) \,\ll \,  y^{1+\epsilon}(1+y)^{-1/2-2\epsilon}$ is satisfied for any $y>0$, then the Poincar\'e series $P(g;\, h)$ converges absolutely and uniformly on every Siegel set, and is an $L^2$-function. Recall that we take $h:= H^{\flat}$ with $H\in \mathcal{C}_{\eta}$ and $\eta >40$ in this article; cf. (\ref{invers}) and Assumption \ref{regassuCeta}. It follows from (\ref{decbdd}) that the desired bound for $h$ holds.


	\begin{lem}\label{comspec}
Let $ \mathfrak{M}_{-\alpha}^{(3)}(s;  H)$ be defined in (\ref{def: speccubicmom}).	For $\sigma > 3/2$ and $H\in \mathcal{C}_{\eta}$, we have
\begin{align}\label{shiftcub}
 2 \cdot  \mathcal{I}\big(s; \, P(\,*\, ; h), \, \widetilde{\Phi}^{\mathrm{ND}}\big) \ = \  \mathfrak{M}_{-\alpha}^{(3)}(s;  H).
\end{align}
	\end{lem}

	\begin{proof}
This follows from a modification of \cite[Proposition 5.25]{Kw23} and Corollary \ref{ranseleis}.
		\end{proof}

The cuspidal part of $ \mathfrak{M}_{-\alpha}^{(3)}(s;  H)$ evidently admits an entire continuation.  The Eisenstein part, initially defined on $\sigma> 1+\epsilon_{0}/2$, also admits a continuation, though this is less immediate.  Define
\begin{align}\label{contana}
	\mathcal{C}_{\alpha}(s; \, \kappa; \, H) \ \ := \ \  \int_{(\kappa)} \   H(\mu)  \, 
	\frac{  \prod_{i=1}^{3} \, \prod_{\pm} \ \Lambda( s \pm  \mu -\alpha_{i} )}{\Lambda(1+2\mu) \Lambda(1-2\mu)} \,  \frac{d\mu}{2\pi i },
\end{align}
and $ \mathcal{R}_{\alpha}(s; H)$ as in (\ref{4thmoMT2}). Since $H\in \mathcal{C}_{\eta}$,  the integrand of  (\ref{contana})  is holomorphic on $|\re \mu|<2\eta$, except for the poles at
\begin{align}\label{poles}
	\mu \ \in  \ \big\{-s +   \alpha_{i},  \ 1  -  s  +   \alpha_{i}\big\}_{i=1}^{3}  \hspace{20pt} \text{ and }  \hspace{20pt}  \mu \ \in \  \big\{ s -    \alpha_{i}, \ s-1 -\alpha_{i}\big\}_{i=1}^{3},
\end{align}
as well as the zeros of $\Lambda(1\pm 2\mu)$.

\begin{lem}\label{contrcon}
	The function $s\mapsto\mathcal{C}_{\alpha}(s; 0; H)$ admits a meromorphic continuation to $\sigma> 1-\eta+\epsilon_{0}$. On the vertical strip $\epsilon_{0} < \sigma<1-\epsilon_{0}$, the continuation is given by $	\mathcal{R}_{\alpha}(1-s; H)+ \mathcal{C}_{\alpha}(s; 0; H)$.
\end{lem}

\begin{proof}
	This is an adaptation of \cite[p. 118]{Mo97} and we sketch the argument for our setting. Firstly,  suppose $1+\epsilon_{0}/2< \sigma< 1+\epsilon_{0}$.  Since $H\in \mathcal{C}_{\eta}$, we may shift the line of integration to $\re \mu =\eta> 40$, picking up the second set of poles in (\ref{poles}) together with the zeros of $\Lambda(1-2\mu)$.  We have
  \begin{align}\label{shifrigh}
  		\mathcal{C}_{\alpha}(s; \, 0; \, H)  \, = \,  -   \sum_{i=1}^{3} \,  (\Res_{\mu=s-1-\alpha_{i}}  \, + \,   \Res_{\mu=s-\alpha_{i}})  \, - \,  \sum_{\rho} \ \Res_{\mu= \frac{1-\rho}{2}}    \ \   +  \ \     	\mathcal{C}_{\alpha}(s; \, \eta; \, H),
  \end{align}
where $\rho$ runs over  all nontrivial zeros of $\zeta(s)$. 

Secondly, observe that $\mathcal{C}_{\alpha}(s; \, \eta; \, H)$ is holomorphic on $1-\eta + \epsilon_{0}  	<   \sigma   <   1+ \epsilon_{0}$, using the holomorphy of $\Lambda(s)$ on $\{\sigma<0\} \cup \{\sigma>1\}$, and the fact that $\Lambda(1\pm 2\mu)\neq 0$ on $\re \mu=\eta$.  Together with the residual terms of (\ref{shifrigh}),  a meromorphic continuation of $(\mathcal{C}_{\alpha}H)(s)$ to $\sigma> 1-\eta+\epsilon_{0}$ is obtained. 

Thirdly, we restrict the continuation obtained to the smaller region  $\epsilon_{0} < \sigma < 1-\epsilon_{0}$. Only the poles $\mu=1-s+\alpha_{i}$, $\mu= s-\alpha_{i}$ and the zeros of $\Lambda(1-2\mu)$ lie between the contour $\re \mu=0$ and $\re \mu= \eta$.  We shift the line of integration of $\mathcal{C}_{\alpha}(s; \, \eta; \, H)$ back to $\re \mu=0$. The residual contributions of $\mu=s-\alpha_{i}$ and $\mu=(1-\rho)/2$ in (\ref{shifrigh}) are now cancelled, leaving
\begin{align}
	 -   \sum_{i=1}^{3} \, \Res_{\mu=s-1-\alpha_{i}}  \, +  \,  \sum_{i=1}^{3} \, \Res_{\mu=1-s+\alpha_{i}} \, +  \,    \mathcal{C}_{\alpha}(s; \, 0; \, H)
\end{align}
as the meromorphic continuation of $\mathcal{C}_{\alpha}(s; 0; H)$   to the domain  $\epsilon_{0} < \sigma < 1-\epsilon_{0}$. A little calculation shows that $\mathcal{R}_{\alpha}(1-s; H) =   -   \sum_{i=1}^{3} \,  \Res_{\mu=s-1-\alpha_{i}}  \, + \,    \sum_{i=1}^{3} \,  \Res_{\mu=1-s+\alpha_{i}}$, and this completes the sketch. 
\end{proof}

The following is immediate from Lemma \ref{comspec} and \ref{contrcon}.
\begin{cor}\label{prop: spectralcont}
    The expression $\mathfrak{M}_{-\alpha}^{(3)}(s;  H)$ admits a meromorphic continuation to $\sigma> 1-\eta+\epsilon_{0}$. On the strip $\epsilon_{0} < \sigma<1-\epsilon_{0}$, the continuation is given by $\mathfrak{M}_{-\alpha}^{(3)}(s;  H) + \mathcal{R}_{\alpha}(1-s; H)$.
\end{cor}


\subsection{Unfolding and rearrangement}\label{unfoldEis}

By unfolding, we have the equality
\begin{align}\label{modperiodcons}
	\mathcal{I}\big(s; \, P(\,*\, ; h), \, \widetilde{\Phi}^{\mathrm{ND}}\big) \ = \  \int_{0}^{\infty}  \int_{0}^{\infty} \,  h(y_{1}) (y_{0}^2y_{1})^{s-\frac{1}{2}}  \mathcal{P}_{\psi^{-1}}\big(\mby; \, \widetilde{\Phi}^{\mathrm{ND}}\big) \ d^{*}\mby,
\end{align}
where the right-hand side converges absolutely on $1+\epsilon_{0}<\sigma< 4$; see \cite[Section 6B]{Kw23}. In what follows, we perform our analysis on suitable vertical strips within the half-plane $\sigma<4$.   

To avoid confusion with computations in the literature (e.g., regarding the sign convention of parameters and normalization of Fourier coefficients), it is better to consider instead
\begin{align}\label{norm: periodwithoutdual}
	\mathcal{J}(s; \, h, \, \Phi) \ := \  \int_{0}^{\infty}  \int_{0}^{\infty} \  h(y_{1}) (y_{0}^2y_{1})^{s-\frac{1}{2}}  \mathcal{P}_{\psi}(\mby; \, \Phi^{\mathrm{ND}}) \ d^{*}\mby 
\end{align}
in the main calculations of Sections \ref{1-swap}--\ref{degt}. As is apparent from \cite[(6.5), (6.7), (6.8)]{Kw23}, we have
\begin{align}\label{warn}
    \mathcal{I}\big(s; \, P(\,*\, ; h), \, \widetilde{\Phi}^{\mathrm{ND}}\big) \, = \, 	\mathcal{J}(s; \, h, \,  \widetilde{\Phi}).
\end{align}

The next lemma explains why our method does not require any elaborate regularization of periods.
\begin{lem}\label{lem: whywecanreg}
For any $\mby\in \mathrm{Y}^{+}$, we have
    \begin{align}
     \mathcal{W}_{(0,1)}(\mby; \, \Phi)\ - \   \sum_{n_{2}\neq 0}  \ \mathcal{P}_{\psi}\big(\mby; \, \mathcal{W}_{(0, n_{2})}(\,*\,; \, \Phi) \big) \, = \, 0.\label{(1,0)}
    \end{align}
\end{lem}

\begin{proof}
	Since $\mathcal{W}_{(0, n_{2})}(\mbn 
		\mby; \, \Phi)  =    \psi_{n_{2}}(\mbn) \mathcal{W}_{(0, n_{2})}(\mby;\, \Phi)$ for $\mbn\in \mrN_{12}$ and $\mby\in \mathrm{Y}^{+}$, the result follows from 
	\begin{align*}
		\sum_{n_{2}\neq 0}  \, \mathcal{P}_{\psi}\big(\mby; \,\mathcal{W}_{(0, n_{2})}(*; \, \Phi) \big)
           \ = \   \sum_{n_{2}\neq 0}  \,  \mathcal{W}_{(0, n_{2})}(\mby;\, \Phi) \,  \int_{0}^{1} e\left( (n_{2}-1) u\right) \, du \ = \   \mathcal{W}_{(0, 1)}(\mby;\, \Phi).  	\end{align*}
\end{proof}

We have the following decompositions for our periods:
\begin{prop}\label{prop: decompo}
For any $\mby\in \mathrm{Y}^{+}$, we have 
    \begin{align}
	\mathcal{P}_{\psi}\big(\mby; \, \Phi^{\mathrm{ND}}\big) \ = \  \mathcal{P}^{\mathrm{reg}}\big(\mby; \, \Phi\big) \ + \   \mathcal{P}_{(0,1)}^{\min}(\mby; \, \Phi) \ - \  \mathcal{P}_{(*,0)}^{\min}(\mby; \, \Phi),
\end{align}
and 
\begin{align}
    \mathcal{J}(s; \, h, \, \Phi) \, = \, \mathcal{J}^{\mathrm{reg}}(s; \, h, \, \Phi) \, + \, \mathcal{J}_{(0,1)}^{\min}(s; \, h, \, \Phi)  \, - \, \mathcal{J}_{(*,0)}^{\min}(s; \, h, \, \Phi), 
\end{align}
where
\begin{align}
  \mathcal{P}^{\mathrm{reg}}\big(\mby; \, \Phi\big) \, &:= \,   \sum_{a_{1}\in \Z-\{0\}} \, \sum_{a_{0}\in \Z}  \  \mathcal{W}_{(a_{1},1)}\big(
	\mbn_{32}(-a_{0}) \mby; \, \Phi\big), 
  \label{usual} \\\
   \mathcal{P}_{(0,1)}^{\min}(\mby; \, \Phi)   \ &:= \ 	\sum_{a_{0}\in \Z-\{0\}}  \  \mathcal{W}_{(0,1)}\left(
		\mbn_{32}(-a_{0}) \mby; \, \Phi\right), \label{eq: mainconstr1swap}\\
        \mathcal{P}_{(*,0)}^{\min}(\mby; \, \Phi) \, &:= \, \sum_{\gamma\in [\Gamma_{2}]} 
	\, \sum_{n_{1}= 1}^{\infty}  \,  \mathcal{P}_{\psi}\big(\mby; \, \gamma\cdot \mathcal{W}_{(n_{1},0)}(*; \, \Phi) \big), \label{(0,1)}   
\end{align}
and $(\gamma\cdot f)(g):= f(\gamma g)$, and for $\dagger\in \{\mathrm{reg}, \, \min\}$\, and\, $\bullet \in \{ \quad  ,\, (0,1),\, (*,0)\}$, we set
\begin{align}
        \mathcal{J}^{\dagger}_{\bullet}(s; \, h, \, \Phi) \, := \, \int_{0}^{\infty}  \int_{0}^{\infty} \, h(y_{1}) (y_{0}^2y_{1})^{s-\frac{1}{2}}  \mathcal{P}^{\dagger}_{\bullet}\big(\mby; \, \Phi\big) \ d^{*}\mby.
\end{align}
\end{prop}

\begin{proof}
    The unipotent period $\mathcal{P}_{\psi}$ clearly annihilates the constant term $\mathcal{W}_{(0,0)}(g; \, \Phi)$. The result  follows from Lemma \ref{genFourlem}, Proposition \ref{incomf}, and Lemma \ref{lem: whywecanreg}.
\end{proof}



\subsection{Degenerate terms---I}\label{1-swap}

In this section, we evaluate $\mathcal{J}_{(0,1)}^{\min}(s; \, h, \, \Phi) $ defined in Proposition \ref{prop: decompo}. The computation produces three terms, which will be identified with terms in the CFKRS conjecture for the cubic moment in Section \ref{CFKRS}. 
  
\begin{prop}\label{1swalem}
Suppose $1/2+\epsilon_{0}< \sigma < 4$. Then  $ \mathcal{J}_{(0,1)}^{\min}(s;  h, \Phi) $ is equal to
	\begin{align}\label{1swapintr}
	2  \cdot \Gamma_{\R}\left(1+\alpha_{2}^{w}-\alpha_{3}^{w}\right)\Gamma_{\R}\left(1+\alpha_{1}^{w}-\alpha_{3}^{w}\right)\sum_{w=w_{2}, w_{4}, w_{\ell}} \ \zeta(w, \alpha, \psi_{(0,1)})   \zeta(2s+\alpha_{1}^{w}+\alpha_{2}^{w}) J_{w, \alpha}(s; h),
	\end{align}
	where 
	\begin{align}\label{1swpint}
		J_{w, \alpha}(s; h) \ := \ 	 \int_{0}^{\infty}  \int_{0}^{\infty} \  h(y_{1})  y_{1}^{s-\frac{\alpha_{3}^{w}}{2}} y_{0}^{2s-\alpha_{3}^{w}} (1+y_{0}^2)^{-\frac{3}{4}(1-\alpha_{3}^{w})} W_{(\alpha_{1}^{w}-\alpha_{2}^{w})/2}\Big(y_{1}\sqrt{1+y_{0}^2}\Big) \, d^{*}\mby.
	\end{align}
\end{prop}

\begin{proof}
	We begin with the matrix identity 
		\begin{align}\label{centralmaid}
		\mbn_{32}(-a_{0}) \mby	\ \equiv \ 
		\mbn_{23}\Big(-\frac{a_{0}y_{0}^2}{1+(a_{0}y_{0})^2}\Big) \mby\Big(\frac{y_{0}}{1+(a_{0}y_{0})^2}, \, y_{1} \sqrt{1+(a_{0}y_{0})^2}\Big) \hspace{15pt} (\bmod\, \mathrm{K}_{3}\cdot \R_{>0}^{\times}).
	\end{align}
	Applying this with (\ref{degWhitBes}), (\ref{butForum}) and (\ref{eq: mainconstr1swap}), we have
	\begin{align}
		 \mathcal{J}_{(0,1)}^{\min}(s; \, h, \, \Phi)  \, = \,   \Gamma_{\R}\left(1+\alpha_{2}^{w}-\alpha_{3}^{w}\right) \  &\Gamma_{\R}\left(1+\alpha_{1}^{w}-\alpha_{3}^{w}\right) \ \  \sum_{w=w_{2}, w_{4}, w_{\ell}} \ \zeta(w, \alpha, \psi_{(0,1)})   \nonumber\\
		& \cdot   \ 	\sum_{a_{0}\neq 0}  \ \int_{0}^{\infty}  \int_{0}^{\infty} \  h(y_{1})(y_{0}^2y_{1})^{s-\frac{1}{2}}   \Big(\frac{y_{0}}{1+(a_{0}y_{0})^2}\Big)^{1-\alpha_{3}^{w}}  \Big(y_{1}\sqrt{1+(a_{0}y_{0})^2}\Big)^{\frac{1-\alpha_{3}^{w}}{2}} \nonumber\\
		&\hspace{80pt} \cdot  \, W_{(\alpha_{1}^{w}-\alpha_{2}^{w})/2}\Big(y_{1}\sqrt{1+(a_{0}y_{0})^2}\Big) \,  d^{*}\mby.
	\end{align}
	Making a change of variables $y_{0}\to |a_{0}|^{-1} y_{0}$, we find that $ \mathcal{J}_{(0,1)}^{\min}(s; \, h, \, \Phi) $ is equal to
	\begin{align*}
2\Gamma_{\R}\left(1+\alpha_{2}^{w}-\alpha_{3}^{w}\right)\Gamma_{\R}\left(1+\alpha_{1}^{w}-\alpha_{3}^{w}\right)  \sum_{w=w_{2}, w_{4}, w_{\ell}} \ \zeta(w, \alpha, \psi_{(0,1)})   \zeta(2s+\alpha_{1}^{w}+\alpha_{2}^{w}) J_{w, \alpha}(s; h)
	\end{align*}
for $  1/2+\epsilon_{0}<\sigma< 4$. This completes the proof.
\end{proof}

\begin{cor}\label{cor: contibreak(0,1)min}
    The function $ \mathcal{J}_{(0,1)}^{\min}(s; \, h, \, \Phi) $ admits a holomorphic continuation to the strip $\epsilon_{0}< \sigma< 4$, except for three simple poles at $s= (1+\alpha_{i})/2$, $i=1,2,3$. 
\end{cor}

\begin{proof}
Recall that $H:= h^{\#} \in \mathcal{C}_{\eta}$ and $\eta> 40$. 
Using the bound 
\begin{align*}
    W_{(\alpha_{1}^{w}-\alpha_{2}^{w})/2}(y)  \, \ll_{A} \, y^{-A}
\end{align*}
with  $A> 2\sigma-3/2+\epsilon_{0}$ and  $\eta-1 -\epsilon_{0}>  |A-\sigma| $, observe that $J_{w, \alpha}(s; h)$ converges absolutely whenever $\epsilon_{0}< \sigma < 4$. With the analytic continuation of the $\zeta$-function, the right-hand side of  (\ref{1swapintr}) serves as a continuation for $ \mathcal{J}_{(0,1)}^{\min}(s; \, h, \, \Phi) $ to $\epsilon_{0}< \sigma < 4$. It is clearly holomorphic on $\epsilon_{0}< \sigma < 4$, except for the simple poles coming from the factor $\zeta(2s+\alpha_{1}^{w}+\alpha_{2}^{w})$ ($w\in \{w_{2}, w_{4}, w_{\ell}\}$). This completes the proof.
\end{proof}

\begin{prop}\label{1swaprop}
	On the region $\epsilon_{0} < \sigma< 1-\epsilon_{0}$, we  have
	\begin{align}\label{1swaparch}
		J_{w, \alpha}(s; h)
		\ = \   \frac{1}{4} \ & \Gamma_{\R}(1+\alpha_{2}^{w}-\alpha_{3}^{w})^{-1}\Gamma_{\R}(1+\alpha_{1}^{w}-\alpha_{3}^{w})^{-1} \nonumber\\
		&\hspace{30pt}\  \cdot \ \int_{(0)} \, H(\mu)\prod_{\pm} \, \Gamma_{\R}(1-s\pm \mu-\alpha_{3}^{w}) \prod_{i=1}^{2}\, \Gamma_{\R}(s\pm \mu+\alpha_{i}^{w}) \, d^{W}\mu.
	\end{align}
\end{prop}

\begin{proof}
	We begin with the following initial domain:
	\begin{align}\label{inirestr}
	\hspace{20pt} 	1/2+\epsilon_{0} \ < \ 	\sigma \ < \ 4,  \  \hspace{15pt}  	\sigma_{u} \ > \ 2\sigma-3/2+ \epsilon_{0},  \hspace{15pt} 40-\epsilon_{0} \ > \  |\sigma_{u}-\sigma+1|.
	\end{align}
    It follows that
	\begin{align}
	J_{w, \alpha}(s; h) \, = \, 	\frac{1}{2} \, \int_{(\sigma_{u})} \ \prod_{\pm} \, \Gamma_{\R}\Big(u+ \frac{1\pm (\alpha_{1}^{w}-\alpha_{2}^{w})}{2}\Big)& \Big( \int_{0}^{\infty} \, h(y_{1}) y_{1}^{s-1-\frac{\alpha_{3}^{w}}{2}-u} \, d^{\times} y_{1}\Big) \nonumber\\
        &\hspace{-25pt}\cdot \, \Big( \int_{0}^{\infty} \, y_{0}^{2s-\alpha_{3}^{w}} (1+y_{0}^2)^{-\frac{3}{4}(1-\alpha_{3}^{w})-\frac{u}{2}} \, d^{\times} y_{0}\Big)  \ \  \frac{du}{2\pi i}.\nonumber
	\end{align}
Indeed, the first two conditions of (\ref{inirestr}) ensure that the $y_{0}$-integral converges absolutely, and the Euler beta integral formula and (\ref{melwhitgl2}) can be applied (with Mellin inversion). The last condition of (\ref{inirestr}) guarantees the absolute convergence of the $y_{1}$-integral. As a result, we have
	\begin{align}\label{almoBarEis}
		J_{w, \alpha}(s; h)\ = \  \frac{1}{4} \ \Gamma\Big(s-\frac{\alpha_{3}^{w}}{2}\Big)  \, \int_{(\sigma_{u})} \  \, \frac{\prod_{\pm} \, \Gamma_{\R}(u+ \frac{1\pm (\alpha_{1}^{w}-\alpha_{2}^{w})}{2})\Gamma(\frac{3-\alpha_{3}^{w}}{4}-s+ \frac{u}{2})}{\Gamma( \frac{3}{4}(1-\alpha_{3}^{w})+\frac{u}{2})}
		\,	\widetilde{h}\Big(s-1-\frac{\alpha_{3}^{w}}{2}-u\Big)  \ \frac{du}{2\pi i}.
	\end{align}

	
	Upon restricting to the strip  $1/2+\epsilon_{0} < \sigma< 1-2\epsilon_{0}$,  we pick $\sigma_{u} 
    \in (2\sigma-3/2 + \epsilon_{0}, \,  \sigma- 1/2-\epsilon_{0})$. The contour  $\re u=\sigma_{u}$ satisfies the Barnes convention, and the condition for (\ref{mellin}) is satisfied. We have
	\begin{align}
	 	\widetilde{h}\Big(s-1-\frac{\alpha_{3}^{w}}{2}-u\Big) \ = \  \frac{1}{2}\,  \int_{(0)}  \ H(\mu)\, \prod_{\pm} \, \Gamma_{\R}\Big(s-\frac{1+\alpha_{3}^{w}}{2}-u\pm\mu\Big)  \ d^{W}\mu, \nonumber
	\end{align} 
	 and hence, $J_{w, \alpha}(s; h) $ is equal to
	\begin{align}
	\frac{\pi^{-(s-\frac{\alpha_{3}^{w}}{2})}}{8} \  \Gamma\Big(s-\frac{\alpha_{3}^{w}}{2}\Big)   \int_{(0)} \, H(\mu) \int_{(\sigma_{u})} \  \frac{\Gamma\Big(\frac{3-\alpha_{3}^{w}}{4}-s+ \frac{u}{2}\Big)\prod_{\pm} \, \Gamma\Big(\frac{u}{2}+ \frac{1\pm (\alpha_{1}^{w}-\alpha_{2}^{w})}{4}\Big)\Gamma\Big(\frac{s-\frac{1+\alpha_{3}^{w}}{2}-u\pm \mu}{2}\Big)  }{\Gamma( \frac{3}{4}(1-\alpha_{3}^{w})+\frac{u}{2})}  \ \frac{du}{2\pi i} \, d^{W}\mu.\nonumber
	\end{align}
    
We make a change of variables $u \to 2u$, and take 
	\begin{align}
		(a,b,c; d,e ) \ := \  \Big( \frac{1+(\alpha_{1}^{w}-\alpha_{2}^{w})}{4}, \  \frac{1-(\alpha_{1}^{w}-\alpha_{2}^{w})}{4}, \  \frac{3-\alpha_{3}^{w}}{4}-s; \ \frac{s+\mu}{2}-\frac{1+\alpha_{3}^{w}}{4}, \ \frac{s-\mu}{2}-\frac{1+\alpha_{3}^{w}}{4} \Big) \nonumber
	\end{align}
    and verify that  
	\begin{align*}
		(a+ b)+c + (d +e) \ \  = \  \   \frac{1}{2} \ + \  \Big(\frac{3-\alpha_{3}^{w}}{4} - s \Big) \ + \ \Big(s -\frac{1+\alpha_{3}^{w}}{2}\Big) \ \ = \ \   \frac{3}{4}\ (1-\alpha_{3}^{w}) \ \  (:=) \ \  	f.
	\end{align*}
    We apply Lemma \ref{secBarn} to the $u$-integral. Observe that a pair of factors $\Gamma(s-\alpha_{3}^{w}/2)$ cancels, and
    (\ref{1swaparch}) follows from $\alpha_{1}^{w}+\alpha_{2}^{w}+\alpha_{3}^{w}=0$ and conversion with $\Gamma_{\R}(s):= \pi^{-s/2}\Gamma(s/2)$. The validity of (\ref{1swaparch}) on the larger domain $\epsilon_{0}<\sigma<1-\epsilon_{0}$ follows from analytic continuation.  This completes the proof.
\end{proof}

\begin{cor}
Let $\mathcal{M}_{-\alpha}^{1}(s; H)$ be defined in (\ref{1swapterm}). We have
\begin{align}\label{eq: simpler1swap}
    \mathcal{J}_{(0,1)}^{\min}(s;  h, \widetilde{\Phi}) \ = \  \frac{1}{2}\,\mathcal{M}_{-\alpha}^{1}(s; H). 
\end{align}    
\end{cor}

\begin{proof}
    Putting Propositions \ref{1swalem} and  \ref{1swaprop} together, we notice that a pair of factors $\Gamma_{\R}(1-\alpha_{2}^{w}+\alpha_{3}^{w})$  (resp. $\Gamma_{\R}(1-\alpha_{1}^{w}+\alpha_{3}^{w})$) cancels. Explicating the Weyl actions of $w_{2}, w_{4}, w_{\ell}$ on the parameters $(\alpha_{1}, \alpha_{2}, \alpha_{3})$ described in (\ref{weylactLan}) and replacing $\Phi \to \widetilde{\Phi}$, the desired result follows. 
\end{proof}



\subsection{Degenerate terms---II}\label{degt}
The main task of this section is to evaluate $\mathcal{J}_{(*,0)}^{\min}(s; \, h, \, \Phi) $ defined in Proposition \ref{prop: decompo}. Interestingly, the $\gamma$-sum in (\ref{(0,1)}), which arises from the $\Gamma_{3}$ Fourier expansion of $\Phi$, will transform into the $\Gamma_{2}$ Eisenstein series, and produce three of the main terms for the fourth moment of the $\zeta$-function. This is also essential for the analytic continuation of (\ref{modperiodcons}).  Such phenomena do not appear in the first two moments of $\mathrm{GL}(2)$ $L$-functions (\cite{Mo92}).

\begin{prop}\label{degprop}
	For $1/2+\epsilon_{0}< \sigma< 4$, we have
    \begin{align}\label{degexpr}
\mathcal{J}_{(*,0)}^{\min}(s; \, h, \, \widetilde{\Phi})\ =  \ \mathcal{R}_{-\alpha}(s; H).
	\end{align}
\end{prop}

\begin{proof}
We have
\begin{align}\label{degterms}
	 \mathcal{J}_{(*,0)}^{\min}(s; \, h, \, \Phi) \ = \   \int_{0}^{\infty}  \int_{0}^{\infty} \  h(y_{1})(y_{0}^2y_{1})^{s-\frac{1}{2}}    \sum_{\gamma\in [\Gamma_{2}]} 
	\ \sum_{n_{1}= 1}^{\infty}  \,  \int_{[\mrN_{12}]} \, \mathcal{W}_{(n_{1},0)}( \gamma\mbn_{12}
	\mby; \, \Phi) \overline{\psi(\mbn_{12})} \ d\mbn_{12} \, d^{*}\mby.
\end{align}
	For any $\gamma\in \Gamma_{2}$, we have the Iwasawa decomposition:
	\begin{align}
		\hspace{15pt} \gamma \begin{pmatrix}
			1 & u \\
			& 1
		\end{pmatrix}
		\begin{pmatrix}
			y_{0}y_{1} & \\
			& y_{0}
		\end{pmatrix} 
		\ \equiv \  y_{0}y_{1}^{1/2} \begin{pmatrix}
			1 & \re \gamma z\\
			& 1
		\end{pmatrix}
		\begin{pmatrix}
			(\im \gamma z)^{1/2} &       \\
			& 	(\im \gamma z)^{-1/2} 
		\end{pmatrix} \hspace{15pt} \left(\bmod\, \mathrm{SO}(2)\right).\nonumber
	\end{align}
	where  $z  :=  u + i y_{1}$.  Using this together with the equivariance of $	\mathcal{W}_{(n_{1}, 0)}$, observe that
\begin{align*}
    \mathcal{W}_{(n_{1},0)}( \gamma\mbn_{12}(u)
	\mby; \, \Phi) \, = \, \mathcal{W}_{(n_{1}, 0)}\big[\big(y_{0} (y_{1}/\im \gamma z)^{1/2} ,  \, \im \gamma z \big); \, \Phi\big]. 
\end{align*}
By this and Lemma \ref{expldegterm}, we have
	\begin{align}
		 \mathcal{J}_{(*,0)}^{\min}(s; \, h, \, \Phi)
		&	\, =\,  \sum_{w=w_{3}, w_{5}, w_{\ell}} \,    \sum_{\gamma\in [\Gamma_{2}]} 
		\, \sum_{n_{1}= 1}^{\infty}  \, \frac{\zeta(w, \alpha, \psi_{(n_{1},0)})}{|n_{1}|^{1-\alpha_{3}^{w}}} \nonumber\\
		& \hspace{30pt} \cdot   \int_{0}^{\infty}  \int_{0}^{\infty}  \int_{0}^{1} \,  h(y_{1}) (y_{0}^2y_{1})^{s-\frac{1}{2}}     W^{(1,0)}_{\alpha,\, w}\big(n_{1}y_{0} (y_{1}/\im \gamma z)^{1/2}, \, \im \gamma z\big) e(-u) \ du\,    d^{*}\mby. \nonumber
	\end{align}
	It follows from the change of variables  $y_{0} \to n_{1}^{-1} (y_{1}/ \im \gamma z)^{-1/2} y_{0}$ that 
	\begin{align}
		 \mathcal{J}_{(*,0)}^{\min}(s; \, h, \, \Phi) \ = \  	\sum_{w=w_{3}, w_{5}, w_{\ell}} \   &
		\ \Big(\sum_{n_{1}= 1}^{\infty}  \ \frac{\zeta(w, \alpha, \psi_{(n_{1},0)})}{|n_{1}|^{1-\alpha_{3}^{w} +(2s-1)}} \Big)  \nonumber\\
		& \hspace{-30pt}\cdot   \sum_{\gamma\in [\Gamma_{2}]}\,  \int_{0}^{\infty}  \int_{0}^{\infty}  \int_{0}^{1} \  h(y_{1})(y_{0}^2 \im \gamma z)^{s-\frac{1}{2}}  W^{(1,0)}_{\alpha,\, w}(y_{0}, \, \im \gamma z) e(-u) \ du\,    d^{*}\mby. \nonumber
	\end{align}
	(Note: $z$ does not depend on $y_{0}$!) Now, substitute (\ref{anodegWhi}) for  $ W^{(1,0)}_{\alpha, \, w}(\cdots)$, we find that
	\begin{align}\label{explic10Wh}
 \mathcal{J}_{(*,0)}^{\min}(s; \, h, \, \Phi)\, = \,	  \sum_{w=w_{3}, w_{5}, w_{\ell}} \  & \Gamma_{\R}\left( 1+\alpha_{1}^{w}-\alpha_{2}^{w}\right) \Gamma_{\R}\left(1+ \alpha_{1}^{w}-\alpha_{3}^{w}\right)
		\  \nonumber\\
		 \ &\cdot \   \Big(\sum_{n_{1}= 1}^{\infty}  \ \frac{\zeta(w, \alpha, \psi_{(n_{1},0)})}{|n_{1}|^{2s-\alpha_{3}^{w} }} \Big) \Big(  \int_{0}^{\infty}  \ W_{(\alpha_{2}^{w}-\alpha_{3}^{w})/2}(y_{0}) y_{0}^{2s-1+ (1+\alpha_{1}^{w})/2} \ d^{\times} y_{0}\Big) \ \nonumber\\
		& \hspace{25pt} \  \cdot \  \ \sum_{\gamma\in [\Gamma_{2}]}  \  \int_{0}^{\infty}  \int_{0}^{1} \  h(y_{1})(\im \gamma z)^{s-\frac{1}{2}+(1+\alpha_{1}^{w})} e(-u) \ du\,   \frac{dy_{1}}{y_{1}^2}. 
	\end{align}
	In line (\ref{explic10Wh}),  we move the $\gamma$-sum inside the double integral, and thus, \footnote{ The constant multiple $2$ comes from the definition of the Eisenstein series! }
	\begin{align}\label{disEis}
	 \mathcal{J}_{(*,0)}^{\min}(s; \, h, \, \Phi) \, = \,	 \ \sum_{w=w_{3}, w_{5}, w_{\ell}} \ &  \Gamma_{\R}\left( 1+\alpha_{1}^{w}-\alpha_{2}^{w}\right) \Gamma_{\R}\left(1+ \alpha_{1}^{w}-\alpha_{3}^{w}\right) \nonumber\\
		& \cdot \ 
		\ \Big(\sum_{n_{1}= 1}^{\infty}  \ \frac{\zeta(w, \alpha, \psi_{(n_{1},0)})}{|n_{1}|^{2s-\alpha_{3}^{w} }} \Big) \Big(  \int_{0}^{\infty}  \ W_{(\alpha_{2}^{w}-\alpha_{3}^{w})/2}(y_{0}) y_{0}^{2s-1+ (1+\alpha_{1}^{w})/2} \ d^{\times} y_{0}\Big)\nonumber\\
		&\hspace{60pt} \ \cdot \  \Big(\, 2\  \int_{0}^{\infty}   \, h(y_{1}) \,  \int_{0}^{1} \, E\left(z; s+ \alpha_{1}^{w} \right)e(-u) \, du\,   \frac{dy_{1}}{y_{1}^2}\Big).
	\end{align}
	
	It remains to evaluate the three expressions in $(\cdots)$ of  (\ref{disEis}). Firstly, from  (\ref{zetaexpldegterm}), we have
	\begin{align}\label{zeta4}
		\sum_{n_{1}= 1}^{\infty}  \ \frac{\zeta(w, \alpha, \psi_{(n_{1},0)})}{|n_{1}|^{2s-\alpha_{3}^{w} }} \ \  &= \ \   \zeta\left(1+ \alpha_{1}^{w}-\alpha_{2}^{w}\right)\zeta\left(1+ \alpha_{1}^{w}-\alpha_{3}^{w}\right) 	\sum_{n_{1}= 1}^{\infty}  \ \frac{\sigma_{\alpha_{2}^{w}-\alpha_{3}^{w}}(|n_{1}|)}{|n_{1}|^{2s-\alpha_{3}^{w} }} \nonumber\\
		\ \ &= \ \   \zeta\left(1+ \alpha_{1}^{w}-\alpha_{2}^{w}\right)\zeta\left(1+ \alpha_{1}^{w}-\alpha_{3}^{w}\right)  \zeta\left(2s-\alpha_{3}^{w}\right)\zeta(2s-\alpha_{2}^{w})
	\end{align}
    for $\sigma > 1/2+\epsilon_{0}$. Secondly, using (\ref{melwhitgl2}) and  the relation $\alpha_{1}^{w}+\alpha_{2}^{w}+\alpha_{3}^{w}=0$, we have
	\begin{align}\label{weygl2trans}
		 \int_{0}^{\infty}  \, W_{(\alpha_{2}^{w}-\alpha_{3}^{w})/2}(y_{0}) y_{0}^{2s-1+ (1+\alpha_{1}^{w})/2} \ \frac{dy_{0}}{y_{0}}
		\  = \  \frac{1}{2}\, \Gamma_{\R}(2s-\alpha_{3}^{w})  \Gamma_{\R}(2s- \alpha_{2}^{w})
	\end{align}
    for $\sigma> \epsilon_{0}$. Thirdly, it follows from (\ref{eisfournorm}) and  \eqref{eqn whitranseq} that
	\begin{align}\label{redisfour}
		2\,  \int_{0}^{\infty}   \, h(y_{1}) \,  \int_{0}^{1} \, E(z; s+ \alpha_{1}^{w} )e(-u) \ du\   \frac{dy_{1}}{y_{1}^2} 
		\  =  \    2\, \frac{H(s+ \alpha_{1}^{w})}{\Lambda(1+2(s+\alpha_{1}^{w}))} \, .
	\end{align}
	Now, from  (\ref{disEis}), (\ref{zeta4}), (\ref{weygl2trans}) and (\ref{redisfour}), it follows that $ \mathcal{J}_{(*,0)}^{\min}(s; \, h, \, \Phi)$ is given by
    \begin{align}
         \sum_{w=w_{3}, w_{5}, w_{\ell}}  \  H(s+ \alpha_{1}^{w})  \, \frac{ \Lambda(1+ \alpha_{1}^{w}-\alpha_{2}^{w})\Lambda(1+ \alpha_{1}^{w}-\alpha_{3}^{w}) \Lambda(2s+\alpha_{1}^{w}+ \alpha_{2}^{w})\Lambda(2s+\alpha_{1}^{w} +\alpha_{3}^{w})}{\Lambda(1+2(s+\alpha_{1}^{w}))}.
    \end{align}
Recall the Weyl actions of $w_{3}, w_{5}, w_{\ell}$ on the parameters $(\alpha_{1}, \alpha_{2}, \alpha_{3})$ in (\ref{weylactLan}) and replace $\Phi \to \widetilde{\Phi}$, the desired result follows. 
\end{proof}

The following result is immediate.
\begin{cor}\label{cor: contJmin*0}
    The function $ \mathcal{J}_{(*,0)}^{\min}(s;  h, \widetilde{\Phi})$ admits a holomorphic continuation to  $\epsilon_{0}< \sigma< 4$, except for the three simple poles at $s= (1-\alpha_{i})/2$ \ ($i=1,2,3$). 
\end{cor}

\subsection{Diagonal and preparation for off-diagonal: $\mathcal{J}^{\mathrm{reg}}(s; \, h, \, \widetilde{\Phi})$}\label{diagpreoff}

\begin{prop}\label{summar}
	On the vertical strip $1+\epsilon_{0}< \sigma < 4$, we have
	\begin{align}
\mathcal{J}(s; \, h, \,  \widetilde{\Phi}) \, \  = \ 	&   \frac{ 1}{2}\, \prod_{i=1}^{3}  \zeta(2s+\alpha_{i}) \      \int_{(0)}  \, H(\mu)\, \prod_{i=1}^{3} \, \prod_{\pm} \Gamma_{\R}(s\pm \mu-\alpha_{i})\ d^{W}\mu \  + \   \mathrm{OD}_{\alpha}(s)\nonumber\\
    \ &\hspace{30pt}+ \   \mathcal{J}_{(0,1)}^{\min}(s; \, h, \, \widetilde{\Phi})   \ - \  \mathcal{J}_{(*,0)}^{\min}(s; \, h, \, \widetilde{\Phi}), \label{eq: periodecom}
	\end{align}
	where
\begin{align}\label{secondcc}
\mathrm{OD}_{\alpha}(s) \ &= \    \lim_{\phi\to \pi/2} \ \frac{1}{4}  \ \int_{(1+2\epsilon_{0})} \   \zeta\left(2s-s_{0}\right) \prod_{i=1}^{3} \, \zeta(s_{0}+\alpha_{i})  \, \sum_{\pm} (\mathcal{F}_{\alpha}^{\pm}H)(s_{0}, \, s; \,  \phi) \ \frac{ds_{0}}{2\pi i}.
\end{align} 
\end{prop}

\begin{proof}
Recall Proposition \ref{prop: decompo} on the decomposition of periods. From \cite[eq. (6.7)]{Kw23}, we have
			\begin{align}
			 \mathcal{J}^{\mathrm{reg}}(s; \, h, \, \widetilde{\Phi}) \, = \, 	2\cdot  L(2s, \Phi) \, \int_{0}^{\infty} \int_{0}^{\infty} h(y_{1})(y_{0}^2 y_{1})^{s-\frac{1}{2}} W_{-\alpha}(y_{0}, y_{1}) \, d^{*}\mby \ \  + \ \ \mathrm{OD}_{\alpha}(s),  \label{firstcc} 
			\end{align}
where the term $\mathrm{OD}_{\alpha}(s)$ denotes the ``off-diagonal'', i.e., terms with  $a_{0}a_{1}\neq 0$ in (\ref{usual}). The first term of (\ref{firstcc}) can be evaluated by (\ref{gl3euspl}), (\ref{invers}) and \eqref{eqn sta1}.  By \cite[eq. (6.8)]{Kw23}, we have 
            \begin{align}
   \mathrm{OD}_{\alpha}(s) \ &:= \ \sum_{a_{0}\neq 0} \sum_{a_{1}\neq 0}  \  \frac{\mathcal{B}_{\Phi}(1,a_{1})}{|a_{0}|^{2s-1}|a_{1}|} \cdot  \int_{0}^{\infty} \int_{0}^{\infty} h(y_{1})\, (y_{0}^2 y_{1})^{s-\frac{1}{2}} \, e\Big( \frac{a_{1}}{a_{0}}\, \frac{y_{0}^2}{1+y_{0}^2}\Big) \nonumber\\
	& \hspace{180pt} \cdot   \,	W_{-\alpha}\Big( \big|\frac{a_{1}}{a_{0}}\big| \, \frac{y_{0}}{1+y_{0}^2},  \ y_{1} \sqrt{1+y_{0}^2} \Big)   \   d^{*}\mby, \label{eq: ODinDSwhitt}
\end{align}
where the expression (\ref{eq: ODinDSwhitt}) converges absolutely whenever $1+\epsilon_{0}< \sigma< 4$ and for $H \in \mathcal{C}_{\eta}$ \ ($\eta>40$). By \cite[Proposition 7.2]{Kw23} and (\ref{gl3euspl}), this can be written as (\ref{secondcc}). This completes the proof.     
\end{proof}


\subsection{Symmetries of the integral transform}\label{sect: 23swap}
The symmetries of $(\mathcal{F}_{\alpha}H)(s_{0}, s)$ play a crucial role in establishing the CFKRS conjecture for the cubic moment and are important for the analytic continuation argument in Section \ref{offdiagEis}.

\begin{prop}\label{Eissimpl}
The function $s\mapsto 	\left(\mathcal{F}_{\alpha}H\right)\left(1-\alpha_{1},s\right) $	 admits a holomorphic continuation to $\epsilon_{0}< \sigma< 4$ except for  a simple pole  $s=(1-\alpha_{1})/2$. Furthermore, we have
\begin{align}\label{ressym}
	 	(\mathcal{F}_{\alpha}H)(1-\alpha_{1}, \, s)  \, = \,  2 \  \frac{\Gamma_{\R}( 2s-1+\alpha_{1})}{\Gamma_{\R}( 2-2s-\alpha_{1})}  \int_{(0)} \, H(\mu)  \, \prod_{\pm} \, \Gamma_{\R}( s\pm \mu-\alpha_{1})\prod_{j=2}^{3}\Gamma_{\R}(1-s\pm \mu+\alpha_{j})    \ d^{W}\mu
\end{align}
on the strip  $\epsilon_{0}< \sigma<1-\epsilon_{0}$. The results for  $\left(\mathcal{F}_{\alpha}H\right)\left(1-\alpha_{i},s\right)$ ($i=2,3$) follow by symmetry. 
\end{prop}

\begin{proof}
	The holomorphy of $s\mapsto 	\left(\mathcal{F}_{\alpha}H\right)\left(1-\alpha_{1},s\right) $ on $\frac{1}{2}+\epsilon_{0}  <  \sigma  < 4$ follows immediately from Proposition \ref{anconpr}.  Next, we consider a smaller domain  $\frac{1}{2}+\epsilon_{0}  <  \sigma  < 1-\epsilon_{0}$. Put $s_{0}=1-\alpha_{1}$. In (\ref{streamlinebarnes}), the factor $\Gamma\left(\frac{1-u}{2}\right)$ in the denominator cancels with  $\Gamma\left( \frac{s_{0}+\alpha_{1}-u}{2}\right)$ in the numerator, and hence,
	\begin{align}\label{applBarn}
	\left(\mathcal{F}_{\alpha}H\right)\left(1-\alpha_{1},s\right)   	\ = \  &  \pi^{\alpha_{1}-1/2} \  \int_{(\eta-1/2)}  \ \widetilde{h}\big(s-s_{1}-\frac{1}{2}\big)  \frac{\pi^{-s_{1}}   \prod_{i=1}^{3} \ \Gamma\left( \frac{s_{1}-\alpha_{i}}{2}\right)  }{ \Gamma\big(\frac{s_{1}-1}{2}+ \alpha_{1}\big)}\nonumber\\
		& \hspace{-10pt}\cdot \, \int_{(\epsilon_{0})}  \ \frac{\Gamma(\frac{u}{2}) \Gamma\big(\frac{s_{1}+\alpha_{1}+u}{2}-s\big) \Gamma\big(\frac{1-\alpha_{1}+\alpha_{2}-u}{2}\big) \Gamma\big(\frac{1-\alpha_{1}+\alpha_{3}-u}{2}\big) \Gamma\big(s+\frac{-1+\alpha_{1}-u}{2}\big)}{\Gamma\big( \frac{1-\alpha_{1}+s_{1}-u}{2}\big)} \ \frac{du}{2\pi i} \ \frac{ds_{1}}{2\pi i}. 
	\end{align}	
    
 We take $(a,b,c; d,e)  :=  \big( \frac{1-\alpha_{1}+\alpha_{2}}{2}, \ \frac{1-\alpha_{1}+\alpha_{3}}{2}, \ s-\frac{1-\alpha_{1}}{2}; \ 0, \ \frac{s_{1}+\alpha_{1}}{2}-s  \big)$, and verify that
	\begin{align*}
		(a+b)+c+d+e \ = \  \Big(1-\alpha_{1}- \frac{\alpha_{1}}{2}\Big) + s-\frac{1-\alpha_{1}}{2}+ \frac{s_{1}+\alpha_{1}}{2}-s \ = \  \frac{s_{1}-\alpha_{1}+1}{2} \ \  (:=f). 
	\end{align*}
    By Lemma \ref{secBarn} and a change of variables $u\to -2u$,	the  $u$-integral of (\ref{applBarn}) is equal to 
	\begin{align}\label{u2swapbarnes}
		2 \cdot \frac{\Gamma\big( \frac{1-\alpha_{1}+\alpha_{2}}{2}\big) \Gamma\big(\frac{1-\alpha_{1}+\alpha_{3}}{2}\big) \Gamma\big( s-\frac{1}{2}+ \frac{\alpha_{1}}{2}\big) \Gamma\big(\frac{s_{1}+1+\alpha_{2}}{2}-s\big) \Gamma\big( \frac{s_{1}+1+\alpha_{3}}{2}-s\big)\Gamma\big( \frac{s_{1}-1}{2}+\alpha_{1}\big)}{\Gamma\big( \frac{s_{1}-\alpha_{2}}{2}\big) \Gamma\big( \frac{s_{1}-\alpha_{3}}{2}\big) \Gamma\big( 1-\alpha_{1}-s +\frac{s_{1}}{2}\big)}. 
	\end{align}

    Now, observe that the factors  $\Gamma( \frac{s_{1}-\alpha_{2}}{2}) \Gamma( \frac{s_{1}-\alpha_{3}}{2}) $ occur in both the denominator of (\ref{u2swapbarnes}) and the numerator of the first line of  (\ref{applBarn}). Making this cancellation, we have
	\begin{align}
		\hspace{-10pt} \left(\mathcal{F}_{\alpha}H\right)\left(1-\alpha_{1},s\right)   	
		\ = \ &  \ \Gamma\Big( \frac{1-\alpha_{1}+\alpha_{2}}{2}\Big) \Gamma\Big(\frac{1-\alpha_{1}+\alpha_{3}}{2}\Big) \Gamma\Big( s-\frac{1}{2}+ \frac{\alpha_{1}}{2}\Big)  \nonumber\\
		& \hspace{-5pt}\cdot\,  2 \pi^{\alpha_{1}-1/2} \ \int_{(\eta-1/2)} \ \widetilde{h}\big(s-s_{1}-\frac{1}{2}\big) \ \frac{\pi^{-s_{1}} \Gamma\big( \frac{s_{1}-\alpha_{1}}{2}\big)\Gamma\big(\frac{s_{1}+1+\alpha_{2}}{2}-s\big) \Gamma\big( \frac{s_{1}+1+\alpha_{3}}{2}-s\big) }{\Gamma\big( 1-\alpha_{1}-s +\frac{s_{1}}{2}\big)} \ \frac{ds_{1}}{2\pi i}. 
	\end{align}
	 Let $\sigma_{1}\in (2\sigma-1+\epsilon_{0},\, \sigma)$. Using  Lemma \ref{mellingl2}, we have
	\begin{align}\label{barnbarn}
\hspace{-10pt}\left(\mathcal{F}_{\alpha}H\right)\left(1-\alpha_{1},s\right)   	
		\ = \ &  \pi^{\alpha_{1}-1/2-s} \  \Gamma\Big( \frac{1-\alpha_{1}+\alpha_{2}}{2}\Big) \Gamma\Big(\frac{1-\alpha_{1}+\alpha_{3}}{2}\Big) \Gamma\Big( s-\frac{1}{2}+ \frac{\alpha_{1}}{2}\Big)  \int_{(0)}  \ H(\mu)  \nonumber\\
		& \hspace{10pt}\cdot  \ \int_{(\sigma_{1})} \ \ \frac{\Gamma\big( \frac{s_{1}-\alpha_{1}}{2}\big)\Gamma\big(\frac{s_{1}+1+\alpha_{2}}{2}-s\big) \Gamma\big( \frac{s_{1}+1+\alpha_{3}}{2}-s\big) \Gamma\big(\frac{s-s_{1}+\mu}{2}\big) \Gamma\big(\frac{s-s_{1}-\mu}{2}\big) }{\Gamma\big( 1-\alpha_{1}-s +\frac{s_{1}}{2}\big)} \ \frac{ds_{1}}{2\pi i} \  d^{W}\mu. 
	\end{align}

In the $s_{1}$-integral of (\ref{barnbarn}), we apply the change of variables $s_{1}\to 2s_{1}$, and Lemma \ref{secBarn} the second time but with $(a,b,c;d, e)  := ( -\frac{\alpha_{1}}{2}, \ \frac{1+\alpha_{2}}{2}-s, \ \frac{1+\alpha_{3}}{2}-s; \  \frac{s+\mu}{2}, \ \frac{s-\mu}{2}).$ The $s_{1}$-integral is now equal to 
	\begin{align}\label{s1intBar2swap}
		2\cdot 	\frac{\Gamma\big( \frac{s+\mu-\alpha_{1}}{2}\big) \Gamma\big( \frac{s-\mu-\alpha_{1}}{2}\big) \Gamma\big( \frac{1+\alpha_{2}+\mu-s}{2}\big)  \Gamma\big( \frac{1+\alpha_{2}-\mu-s}{2}\big) \Gamma\big( \frac{1+\alpha_{3}+\mu-s}{2}\big)  \Gamma\big( \frac{1+\alpha_{3}-\mu-s}{2}\big)}{\Gamma\big(1-s-\frac{\alpha_{1}}{2}\big) \Gamma\big(\frac{1-\alpha_{2}}{2}-\alpha_{1}\big) \Gamma\big(\frac{1-\alpha_{3}}{2}-\alpha_{1}\big)}, 
	\end{align}
	  upon observing that $a+(b+c)+d+e \, = \,  -\frac{\alpha_{1}}{2}  + (  1-\frac{\alpha_{1}}{2}-2s) \, + s\,  = \,   1-\alpha_{1}-s \, (:=f)$.

	The conclusion (\ref{ressym}) follows from the cancellation of $\Gamma(\frac{1-\alpha_{2}}{2}-\alpha_{1}) \Gamma(\frac{1-\alpha_{3}}{2}-\alpha_{1})$  in the denominator of (\ref{s1intBar2swap}) and $ \Gamma( \frac{1-\alpha_{1}+\alpha_{2}}{2}) \Gamma(\frac{1-\alpha_{1}+\alpha_{3}}{2}) $ in the first line of  (\ref{barnbarn}), since  $\alpha_{1}+\alpha_{2}+\alpha_{3}=0$. Finally, it is clear that (\ref{ressym}) is  holomorphic on $\epsilon_{0}<\sigma< 1-\epsilon_{0}$ except at $s=(1-\alpha_{1})/2$.  Combining this with the initial region of holomorphy, the continuation of  $s\mapsto 	\left(\mathcal{F}_{\alpha}H\right)\left(1-\alpha_{1},s\right) $ to $\epsilon_{0}< \sigma< 4$ is now established. 
\end{proof}

\begin{prop}\label{secMTcom}
	The function $s\mapsto 		\left(\mathcal{F}_{\alpha}H\right)\left(2s-1,  s \right) $ admits a holomorphic continuation to $\epsilon_{0}<\sigma <4$ except for three simple poles $s=(1-\alpha_{i})/2$ \, $(i=1,2,3)$.   On $\epsilon_{0} <  \sigma  < 1-\epsilon_{0}$, we have
	\begin{align}\label{secMTcomid}
		\left(\mathcal{F}_{\alpha}H\right)\left(2s-1,  s \right) 
		\ &= \ 2 \, \prod_{i=1}^{3} \ \frac{\Gamma_{\R}(2s-1+ \alpha_{i})}{\Gamma_{\R}(2-2s- \alpha_{i})} \, \int_{(0)} \,  H(\mu)\,    \prod_{i=1}^{3}  \ \prod_{\pm} \,  \Gamma_{\R}\big( 1-s+ \alpha_{i}\pm \mu\big) \ d^{W}\mu. 
	\end{align}
\end{prop}

\begin{proof}
The argument also appears in \cite{Kw23}; for completeness, we include a proof here. Suppose $1/2+\epsilon_{0}  <  \sigma  < 4$ and  $s_{0}=2s-1$. In (\ref{streamlinebarnes}), the factor $\Gamma\left(\frac{1-u}{2}\right)$ in the denominator cancels with $\Gamma\left(s-\frac{s_{0}+u}{2}\right)$ in the numerator. This leads to
		\begin{align}\label{firuseBar}
			\left(\mathcal{F}_{\alpha} H\right)\left(2s-1,  s \right)  \ = \   &\pi^{3/2-2s} \  \int_{(\eta-1/2)}  \ \widetilde{h}\Big(s-s_{1}-\frac{1}{2}\Big)  \frac{\pi^{-s_{1}}   \prod_{i=1}^{3} \ \Gamma\big( \frac{s_{1}-\alpha_{i}}{2}\big)  }{  \Gamma\big(\frac{1+s_{1}}{2}+1-2s\big)} \nonumber\\
			& \hspace{60pt} \ \cdot \  \int_{(\epsilon_{0})}  \   \frac{ \Gamma\left( \frac{u}{2}\right) \Gamma\left( \frac{u+s_{1}}{2}+ 1-2s\right)  \cdot   \prod_{i=1}^{3} \ \Gamma\left(s-\frac{1}{2}+ \frac{\alpha_{i}-u}{2}\right)   }{ \Gamma\left(s-\frac{1}{2}+\frac{s_{1}-u}{2}\right)}  \  \frac{du}{2\pi i} \ \frac{ds_{1}}{2\pi i}. 
		\end{align}
		
	 Make the change of variables $u\to -2u$ and take $	\left(a,b,c\right)=  \left(s-\frac{1}{2}+ \frac{\alpha_{1}}{2}, \ s-\frac{1}{2}+ \frac{\alpha_{2}}{2}, \ s-\frac{1}{2}+ \frac{\alpha_{3}}{2} \right)$ and $(d,e)= ( 0, \ \frac{s_{1}}{2}+1-2s)$. By Lemma \ref{secBarn},  we find that the $u$-integral is equal to 
		\begin{align}
			2 \,  \prod_{i=1}^{3} \  \Gamma\Big(s-\frac{1}{2}+ \frac{\alpha_{i}}{2}\Big) \Gamma\Big(\frac{s_{1}+1+\alpha_{i}}{2}-s\Big)\Gamma\Big( \frac{s_{1}-\alpha_{i}}{2}\Big)^{-1}. \label{3swapuBarn}
		\end{align}
	Observe that the three $\Gamma$-factors in the numerator of the first line of (\ref{firuseBar}) cancel with those in (\ref{3swapuBarn}).  Hence, we have
		\begin{align}
			\left(\mathcal{F}_{\alpha} H\right)\left(2s-1,  s \right)  \ = \   & 2\pi^{3/2-2s} \,  \prod_{i=1}^{3} \, \Gamma\big(s-\frac{1}{2}+ \frac{\alpha_{i}}{2}\big)   \int_{(\eta-1/2)}  \ \widetilde{h}\big(s-s_{1}-\frac{1}{2}\big) \,  \frac{\pi^{-s_{1}} \prod_{i=1}^{3} \Gamma\big(\frac{s_{1}+1+\alpha_{i}}{2}-s\big)   }{  \Gamma\big(\frac{1+s_{1}}{2}+1-2s\big)} \  \frac{ds_{1}}{2\pi i}. \nonumber
		\end{align}
		
		We further restrict to $1/2+\epsilon_{0} < \sigma< 1-\epsilon_{0}$. We shift the line of integration to the left from $\re s_{1}=\eta-1/2$ to $\re s_{1}=\sigma_{1} \in ( 2\sigma-1 +\epsilon_{0},\, \sigma)$. No pole is crossed and we apply Lemma \ref{mellingl2}:
		\begin{align}
			\left(\mathcal{F}_{\alpha} H\right)\left(2s-1,  s \right) 
			\ &= \   \pi^{\frac{3}{2}-3s} \, \prod_{i=1}^{3} \ \Gamma\big(s-\frac{1}{2}+ \frac{\alpha_{i}}{2}\big) \, \int_{(0)} \, H(\mu)      \int_{(\sigma_{1})} \, \frac{\prod_{i=1}^{3} \Gamma\big(\frac{s_{1}+1+\alpha_{i}}{2}-s\big) \,\prod_{\pm} \, \Gamma\big(\frac{s-s_{1}\pm\mu}{2}\big)  }{\Gamma( \frac{1+s_{1}}{2}+1-2s)} \ \frac{ds_{1}}{2\pi i} \ d^{W}\mu. \nonumber
		\end{align}
        We apply the change of variable $s_{1}\to 2s_{1}$, and Lemma \ref{secBarn} the second time, but with  $(a,b, c)$ $=$ $(\frac{1}{2}-s+ \frac{\alpha_{1}}{2}, \ \frac{1}{2}-s+ \frac{\alpha_{2}}{2}, \ \frac{1}{2}-s+ \frac{\alpha_{3}}{2})$ and $	\left( d,e\right) =  \left( \frac{s+\mu}{2}, \  \frac{s-\mu}{2} \right)$. The  result follows as the $s_{1}$-integral becomes 
		\begin{align*}
			2\ \prod_{i=1}^{3}\, 	 \prod_{\pm} \  \Gamma\Big( \frac{1-s+ \alpha_{i}\pm \mu}{2}\Big)\Gamma\Big(1-s- \frac{\alpha_{i}}{2}\Big)^{-1}. 
		\end{align*}
\end{proof}


\subsection{Analytic continuation and polar terms}\label{offdiagEis}

The analytic continuation argument performed  in \cite{Kw23} is robust and carries over to the present case with minor modifications.

\begin{prop}\label{contod}
	The function $\mathrm{OD}_{\alpha}(s)$ admits a meromorphic continuation to the domain $1/4  + \epsilon_{0}\, <  \, \sigma \, < \, 4$.  On the smaller domain  $1/4 +\epsilon_{0}\, < \,  \sigma \, < \,  3/4$, the following equality holds:
	\begin{align}
		4\cdot \mathrm{OD}_{\alpha}(s) \ = \  & \    \ \sum_{i=1}^{3} \ \zeta(2s- 1+\alpha_{i}) \prod_{\substack{ 1\le j\le 3\\ j\neq i}}\zeta(1-\alpha_{i}+\alpha_{j})  	\left(\mathcal{F}_{\alpha}H\right)\left(1-\alpha_{i},  s\right)   \label{2swapss} \\
		&\hspace{40pt} \ + \  \ \prod_{i=1}^{3} \zeta(2s-1+\alpha_{i})   	\left(\mathcal{F}_{\alpha}H\right)\left(2s-1,  s\right)   \label{3swapss}   \\
		& \hspace{80pt} \ + \    \  \int_{(1/2)} \zeta(2s-s_{0}) \prod_{i=1}^{3}\,  \zeta(s_{0}+\alpha_{i})    	(\mathcal{F}_{\alpha}H)(s_{0},  s) \ \frac{ds_{0}}{2\pi i}. \label{fin: dualzeta}
	\end{align}
\end{prop}

\begin{proof}
 Readers are invited to consult  \cite[Section 9]{Kw23} for fuller details. Parallel to  \cite[Section 9A]{Kw23} (\textbf{`Step $1$'}), we shift the line of integration in (\ref{secondcc})   to $\re s_{0}=2\epsilon_{0}$. This time, however, we  pick up the residues of three extra simple poles at  $s_{0}  = 1-\alpha_{i} $ ($i=1,2,3$), which results in   
	\begin{align}\label{pickup}
		4\cdot 	\mathrm{OD}_{\alpha}(s) \ = \  \lim_{\phi \to \pi/2} \ \bigg\{ & \sum_{i=1}^{3} \ \zeta(2s- 1+\alpha_{i}) \prod_{\substack{ 1\le j\le 3\\ j\neq i}}\zeta(1-\alpha_{i}+\alpha_{j}) \sum_{\pm} (\mathcal{F}_{\alpha}^{\pm}H)(1-\alpha_{i}, \, s; \,  \phi) \nonumber\\
		&  \hspace{30pt} +   \int_{(2\epsilon_{0})} \zeta\left(2s-s_{0}\right)  \prod_{i=1}^{3}\,  \zeta(s_{0}+\alpha_{i})      \sum_{\pm} (\mathcal{F}_{\alpha}^{\pm}H)(s_{0}, \, s; \,  \phi) \ \  \frac{ds_{0}}{2\pi i} \,\bigg\}.
	\end{align}
	This serves as a holomorphic continuation to $1/2+\epsilon_{0} < \sigma <4$, except for the three simple poles at $s=1- \alpha_{i}/2$ ($i=1,2,3$).   Parallel to  \cite[Section 9B]{Kw23} (`\textbf{Step $2$}'),  we restrict to the strip $1/2+ \epsilon_{0} <  \sigma  <  3/4$, and shift the line of integration from $\re s_{0}=2\epsilon_{0}$ to $\re s_{0}=1/2$, crossing the simple pole of  \ $\zeta(2s-s_{0})$ with residue equal to  $-1$. The function on the second line of (\ref{pickup}) is holomorphic on the strip $1/4+\epsilon_{0}<\sigma< 3/4$.
    
 Section 9C of \cite{Kw23} (`\textbf{Step 3}') carries over to the present context without change as it merely takes care of the necessary regularity (using Proposition \ref{anconpr} and the imposed assumptions on our class of test functions).  Section 9D of \cite{Kw23}  (`\textbf{Step 4}') concerns the continuation of (\ref{2swapss}) and (\ref{3swapss}).  Here, we instead adopt an explicit approach and the desired conclusion follows from  Proposition \ref{Eissimpl}--\ref{secMTcom}.  This completes the proof. \footnote{ The content of \cite[Section 9E]{Kw23} (`\textbf{Step 5}') is postponed to Corollary \ref{totalcontdual} and Section \ref{CFKRS}.} 
	\end{proof}

\begin{cor}\label{totalcontdual}
    The function $\mathcal{J}(s; \, h, \,  \widetilde{\Phi})$ admits a meromorphic continuation to $1/4  + \epsilon_{0}\, <  \, \sigma \, < \, 4$.
\end{cor}

\begin{proof}
    In Proposition \ref{summar}, every term on the right-hand side of (\ref{eq: periodecom}), except for $\mathrm{OD}_{\alpha}(s)$, admits a meromorphic continuation to the vertical strip $\epsilon_{0} < \sigma < 4$; see Corollary \ref{cor: contibreak(0,1)min} and \ref{cor: contJmin*0}. Now, the desired result follows from Proposition \ref{contod}.
\end{proof}



\subsection{Agreement with CFKRS and completion of Theorem \ref{maingl3gl2Eiscase}}\label{CFKRS}
By Lemma \ref{comspec}, (\ref{modperiodcons}), and (\ref{warn}), we have
\begin{align}\label{abs: specid}
 \mathfrak{M}_{-\alpha}^{(3)}(s;  H) \ = \  2 \cdot \mathcal{J}(s; \, h, \,  \widetilde{\Phi})
\end{align}
on the vertical strip $3/2< \sigma< 4$. By Corollaries \ref{prop: spectralcont} and \ref{totalcontdual}, both sides of (\ref{abs: specid}) admit a meromorphic continuation to the strip $1/4+\epsilon_{0}<\sigma< 4$, and (\ref{abs: specid}) remains valid on the new strip.  Now, we restrict to  $1/4 + \epsilon_{0}  <  \sigma  < 3/4$. Putting Corollary \ref{prop: spectralcont}, (\ref{eq: periodecom}), (\ref{eq: simpler1swap}), (\ref{degexpr}), (\ref{2swapss}), (\ref{3swapss}), (\ref{fin: dualzeta}) together, 
\begin{align}
	& \mathfrak{M}_{-\alpha}^{(3)}(s;  H) \ + \  \mathcal{R}_{-\alpha}(s; H) \ + \ \mathcal{R}_{\alpha}(1-s; H) \\
&\hspace{30pt} 	\ = \    \prod_{i=1}^{3}\, \zeta(2s+\alpha_{i})    \int_{(0)}  \, H(\mu) \, \prod_{i=1}^{3} \, \prod_{\pm} \, \Gamma_{\R}(s\pm \mu-\alpha_{i}) \ d^{W}\mu  \  \ \ + \   \mathcal{M}_{-\alpha}^{1}(s; H) \label{tobem: 01swap} \\
	& \hspace{120pt}\ +\  \frac{1}{2} \ \sum_{i=1}^{3} \ \zeta(2s- 1+\alpha_{i}) \prod_{\substack{ 1\le j\le 3\\ j\neq i}}\zeta(1-\alpha_{i}+\alpha_{j})  \left(\mathcal{F}_{\alpha}H\right)\left(1-\alpha_{i},  s\right) \label{tobem: 2swap}\\
	&\hspace{160pt} \ + \   \frac{1}{2} \, \prod_{i=1}^{3}\, \zeta(2s-1+\alpha_{i})   \left(\mathcal{F}_{\alpha} H\right)\left(2s-1,  s\right)  \label{tobem: 3swap}   \\
	& \hspace{200pt} \ + \   \frac{1}{2}  \  \int_{(1/2)} \zeta\left(2s-s_{0}\right)\  \prod_{i=1}^{3} \   \zeta(s_{0}+\alpha_{i})  \left(\mathcal{F}_{\alpha}H\right)\left(s_{0},  s\right) \ \frac{ds_{0}}{2\pi i}. \label{inter: thmid}
	\end{align}

We complete the proof of Theorem \ref{maingl3gl2Eiscase} by demonstrating that the totality of the terms (\ref{tobem: 01swap}), (\ref{tobem: 2swap}) and (\ref{tobem: 3swap}) agree with the CFKRS predictions for the $\mathrm{SO}(\text{even})$ symmetry described in Section \ref{CFKRSSOU}, i.e., with $\sum_{0\le i\le 3} \, \mathcal{M}_{-\alpha}^{i}(s; H)$ defined in Theorem \ref{maingl3gl2Eiscase}. Indeed, observe that:
\begin{enumerate}
	\item \textbf{($0$ \& $1$-swap).} By $\alpha_{1}+\alpha_{2}+\alpha_{3}=0$, the first term of (\ref{tobem: 01swap}) can be written as $\mathcal{M}_{-\alpha}^{0}(s; H)$ as defined by (\ref{thm:0swapterm}). The second term of (\ref{tobem: 01swap}) is already in the form of the $1$-swap prediction. 

	\item \textbf{($2$-swap).}  By Proposition \ref{Eissimpl}  and the functional equation of the $\zeta$-function of the form
	\begin{align}\label{funczetaflip}
\zeta(2s-1+ \alpha_{i})  \frac{\Gamma_{\R}( 2s-1+\alpha_{i})}{\Gamma_{\R}( 2-2s-\alpha_{i})} \ = \  \zeta(2-2s-\alpha_{i}) \ = \ \zeta(2-2s + \sum_{\substack{1\le j\le 3\\ j\neq i}}\alpha_{j}),
	\end{align}
	it follows that eq. (\ref{tobem: 2swap}) coincides with $\mathcal{M}_{-\alpha}^{2}(s; H)$ as defined by (\ref{2swapterm}). 
    
	\item \textbf{($3$-swap).} Using Proposition \ref{secMTcom} and  (\ref{funczetaflip}) thrice, observe that (\ref{tobem: 3swap}) coincides with $\mathcal{M}_{-\alpha}^{3}(s; H)$ as defined by (\ref{3swapterm}). 
\end{enumerate}
We set $s=1/2$. The matching with the predictions in Section \ref{CFKOver} follows from the Weyl actions (\ref{weylactLan}). The restriction $\alpha_{1}+\alpha_{2}+\alpha_{3}=0$ can be removed by analytic continuation in $\alpha_{i}$'s.



\section{Concluding remarks}

\subsection{Agreement with the unitary CFKRS}\label{CFKRSU}
To align with the set-up of Section \ref{CFKUsym}, we re-label the dual moment of Theorem \ref{maingl3gl2Eiscase} as
\begin{align}\label{duomo}
	\int_{(\frac{1}{2}+\gamma)} \ \zeta(2s-s_{0})\zeta(s_{0}+\nu_{1}) \zeta(s_{0}+\nu_{2}) \zeta(s_{0}+\nu_{3}) \left(\mathcal{F}_{\nu}H\right)\left(s_{0},  s\right) \ \frac{ds_{0}}{2\pi i}.
\end{align}
Notice that the line of integration is shifted by a small quantity $\gamma$.  We put 
\begin{align}
    s \, := \, \frac{1+\delta}{2},
\end{align}
and apply the functional equation (\ref{riemannFE}) of the $\zeta$-function to the factor $ \zeta(s_{0}+\nu_{3})$, the moment (\ref{duomo}) takes the form:
\begin{align}
\frac{1}{2\pi} \ 	\int_{\R} \ \zeta\Big(\frac{1}{2}+\delta-\gamma-it\Big) &\zeta\Big(\frac{1}{2}-\gamma-\nu_{3}-it\Big)  \prod_{j=1}^{2} \, \zeta\Big(\frac{1}{2}+\gamma+\nu_{j}+it\Big) \nonumber\\
	&\hspace{20pt} \  \cdot \  \frac{\Gamma_{\R}(1/2-\gamma-\nu_{3}-it)}{\Gamma_{\R}(1/2+\gamma+\nu_{3}+it)} \ \left(\mathcal{F}_{\nu}H\right)\Big(\frac{1}{2}+\gamma+it,\   \frac{1+\delta}{2}\Big) \, dt.\nonumber
\end{align}
Following the setting of Section \ref{CFKUsym}, we  introduce  the linear re-parametrizations: 
\begin{align*}
	\begin{cases}
		\alpha_{1} \ =&  \gamma \ + \ \nu_{1}\\ 
		\alpha_{2} \ =&   \gamma \ + \ \nu_{2}\\
		\beta_{1} \  =&  -\gamma \ - \  \nu_{3} \\
		\beta_{2} \ =&  \delta \ - \ \gamma,
	\end{cases}
\end{align*}
where $\sum \nu_{j}=0$.  We solve for $\gamma$ using the first three equations: 	$\gamma =   (\alpha_{1}  +  \alpha_{2}  -  \beta_{1} )/3$. It follows that
\begin{align}\label{linchanVar}
	\begin{cases}
		\nu_{1} \ =&     (2\alpha_{1}  -  \alpha_{2}  +  \beta_{1})/3\\
		\nu_{2} \ =&    (2 \alpha_{2} -\alpha_{1}  +  \beta_{1})/3\\
		\nu_{3} \ =& -(\alpha_{1}+\alpha_{2}+2\beta_{1})/3 \\
		3\delta \  =&  3\beta_{2}  +   	\alpha_{1}  +  \alpha_{2}  -  \beta_{1}.  
	\end{cases}
\end{align}
As a result, the moment (\ref{duomo}) can be expressed as
\begin{align}\label{dualmo4th}
	\int_{\R} \,  \zeta\Big( \frac{1}{2}+ \alpha_{1}+it\Big)  &\zeta\Big( \frac{1}{2}+ \alpha_{2}+it\Big) \zeta\Big( \frac{1}{2}+ \beta_{1}-it\Big) \zeta\Big( \frac{1}{2}+ \beta_{2}-it\Big) 	\eta_{\alpha_{1}, \alpha_{2}; \beta_{1}, \beta_{2}}(t)  \, dt,
\end{align}
where the weight function is given by
\begin{align*}
	\eta_{\alpha_{1}, \alpha_{2}; \beta_{1}, \beta_{2}}(t) \ := \ 	\frac{1}{2\pi} \  \frac{\Gamma_{\R}(1/2+\beta_{1}-it)}{\Gamma_{\R}(1/2-\beta_{1}+it)} \ \left(\mathcal{F}_{\nu}H\right)\Big(\frac{1}{2}+\frac{\alpha_{1}  +  \alpha_{2}  -  \beta_{1} }{3}+it,\   \frac{1}{2}+ \frac{\alpha_{1}+ \alpha_{2}-\beta_{1}+3\beta_{2}}{6}\Big). 
\end{align*}

We write  $ \mathcal{R}_{\pm \alpha}(s; H) = \mathcal{R}_{\pm\nu}(s; H)$ (see (\ref{4thmoMT2})). Using the re-labeling (\ref{linchanVar}) and repeated application of the functional equation of $\Lambda(s)$, it follows that $(1/2)\,\mathcal{R}_{\nu}(1-s; H)+ (1/2) \, \mathcal{R}_{-\nu}(s; H)$ is given by
\begin{align}
  & H\Big(\frac{1+\alpha_{1}  -  \alpha_{2}  +  \beta_{1} -\beta_{2}}{2}\Big)  \frac{\Lambda(1+\alpha_{1}-\alpha_{2})\Lambda(1+\beta_{1}-\beta_{2})\Lambda(1+\alpha_{1}+\beta_{1}) \Lambda(1-\alpha_{2}-\beta_{2})}{\Lambda(2+\alpha_{1}-\alpha_{2}+\beta_{1}-\beta_{2})}  \nonumber\\
	 \ + \  &H\Big(\frac{1-\alpha_{1}+\alpha_{2}+\beta_{1}-\beta_{2}}{2}\Big)  \frac{\Lambda(1-\alpha_{1}+\alpha_{2}) \Lambda\left(1+\beta_{1}-\beta_{2}\right)\Lambda(1+\alpha_{2}+\beta_{1}) \Lambda(1-\alpha_{1}-\beta_{2}) }{\Lambda(2-\alpha_{1}+\alpha_{2}+\beta_{1}-\beta_{2})} \nonumber\\
	\ + \  &H\Big(\frac{1-\alpha_{1}-\alpha_{2}-\beta_{1}-\beta_{2}}{2}\Big) \frac{\Lambda(1-\alpha_{1}-\beta_{1})\Lambda(1-\alpha_{2}-\beta_{2}) \Lambda(1-\alpha_{2}-\beta_{1})\Lambda(1-\alpha_{1}-\beta_{2})}{\Lambda(2-\alpha_{1}-\alpha_{2}-\beta_{1}-\beta_{2})}, \nonumber\\
+ &H\Big(\frac{1-\alpha_{1}+\alpha_{2}-\beta_{1}+\beta_{2}}{2}\Big) \frac{\Lambda(1-\alpha_{1}-\beta_{1}) \Lambda(1-\alpha_{1}+\alpha_{2})\Lambda(1+\alpha_{2}+\beta_{2})\Lambda(1-\beta_{1}+\beta_{2})}{\Lambda(2-\alpha_{1}+\alpha_{2}-\beta_{1}+\beta_{2})}  \nonumber\\
	  \ + \  &H\Big(\frac{1+\alpha_{1}-\alpha_{2}-\beta_{1}+\beta_{2}}{2}\Big) \frac{\Lambda(1-\alpha_{2}-\beta_{1})\Lambda(1+\alpha_{1}-\alpha_{2})\Lambda(1+\alpha_{1}+\beta_{2}) \Lambda(1-\beta_{1}+\beta_{2})}{\Lambda(2+\alpha_{1}-\alpha_{2}-\beta_{1}+\beta_{2})} \nonumber\\
	  \ + \ &H\Big(\frac{1+\alpha_{1}+\alpha_{2}+\beta_{1}+\beta_{2}}{2}\Big) \frac{\Lambda(1+\alpha_{2}+\beta_{1})\Lambda(1+\alpha_{1}+\beta_{1})\Lambda(1+\alpha_{1}+\beta_{2})\Lambda(1+\alpha_{2}+\beta_{2})}{\Lambda(2+\alpha_{1}+\alpha_{2}+\beta_{1}+\beta_{2})}. \nonumber
\end{align}
We now readily observe that the quotients and products of $\zeta$'s above match exactly with those in the CFKRS Conjecture (\ref{expl4thMT})! Moreover, the cubic moment dual to (\ref{dualmo4th}) takes the form:
	\begin{align*}
		\sum_{j} \ \frac{H(\mu_{j})}{\langle \, \phi_{j}, \phi_{j}\, \rangle} \, \Lambda\big( \frac{1-\alpha_{1}  +  \alpha_{2}  -  \beta_{1} +\beta_{2}}{2}, \, \phi_{j}\big) \Lambda\big( \frac{1-\alpha_{2}  +  \alpha_{1}  -  \beta_{1} +\beta_{2}}{2} , \,\phi_{j}\big) \Lambda\big(\frac{1+\alpha_{1}+\alpha_{2}+\beta_{1}+\beta_{2}}{2}, \,\phi_{j}\big)
	\end{align*}
plus the corresponding continuous contribution.


\subsection{Comments}\label{sect: comment}

In the literature, moments of $L$-functions are more commonly approached with the approximate functional equations. A clear advantage is that one can take $s=1/2$ right from the start. For example, tracing the arguments of \cite{Pe15} for the cubic moment of $\hbox{GL}(2)$ $L$-functions (readily adaptable to the Maass case and closely related to \cite{Y17}),  the main terms for the fourth moment of the $\hbox{GL}(1)$ $L$-functions \emph{do not appear}: most likely, they are small and absorbed into the error term of the approximate Motohashi-type formula obtained in \cite{Pe15}. In other words,  the ``approximate'' treatment precludes the possibility of spectral inversion by changing the test vectors (see the end of Section \ref{feaDIS}). This is not satisfactory, as one should be able to pass between the two different-looking moments in a Motohashi-type formula for distinct applications, much as with the celebrated Kuznetsov formulae. 

In Section \ref{mainres}, we compare the sources of the main terms in the fourth moment with earlier methods.  For the cubic moment,  previous works of  \cite{Iv02} and \cite{Fr20} rely on the shifted divisor sums to extract the main terms. The work \cite{Fr20} (in the weight aspect) is closer in spirit to \cite{CFK+05}, which begins with an \textit{exact} identity for the twisted second moment of $\hbox{GL}(2)$ $L$-functions, obtained quite non-trivially from the Petersson formula. The dual side of the moment identity comprises three pieces of shifted divisor sums. In total, $12$ different residues arise from these sums.  As in \cite{HY10, Y11}, delicate combinations of terms are required, and  \cite{Fr20} carries this out. Interestingly, some combinations contribute to main terms, some to error terms, and some to a mixture of both! This appears to be the primary reason \cite{Fr20} arrives at an \emph{approximate} cubic moment identity, despite starting with an exact moment identity. The approaches of \cite{Mo93, BM05, Ne20+, Wu22, BFW21+} and the present work do not encounter this issue. This phenomenon deserves further investigation.


 \section{Acknowledgment}
 The author is grateful to the referee(s) for their thoughtful and valuable comments on the manuscript, which have led to substantial improvements in the exposition of the paper. Part of this work was completed during the author's visits to the Chinese University of Hong Kong and Queen's University, whose generous hospitality is warmly acknowledged.


\printbibliography

\ \\
\end{document}